\documentclass[reqno,11 pt]{amsart}

\usepackage{amscd}
\usepackage{a4wide}
\usepackage[english]{babel}
\usepackage[utf8]{inputenc}
\usepackage{amsmath,amsthm,amsfonts,amssymb}
\usepackage{braket}
\usepackage{hyperref}
\usepackage{tikz}

\theoremstyle{definition}
\newtheorem{defi}{Definition}[section]
\newtheorem{lem}[defi]{Lemma}
\newtheorem{prop}[defi]{Proposition}
\newtheorem{thm}[defi]{Theorem}
\newtheorem{cor}[defi]{Corollary}
\newtheorem{ass}[defi]{Assumption}
\newtheorem{rem}[defi]{Remark}

\newcommand{\ran}{\operatorname{ran}}
\newcommand{\dom}{\operatorname{dom}}
\newcommand{\Arg}{\operatorname{Arg}}
\renewcommand{\Re}{\operatorname{Re}}
\renewcommand{\Im}{\operatorname{Im}}

\title[$H^\infty$-functional calculi for the quaternionic fine structures]{The $H^\infty$-functional calculi for the quaternionic fine structures of Dirac type}

\author[F. Colombo]{Fabrizio Colombo}
\address{(FC) Politecnico di Milano, Dipartimento di Matematica, Via E. Bonardi 9, 20133 Milano, Italy}
\email{fabrizio.colombo@polimi.it}

\author[S. Pinton]{Stefano Pinton}
\address{(SP) Politecnico di Milano, Dipartimento di Matematica, Via E. Bonardi 9, 20133 Milano, Italy}
\email{stefano.pinton@polimi.it}

\author[P. Schlosser]{Peter Schlosser}
\address{(PS) Politecnico di Milano, Dipartimento di Matematica, Via E. Bonardi 9, 20133 Milano, Italy}
\email{pschlosser@math.tugraz.at}

\begin{document}

\begin{abstract}
In recent works, various integral representations have been proposed for specific sets of functions. These representations are derived from the Fueter-Sce extension theorem,  considering all possible factorizations of the Laplace operator in relation to both the Cauchy-Fueter operator (often referred to as the Dirac operator) and its conjugate. The collection of these function spaces, along with their corresponding functional calculi, are called the quaternionic fine structures within the context of the $S$-spectrum. In this paper, we utilize these integral representations of functions  to introduce novel functional calculi tailored for quaternionic operators of sectorial type. Specifically, by leveraging the aforementioned factorization of the Laplace operator, we identify four distinct classes of functions: slice hyperholomorphic functions (leading to the $S$-functional calculus), axially harmonic functions (leading to the $Q$-functional calculus), axially polyanalytic functions of order $2$ (leading to the $P_2$-functional calculus), and axially monogenic functions (leading to the $F$-functional calculus). By applying the respective product rule, we establish the four different $H^\infty$-versions of these functional calculi.
\end{abstract}

\maketitle

 \medskip

\textbf{Acknowledgements:} The research of Peter Schlosser was funded by the Austrian Science Fund (FWF) under Grant No. J 4685-N and by the European Union--NextGenerationEU.

\tableofcontents

\section{Introduction}

The spectral theory for quaternionic operators was originally motivated by the foundations of quaternionic quantum mechanics by Birkhoff and von Neumann \cite{BF}. The discovery of the appropriate notion of quaternionic spectrum called, the $S$-spectrum, in 2006, opened the way to an intensive development of the quaternionic and Clifford spectral theory with applications that go beyond the original motivations of quantum mechanics. In fact this theory has applications in fractional diffusion problems via the generation of fractional powers of vector operators, for more details see the introduction of the book \cite{CGK}. \medskip

Furthermore, it has been recently demonstrated that both, the quaternionic as well as the Clifford setting, are specific instances within a broader framework where the spectral theory regarding the $S$-spectrum can be developed, as outlined in \cite{ADVCGKS, PAMSCKPS}, along with relevant references therein. By employing the notion of the $S$-spectrum, researchers have also successfully established the quaternionic version of the spectral theorem. We refer the reader to \cite{ACKS16} where the spectral theorem for unitary operators has been proven utilizing Herglotz's functions, and to \cite{ACK} where it is proved the quaternionic spectral theorem for normal operators. More recently, the spectral theorem grounded in the concept of the $S$-spectrum has been extended to Clifford operators in \cite{ColKim}. \medskip

The development of the spectral theory on the $S$-spectrum also has opened up several research directions in hypercomplex analysis and operator theory. Without claiming completeness, we mention the slice hyperholomorphic Schur analysis \cite{ACS2016}, the characteristic operator functions \cite{AlpayColSab2020}, the quaternionic perturbation theory and invariant subspaces \cite{CereColKaSab}, and new classes of fractional diffusion problems that are based on the $H^\infty$-version of the $S$-functional calculus \cite{ColomboDenizPinton2020,ColomboDenizPinton2021,CGdiffusion2018,FJBOOK,ColomboPelosoPinton2019}. Moreover, recently nuclear operators and Grothendieck-Lidskii formula for quaternionic operators has been studied in \cite{DEB_NU} and quaternionic triangular linear operators have been investigated in \cite{TRIANG}. Finally, we mention that the spectral theory on the $S$-spectrum is systematically organized in the books \cite{FJBOOK,CGK,ColomboSabadiniStruppa2011}. \medskip

In recent times a new branch of the spectral theory on the $S$-spectrum  has been developed, that is called fine structures on the $S$-spectrum. It consists of function spaces arising from the Fueter-Sce extension theorem \cite{Fueter,TaoQian1,Sce}, which in the Clifford algebra $\mathbb{R}_n$ connects the class of slice hyperholomorphic functions with the class of axially monogenic functions via the powers $\Delta^{\frac{n-1}{2}}$ of the Laplace operator in dimension $n+1$. Note, that for odd $n$ the operator $\Delta^{\frac{n-1}{2}}$ is a pointwise differential operator, see \cite{ColSabStrupSce,Sce}, while for even values of $n$ we are dealing with fractional powers of the Laplace operator, see \cite{TaoQian1}. Analogously, in the quaternions $\mathbb{H}$, the Fueter mapping theorem connects slice hyperholomorphic functions and axially monogenic functions via the four dimensional Laplace operator $\Delta$. Note, that although $\mathbb{H}$ is classically identified with the Clifford algebra $\mathbb{R}_2$, one has to choose $n=3$ in  $\Delta^{\frac{n-1}{2}}$ and the dimension of the Laplace operator is $4$. If we denote by $\mathcal{SH}(U)$ the set of slice hyperholomorphic functions on some axially symmetric domain $U$ and $\mathcal{AM}(U)$ the class of axially monogenic functions on $U$, the Fueter  mapping theorem claims that
\begin{equation*}
\Delta:\mathcal{SH}(U)\rightarrow\mathcal{AM}(U)\quad\text{is surjective}.
\end{equation*}
For more information see the translation of the work of M. Sce in \cite{ColSabStrupSce}. For a different description of the Fueter-Sce theorem, see \cite{DDG, DDG1}. \medskip

The quaternionic fine structures of Dirac type is now based on the two different ways we can factorize the Laplacian
\begin{equation}\label{Eq_Laplace_operator}
\Delta=\mathcal D\overline{\mathcal D}=\overline{\mathcal D} \mathcal D,
\end{equation}
using the Cauchy-Fueter operator (also called Dirac-operator) and its conjugate
\begin{subequations}
\begin{align}
\mathcal D:=&\frac{\partial}{\partial q_0}+e_1\frac{\partial}{\partial q_1}+e_2\frac{\partial}{\partial q_2}+e_3\frac{\partial}{\partial q_3}, \label{Eq_Cauchy_Fueter_operator} \\
\overline{\mathcal D}:=&\frac{\partial}{\partial q_0}-e_1\frac{\partial}{\partial q_1}-e_2\frac{\partial}{\partial q_2}-e_3\frac{\partial}{\partial q_3}. \label{Eq_Cauchy_Fueter_operator_conjugate}
\end{align}
\end{subequations}
Depending on whether $\mathcal D$ or $\overline{\mathcal D}$ is applied first on some function $f\in\mathcal{SH}(U)$, we get the following four function spaces:
\begin{subequations}\label{Eq_Functionspaces}
\begin{align}
\mathcal{SH}(U)&\text{ from Definition \ref{defi_Slice_hyperholomorphic_functions}}, && \textit{(slice hyperholomorphic functions)} \label{Eq_Slice_hyperholomorphic_functions} \\
\mathcal{AH}(U)&=\Set{\mathcal Df | f\in\mathcal{SH}(U)}, && \textit{(axially harmonic functions)} \label{Eq_Axially_harmonic_functions} \\
\mathcal{AP}_2(U)&=\Set{\overline{\mathcal D}f | f\in\mathcal{SH}(U)}, && \textit{(polyanalytic functions of order $2$)} \label{Eq_Polyanalytic_functions} \\
\mathcal{AM}(U)&=\Set{\Delta f | f\in\mathcal{SH}(U)}, && \textit{(axially monogenic functions).} \label{Eq_Axially_monogenic_functions}
\end{align}
\end{subequations}
Observe that just the spaces $\mathcal{AH}(U)$ and $\mathcal{AP}_2(U)$ depend on the factorization of the Laplace operator. This construction can also be visualized in the following diagram: \medskip

\begin{center}
\begin{tikzpicture}
\draw (0,0) node[anchor=east] {$\mathcal{SH}(U)$};
\draw[->] (0,0.1)--(1,0.5) node[anchor=west] {$\mathcal{AH}(U)$};
\draw[->] (0,-0.1)--(1,-0.5) node[anchor=west] {$\mathcal{AP}_2(U)$};
\draw[->] (2.5,0.5)--(3.6,0.1);
\draw[->] (2.6,-0.5)--(3.6,-0.1);
\draw (3.6,0) node[anchor=west] {$\mathcal{AM}(U).$};
\draw (0.5,0.3) node[anchor=south] {\scriptsize{$\mathcal D$}};
\draw (0.5,-0.3) node[anchor=north] {\scriptsize{$\overline{\mathcal D}$}};
\draw (3.1,0.3) node[anchor=south] {\scriptsize{$\overline{\mathcal D}$}};
\draw (3.1,-0.3) node[anchor=north] {\scriptsize{$\mathcal D$}};
\end{tikzpicture}
\end{center}

While slice hyperholomorphic, axially harmonic and axially monogenic functions appear in many fields of pure and applied mathematics, the polyanalytic functions are less known but still have several applications. They were first considered by G.V. Kolossov in connection with his research on elasticity and also have applications in signal analysis, particularly in the context of Gabor frames with Hermite functions, as shown by the results of Gröchenig and Lyubarskii. Polyanalytic functions provide explicit representation formulas for functions in the eigenspaces of the Euclidean Laplacian with a magnetic field, which are referred to as Landau levels. This likely has significant implications in quantum mechanics and related fields. For an overview of the applications of this class of functions see the paper \cite{FISH} and also \cite{Balk,Russ,Bergmanvasi} for further material on this theory. \medskip

The corresponding integral representation of the functions of the fine structures \eqref{Eq_Functionspaces} will now give rise to various functional calculi. The original idea comes from the complex Riesz-Dunford functional calculus \cite{RD}, which is based on the Cauchy integral formula and it allows to replace a complex variable $z$ of a suitable holomorphic function $f(z)$ with a bounded linear operator $A$ in order  to define $f(A)$. The generalization of the holomorphic functional calculus to sectorial operators leads to the $H^\infty$-functional calculus that, in the complex setting, was introduced in the paper \cite{McI1}, see also the books \cite{Haase, HYTONBOOK1, HYTONBOOK2}. Moreover, the boundedness  of the $H^\infty$ functional calculus depends on suitable quadratic estimates and this calculus has several applications to boundary value problems, see  \cite{MC10,MC97,MC06,MC98}. \medskip

We remark that the $H^\infty$-functional calculus exists also for the monogenic functional calculus (see \cite{JM}) and it was introduced by A. McIntosh and his collaborators, see the books \cite{JBOOK,TAOBOOK} for more details. \medskip

In the quaternionic setting the functional calculi for bounded operators and slice hyperholomorphic functions \eqref{Eq_Slice_hyperholomorphic_functions} is already done in \cite{CS11,ColSab2006}, for bounded operators and axially monogenic functions in \cite{CDS,CDS1,CG,CS,CSF,CSS} and for functions in the spaces \eqref{Eq_Polyanalytic_functions}, \eqref{Eq_Axially_monogenic_functions} more recently in \cite{CDPS1,Fivedim,Polyf1,Polyf2}. \medskip

Unbounded operators for a restricted class of functions were considered in \cite{CDP23,CSF,ColSab2006,G17}. In the literature there exists the $H^\infty$-functional calculus for slice hyperholomorphic functions \eqref{Eq_Slice_hyperholomorphic_functions} in \cite{ACQS2016,CGK} and for harmonic functions \eqref{Eq_Axially_harmonic_functions} in \cite{MPS23}. \medskip

This paper on the one hand we revisit these constructions for unbounded operators with commuting components and enlarges the class of admissible operators $T$ for the $S$- and the $Q$-functional calculus to operators of type $(\alpha,\beta,\omega)$. On the other hand we also treat the two not yet investigated cases introducing the $H^\infty$-functional calculus for polynomially growing functions which are polyanalytic of order $2$ \eqref{Eq_Polyanalytic_functions} and for axially monogenic functions \eqref{Eq_Axially_monogenic_functions}. \medskip

\textit{Plan of the paper}: In Section \ref{sec_Preliminary_results} we discuss the preliminary results on quaternionic function theory and several issues related to quaternionic closed operators. In particular, we define the class of operators with commuting components in Definition \ref{defi_Operators_with_commuting_components} and the operators of type $(\alpha,\beta,\omega)$ in Definition \ref{defi_Operators_of_type_omega}. Since the $H^\infty$-functional calculus is a two step procedure we first introduce in Section \ref{sec_Functional_calculi_for_decaying_functions} the functional calculus for functions which decay suitably at infinity and at the origin. We will do this directly via the integrals
\begin{subequations}\label{Eq_Functional_calculus_formal}
\begin{align}
f(T)&=\frac{1}{2\pi}\int_{\partial(U\cap\mathbb{C}_J)}S_L^{-1}(s,T)ds_Jf(s), && \textit{($S$-functional calculus)} \label{Eq_S_functional_calculus_formal} \\
\mathcal Df(T)&=\frac{-1}{\pi}\int_{\partial(U\cap\mathbb{C}_J)}Q_{c,s}^{-1}(T)ds_Jf(s), && \textit{($Q$-functional calculus)}\label{Eq_Q_functional_calculus_formal} \\
\overline{\mathcal D}f(T)&=\frac{1}{2\pi}\int_{\partial(U\cap\mathbb{C}_J)}P_2^L(s,T)ds_Jf(s), && \textit{($P_2$-functional calculus)} \\
\Delta f(T)&=\frac{1}{2\pi}\int_{\partial(U\cap\mathbb{C}_J)}F_L(s,T)ds_Jf(s). && \textit{($F$-functional calculus)} \label{Eq_F_functional_calculus_formal}
\end{align}
\end{subequations}
The kernel of the $S$-functional calculus \eqref{Eq_S_functional_calculus_formal} is motivated by the quaternionic Cauchy formula
\begin{equation}\label{Eq_Cauchy_formula}
f(q)=\frac{1}{2\pi}\int_{\partial(U\cap\mathbb{C}_J)}\underbrace{(s-\overline{q})(s^2-2sq_0+|q|^2)^{-1}}_{=:S_L^{-1}(s,q)}ds_Jf(s),
\end{equation}
while the other kernels in \eqref{Eq_Q_functional_calculus_formal} -- \eqref{Eq_F_functional_calculus_formal} are generated by applying \eqref{Eq_Cauchy_Fueter_operator}, \eqref{Eq_Cauchy_Fueter_operator_conjugate} and \eqref{Eq_Laplace_operator}, to the Cauchy kernel $S_L^{-1}(s,q)$, i.e.,
\begin{equation}\label{Eq_Cauchy_kernel_derivatives}
Q_{c,s}^{-1}(q)=-\frac{1}{2}\mathcal DS_L^{-1}(s,q),\qquad P_2^L(s,q)=\overline{\mathcal D}S_L^{-1}(s,q),\qquad F_L(s,q)=\Delta S_L^{-1}(s,q),
\end{equation}
and afterwards formally replacing the quaternion $q$ by the operator $T$. The explicit representations of the kernels are given in \eqref{Eq_S_resolvent}, \eqref{Eq_Q_operator}, \eqref{Eq_P_resolvent} and \eqref{Eq_F_resolvent}. For these functional calculi it is now important to derive the respective product rules
\begin{subequations}\label{Eq_Product_rules_formal}
\begin{align}
(fg)(T)&=f(T)g(T), \\
\mathcal D(fg)(T)&=\mathcal Df(T)g(T)+f(\overline{T})\mathcal Dg(T), \\
\overline{\mathcal D}(fg)(T)&=\overline{\mathcal D}f(T)g(T)+f(T)\overline{\mathcal D}g(T)+\mathcal Df(T)g(T)-\mathcal Df(T)g(\overline{T}), \\
\Delta(fg)(T)&=\Delta f(T)g(T)+f(T)\Delta g(T)-\mathcal Df(T)\mathcal Dg(T),
\end{align}
\end{subequations}
see also Theorem \ref{thm_Product_rule_decaying}, in order to define the $H^\infty$-functional calculus. For example in the $S$-functional calculus case we consider a polynomially growing function $f\in\mathcal{SH}(U)$ and choose a regularizer function $e\in\mathcal{SH}(U)$ which decays fast enough for $(ef)(T)$ and $e(T)$ to be well defined in the sense \eqref{Eq_S_functional_calculus_formal}. Then we define the $H^\infty$-functional calculus as
\begin{equation*}
f(T):=e(T)^{-1}(ef)(T),
\end{equation*}
and prove that it is independent of the regularizer $e$. The similar definitions for the $Q$-, the $P_2$-  and the $F$-functional calculus can be seen in Definition \ref{defi_Functional_calculus_growing}. This regularization procedure is much more involved  when the product rule of a given functional calculus contains two or more addends.

\section{Preliminary results on quaternionic function theory and operators}\label{sec_Preliminary_results}

This section on the one hand presents several widely recognized concepts related to slice hyperholomorphic functions and quaternionic operators. But it also extends and improves, compared to previous results in the known literature, for instance the notion of operators with commuting components in Definition \ref{defi_Operators_with_commuting_components}, the growth conditions of sectorial operators of type $\omega$ in Definition \ref{defi_Operators_of_type_omega} and the domain of the $Q$-resolvent operator \eqref{Eq_Q_operator}. When working with closed operators, the idea of operators commuting with each other becomes intricate due to considerations involving their domains. The definition of commuting components for quaternionic linear operators holds significant importance throughout the entire paper, as it profoundly influences the definitions of the spectrum and resolvent set for these operators. \medskip

The \textit{quaternionic numbers} are defined as
\begin{equation*}
\mathbb{H}:=\Set{s_0+s_1e_1+s_2e_2+s_3e_3 | s_0,s_1,s_2,s_3\in\mathbb{R}},
\end{equation*}
with the three imaginary units $e_1,e_2,e_3$ satisfying the relations
\begin{equation*}
e_1^2=e_2^2=e_3^2=-1\qquad\text{and}\qquad\begin{array}{l} e_1e_2=-e_2e_1=e_3, \\ e_2e_3=-e_3e_2=e_1, \\ e_3e_1=-e_1e_3=e_2. \end{array}
\end{equation*}
For every quaternion $s\in\mathbb{H}$, we set
\begin{align*}
\Re(s)&:=s_0, && (\textit{real part}) \\
\Im(s)&:=s_1e_1+s_2e_2+s_3e_3, && (\textit{imaginary part}) \\
\overline{s}&:=s_0-s_1e_1-s_2e_2-s_3e_3, && (\textit{conjugate}) \\
|s|&:=\sqrt{s_0^2+s_1^2+s_2^2+s_3^2}. && (\textit{modulus})
\end{align*}
The unit sphere of purely imaginary quaternions is defined as
\begin{equation*}
\mathbb{S}:=\Set{s\in\mathbb{H} | s_0=0\text{ and }|s|=1},
\end{equation*}
and for every $J\in\mathbb{S}$ we consider the complex plane
\begin{equation*}
\mathbb{C}_J:=\Set{x+Jy | x,y\in\mathbb{R}},
\end{equation*}
which is an isomorphic copy of the complex numbers, since every $J\in\mathbb{S}$ satisfies $J^2=-1$. Moreover, for every quaternion $s\in\mathbb{H}$ we consider the corresponding $2$\textit{-sphere}
\begin{equation*}
[s]:=\Set{\Re(s)+J\vert\Im(s)\vert | J\in\mathbb{S}}.
\end{equation*}
Next, we introduce the notion of \textit{slice hyperholomorphic functions}, which is a quaternionic analog to the complex holomorphic functions. The sets upon which these functions are defined are the following \textit{axially symmetric sets}.

\begin{defi}
A subset $U\subseteq\mathbb{H}$ is called \textit{axially symmetric}, if $[s]\subseteq U$ for every $s\in U$.
\end{defi}

\begin{defi}\label{defi_Slice_hyperholomorphic_functions}
Let $U\subseteq\mathbb{H}$ be an axially symmetric open set and consider
\begin{equation}\label{Eq_Axially_symmetric_reduced_set}
\mathcal{U}:=\Set{(x,y)\in\mathbb{R}^2 | x+\mathbb{S}y\subseteq U}.
\end{equation}
A function $f:U\rightarrow\mathbb{H}$ is called \textit{left} (resp. \textit{right}) \textit{slice hyperholomorphic}, if there exists continuously differentiable functions $\alpha,\beta:\mathcal{U}\rightarrow\mathbb{H}$, such that for every $(x,y)\in\mathcal{U}$:

\begin{enumerate}
\item[i)] The function $f$ admits for every $J\in\mathbb{S}$ the representation
\begin{equation}\label{Eq_Holomorphic_decomposition}
f(x+Jy)=\alpha(x,y)+J\beta(x,y),\quad\Big(\text{resp.}\;f(x+Jy)=\alpha(x,y)+\beta(x,y)J\Big).
\end{equation}

\item[ii)] The functions $\alpha,\beta$ satisfy the even-odd conditions
\begin{equation}\label{Eq_Symmetry_condition}
\alpha(x,-y)=\alpha(x,y)\quad\text{and}\quad\beta(x,-y)=-\beta(x,y).
\end{equation}

\item[iii)] The functions $\alpha,\beta$ satisfy the Cauchy-Riemann equations
\begin{equation}\label{Eq_Cauchy_Riemann_equations}
\frac{\partial}{\partial x}\alpha(x,y)=\frac{\partial}{\partial y}\beta(x,y)\quad\text{and}\quad\frac{\partial}{\partial y}\alpha(x,y)=-\frac{\partial}{\partial x}\beta(x,y).
\end{equation}
\end{enumerate}

The class of left (resp. right) slice hyperholomorphic functions on $U$ is denoted by $\mathcal{SH}_L(U)$ (resp. $\mathcal{SH}_R(U)$). In the special case that $\alpha$ and $\beta$ are real valued, we call the function $f$ \textit{intrinsic} and denote the space of intrinsic functions by $\mathcal{N}(U)$.
\end{defi}

For those slice hyperholomorphic functions we now introduce quaternionic path integrals. Since it is sufficient to consider paths embedded in only one complex plane $\mathbb{C}_J$, the idea is to reduce it to a classical complex path integral.

\begin{defi}\label{defi_Path_integral}
Let $U\subseteq\mathbb{H}$ be open, axially symmetric and $f\in\mathcal{SH}_R(U)$, $g\in\mathcal{SH}_L(U)$. For $J\in\mathbb{S}$ and a continuously differentiable curve $\gamma:(a,b)\rightarrow U\cap\mathbb{C}_J$, we define the integral
\begin{equation*}
\int_\gamma f(s)ds_Jg(s):=\int_a^bf(\gamma(t))\frac{\gamma'(t)}{J}g(\gamma(t))dt.
\end{equation*}
In the case that $a,b$ are $\infty$ or lie on the boundary $\partial U$, the functions $f,g$ need to satisfy certain decay properties in order for the integral to exist.
\end{defi}

Next we turn our attention to quaternionic operator theory. From now on $V$ always denotes a two-sided linear Banach space over the quaternions $\mathbb{H}$. The set of bounded, everywhere defined operators will be denoted by $\mathcal{B}(V)$, and the set of closed operators with $\mathcal{K}(V)$. In the following we will specify the class of operators with commuting components, which will be of interest in this paper. We start with a lemma providing the existence of two-sided linear components of right linear operators. The proof of this lemma is left as an exercise, the same statement for bounded operators can for example be found in \cite{G18}.

\begin{lem}\label{lem_Components_of_operators}
Let $T:V\rightarrow V$ be some right linear operator with $\dom(T)$ being a two-sided linear subspace of $V$. Then there exist unique two-sided linear operators $T_i:V\rightarrow V$ with $\dom(T_i)=\dom(T)$, $i\in\{0,1,2,3\}$, such that
\begin{equation}\label{Eq_T_components}
T=T_0+e_1T_1+e_2T_2+e_3T_3.
\end{equation}
The components are explicitly given by
\begin{equation}\label{Eq_Components_representation}
T_i=\frac{e_i}{4}\Big(e_iTe_i-\sum\limits_{j=0,j\neq i}^3e_jTe_j\Big),\qquad i\in\{0,1,2,3\}.
\end{equation}
\end{lem}

Using the components $T_0,T_1,T_2,T_3$ from Lemma \ref{lem_Components_of_operators}, we can now define some more operators.

\begin{defi}
Let $T:V\rightarrow V$ be right linear with a two-sided linear domain. With the components $T_0,T_1,T_2,T_3$ from Lemma \ref{lem_Components_of_operators}, we define the \textit{conjugate operator}
\begin{equation}\label{Eq_Tbar}
\overline{T}:=T_0-e_1T_1-e_2T_2-e_3T_3,\qquad\text{with }\dom(\overline{T}):=\dom(T),
\end{equation}
and the \textit{modulus operator}
\begin{equation}\label{Eq_TT}
|T|^2:=T_0^2+T_1^2+T_2^2+T_3^2,\qquad\text{with }\dom(|T|^2):=\bigcap\limits_{k=0}^3\dom(T_k^2).
\end{equation}
\end{defi}

\begin{lem}\label{lem_Norm_of_components}
Let $B\in\mathcal{B}(V)$. Then also $B_i,\overline{B}\in\mathcal{B}(V)$ for every $i\in\{0,1,2,3\}$ and they admit the norm estimates
\begin{equation}\label{Eq_Norm_of_components}
\Vert B_i\Vert\leq\Vert B\Vert\qquad\text{and}\qquad\Vert\overline{B}\Vert\leq 2\Vert B\Vert.
\end{equation}
\end{lem}

\begin{proof}
From the explicit representation \eqref{Eq_Components_representation} of the components, we get
\begin{equation*}
\Vert B_i\Vert\leq\frac{1}{4}\Big(\Vert e_iBe_i\Vert+\sum\limits_{j=0,j\neq i}^3\Vert e_jBe_j\Vert\Big)=\frac{1}{4}\Big(\Vert B\Vert+\sum\limits_{j=0,j\neq i}^3\Vert B\Vert\Big)=\Vert B\Vert.
\end{equation*}
Also from the components \eqref{Eq_Components_representation} one easily deduces the representation
\begin{equation*}
\overline{B}=B_0-e_1B_1-e_2B_2-e_3B_3=-\frac{1}{2}(B+e_1Be_1+e_2Be_2+e_3Be_3),
\end{equation*}
from which the norm estimate $\Vert\overline{B}\Vert\leq 2\Vert B\Vert$ follows immediately.
\end{proof}

The following lemma gives a characterization of the commutation property of the components of bounded operators.

\begin{lem}\label{lem_Commutation_of_components}
Let $B\in\mathcal{B}(V)$ and $T:V\rightarrow V$ right linear with a two-sided linear domain. Then the following statements are equivalent \medskip

\begin{enumerate}
\item[1)] $T_iB_j=B_jT_i$,\quad on $\dom(T)$,\qquad $i,j\in\{0,1,2,3\}$, \\
\item[2)] $T_iB=BT_i$,\hspace{0.53cm} on $\dom(T)$,\hspace{1.01cm} $i\in\{0,1,2,3\}$, \\
\item[3)] $TB_j=B_jT$,\hspace{0.4cm} on $\dom(T)$,\hspace{0.93cm} $j\in\{0,1,2,3\}$.
\end{enumerate}
\end{lem}

\begin{proof}
The implication $\text{\grqq}1)\Rightarrow 2)\text{\grqq}$ and $\text{\grqq}1)\Rightarrow 3)\text{\grqq}$ are trivial. For $\text{\grqq}2)\Rightarrow 1)\text{\grqq}$ let us fix $i\in\{0,1,2,3\}$. From to the explicit representation of the components \eqref{Eq_Components_representation}, we get
\begin{equation*}
T_iB_j=T_i\frac{e_j}{4}\Big(e_jBe_j-\sum\limits_{k=0,k\neq j}^3e_kBe_k\Big)=\frac{e_j}{4}\Big(e_jBe_j-\sum\limits_{k=0,k\neq j}^3e_kBe_k\Big)T_i=B_jT_i.
\end{equation*}
The implication $\text{\grqq}3)\Rightarrow 1)\text{\grqq}$ similarly holds by
\begin{equation*}
T_iB_j=\frac{e_i}{4}\Big(e_iTe_i-\sum\limits_{k=0,k\neq i}^3e_kTe_k\Big)B_j=B_j\frac{e_i}{4}\Big(e_iTe_i-\sum\limits_{k=0,k\neq i}^3e_kTe_k\Big)=B_jT_i. \qedhere
\end{equation*}
\end{proof}

\begin{defi}\label{defi_Operators_with_commuting_components}
A right-linear $T:V\rightarrow V$ with a two-sided linear domain is called \textit{operator with commuting components}, if the components $T_0,T_1,T_2,T_3$ from Lemma \ref{lem_Components_of_operators} commute as
\begin{equation}\label{Eq_Commuting_components}
T_iT_jv=T_jT_iv,\qquad\text{for every }v\in\dom(|T|^2),\qquad i,j\in\{0,1,2,3\},
\end{equation}
where it is clear the $\dom(T_iT_j) \subseteq\dom(|T|^2)$ since $\dom(T_j)=\dom(T)$ for $j\in \{0,1,2,3\}$. We will denote the class of \textit{closed operators with commuting components} as $\mathcal{KC}(V)$ and the class of bounded operators with commuting components as $\mathcal{BC}(V)$.
\end{defi}

\begin{rem}
It is obvious that for every operator $T$ with commuting components, also its conjugate $\overline{T}$ is an operator with commuting components. However, from $T\in\mathcal{KC}(V)$ one cannot conclude $\overline{T}\in\mathcal{KC}(V)$, because in general the conjugate $\overline{T}$ may fail to be closed.
\end{rem}

For the multiplication of an operator with commuting components with its conjugate we obtain the following result.

\begin{lem}\label{lem_TT}
Let $T$ be an operator with commuting components. Then
\begin{equation}\label{Eq_Domain_TT}
\dom(|T|^2)\subseteq\dom(T\overline{T})\cap\dom(\overline{T}T),
\end{equation}
and there holds
\begin{equation}\label{Eq_Action_TT}
|T|^2v=\overline{T}Tv=T\overline{T}v,\qquad\text{for every }v\in\dom(|T|^2).
\end{equation}
\end{lem}

\begin{proof}
Let $v\in\dom(|T|^2)$, i.e. $v\in\dom(T)$ and $T_iv\in\dom(T)$ for every $i\in\{0,1,2,3\}$, see \eqref{Eq_TT}. Since $\dom(T)$ is two-sided linear and $\dom(\overline{T})=\dom(T)$, we also have
\begin{equation*}
Tv=\sum\limits_{i=0}^3e_iT_iv\in\dom(\overline{T})\qquad\text{and}\qquad\overline{T}v=\sum\limits_{i=0}^3\overline{e_i}\,T_iv\in\dom(T).
\end{equation*}
This proves that $v\in\dom(\overline{T}T)$ as well as $v\in\dom(T\overline{T})$. Moreover, there holds
\begin{equation*}
\overline{T}Tv=\sum\limits_{i,j=0}^3\overline{e_i}\,T_ie_jT_jv=\sum\limits_{i,j=0}^3\overline{e_i}e_jT_iT_jv=\sum\limits_{i=0}^3|e_i|^2T_i^2v+\sum\limits_{i\neq j=0}^3\overline{e_i}e_jT_iT_jv=\sum\limits_{i=0}^3T_i^2v=|T|^2v,
\end{equation*}
where in the second equation we used the fact that the operators $T_i$ are two-sided linear and the sum of the mixed terms vanishes in the second last equation due to $T_iT_jv=T_jT_iv$ and $\overline{e_k}e_l=-\overline{e_l}e_k$. Analogously there also holds $T\overline{T}v=|T|^2v$.
\end{proof}

\begin{ass}\label{ass_D}
Let $T,\overline{T}\in\mathcal{KC}(V)$. Then we consider some two-sided linear and dense subspace $D\subseteq\dom(|T|^2)$ with the following properties: \medskip

\begin{enumerate}
\item[i)] For every $i\in\{0,1,2,3\}$, there holds
\begin{equation}\label{Eq_Ti_range}
T_iv\in D,\qquad\text{for every }v\in D_1:=\Set{v\in D | \vert T\vert^2v\in\dom(T)}.
\end{equation}
\item[ii)] For every $s\in\mathbb{H}$, the operator
\begin{equation}\label{Eq_Q_operator}
Q_{c,s}(T):=s^2-2sT_0+|T|^2,\qquad\text{with }\dom(Q_{c,s}(T)):=D,
\end{equation}
is closed.
\end{enumerate}
\end{ass}

Let us now prove some basic properties of the operator $Q_{c,s}(T)$.

\begin{lem}\label{lem_Properties_Qs}
Let $T,\overline{T}\in\mathcal{KC}(V)$ and $D$ as in Assumption \ref{ass_D}. Then there holds: \medskip

\begin{enumerate}
\item[i)] For every $s\in\mathbb{H}$ we have $Q_{c,s}(T)\in\mathcal{KC}(V)$ with $\overline{Q_{c,s}(T)}=Q_{c,\overline{s}}(T)$. \\
\item[ii)] For every $i\in\{0,1,2,3\}$ we have the commutation relations
\begin{equation}\label{Eq_Commutation_Ti_Qs}
T_iQ_{c,s}(T)=Q_{c,s}(T)T_i,\qquad\text{on }D_1.
\end{equation}
\item[iii)] For every $s,p\in\mathbb{H}$ the domain of the product of two $Q$-operators is given by
\begin{equation}\label{Eq_Domain_QpQs}
\dom(Q_{c,s}(T)Q_{c,p}(T))=\Set{v\in D | \vert T\vert^2v\in D}=:D_2.
\end{equation}
\item[iv)] For every $s,p\in\mathbb{H}$ with $sp=ps$, also
\begin{equation}\label{Eq_Commutation_QsQp}
Q_{c,s}(T)Q_{c,p}(T)=Q_{c,p}(T)Q_{c,s}(T),\qquad\text{on }D_2.
\end{equation}
\item[v)] For every $s\in\mathbb{H}$ the operator $Q_{c,s}(T)Q_{c,\overline{s}}(T)$ is two-sided linear and writes as
\begin{equation}\label{Eq_QsQsbar}
Q_{c,s}(T)Q_{c,\overline{s}}(T)=(|s|^2-|T|^2)^2+4(T_0-s_0)(|s|^2T_0-s_0|T|^2),\qquad\text{on }D_2.
\end{equation}
\end{enumerate}
\end{lem}

\begin{rem}
Since the subspace $D$ from Assumption \ref{ass_D} is contained in $\dom(|T|^2)$, it is in particular contained in $\dom(T)$ and it follows that $D_2\subseteq D_1$ for the sets \eqref{Eq_Ti_range} and \eqref{Eq_Domain_QpQs}.
\end{rem}

\begin{proof}[Proof of Lemma \ref{lem_Properties_Qs}]
We start the proof by verifying the commutation relation
\begin{equation}\label{Eq_Commutation_Ti_TTbar}
|T|^2T_iv=T_i|T|^2v,\qquad v\in D_1.
\end{equation}
In order to show this, let $v\in D_1$. Then $v\in D$ and by the assumption \eqref{Eq_Ti_range} also $T_kv\in D$ for every $k\in\{0,1,2,3\}$. Using now \eqref{Eq_Commuting_components} for the vector $v$ and a second time for $T_kv$, we conclude the stated commutation \eqref{Eq_Commutation_Ti_TTbar}, namely
\begin{equation*}
|T|^2T_iv=\sum\limits_{k=0}^3T_kT_kT_iv=\sum\limits_{k=0}^3T_kT_iT_kv=\sum\limits_{k=0}^3T_iT_kT_kv=T_i|T|^2v.
\end{equation*}
i)\;\;First of all, writing $s=\sum_{i=0}^3e_is_i$, where $e_0=1$, we decompose the operator \eqref{Eq_Q_operator} into its components
\begin{equation}\label{Eq_Q_components}
Q_{c,s}(T)=\underbrace
{2s_0^2-|s|^2-2s_0T_0+|T|^2}_{=:A_0}+\sum\limits_{i=1}^3\underbrace{2s_i(s_0-T_0)}_{=:A_i}e_i,
\end{equation}
with $\dom(A_i)=D$, $i\in\{0,1,2,3\}$. To prove the commutation relation \eqref{Eq_Commuting_components}, we only consider the case $\Im(s)\neq 0$. For $\Im(s)=0$ it is $A_1=A_2=A_3=0$ and \eqref{Eq_Commuting_components} is trivially satisfied. Since $\Im(s)\neq 0$ means that $s_k\neq 0$ for at least one $k\in\{1,2,3\}$, the domain on which we have to check \eqref{Eq_Commuting_components} is given by
\begin{align}
\bigcap\nolimits_{k=0}^3\dom(A_k^2)&=\bigcap\nolimits_{k=0}^3\Set{v\in D | A_kv\in D} \notag \\
&=\Set{v\in D | (2s_0^2-|s|^2-2sT_0+|T|^2)v\in D,\,2s_k(s_0-T_0)v\in D,\,k\in\{1,2,3\}} \notag \\
&=\Set{v\in D | (-2sT_0+|T|^2)v,\,T_0v\in D} \notag \\
&=\Set{v\in D | \vert T\vert^2v\in D,\,T_0v\in D}=\Set{v\in D | \vert T\vert^2v\in D}=D_2, \label{Eq_Properties_Qs_1}
\end{align}
where in the second last equation we used the assumption \eqref{Eq_Ti_range}. Then, for every $v\in D_2$ and $i,j\in\{1,2,3\}$ we use \eqref{Eq_Commutation_Ti_TTbar} to show that there holds
\begin{align}
A_iA_0v&=2s_i(s_0-T_0)(2s_0^2-|s|^2-2s_0T_0+|T|^2)v \notag \\
&=(2s_0^2-|s|^2-2s_0T_0+|T|^2)2s_i(s_0-T_0)v=A_0A_iv, \label{Eq_Properties_Qs_2} \\
A_iA_jv&=2s_i(s_0-T_0)2s_j(s_0-T_0)v=2s_j(s_0-T_0)2s_i(s_0-T_0)v=A_jA_iv. \notag
\end{align}
Since $Q_{c,s}(T)$ is closed by definition \eqref{Eq_Q_operator}, we have proven that $Q_{c,s}(T)\in\mathcal{KC}(V)$. The equality $\overline{Q_{c,s}(T)}=Q_{c,\overline{s}}(T)$ follows immediately from \eqref{Eq_Q_components}. \medskip

ii)\;\;Let $v\in D_1$. Then it follows from \eqref{Eq_Commuting_components} and \eqref{Eq_Commutation_Ti_TTbar} that
\begin{equation*}
T_iQ_{c,s}(T)v=T_i(s^2-2sT_0+|T|^2)v=(s^2-2sT_0+|T|^2)T_iv=Q_{c,s}(T)T_iv.
\end{equation*}
iii)\;\;First, let $v\in\dom(Q_{c,s}(T)Q_{c,p}(T))$, i.e. $v\in D$ and $(p^2-2pT_0+|T|^2)v\in D$. Consequently $$(-2pT_0+|T|^2)v\in D\subseteq\dom(T)$$ and since $T_0v\in\dom(T)$ by $v\in D\subseteq\dom(T_0^2)$, we have $|T|^2v\in\dom(T)$. It then follows from \eqref{Eq_Ti_range}, that $T_0v\in D$ and hence also $|T|^2v\in D$. For the inverse inclusion consider $v\in D_2$, i.e. $v\in D$ with $|T|^2v\in D$. By \eqref{Eq_Ti_range} then also $T_0v\in D$ and
\begin{equation*}
Q_{c,p}(T)v=(p^2-2pT_0+|T|^2)v\in D,
\end{equation*}
which proves $v\in\dom(Q_{c,s}(T)Q_{c,p}(T))$. \medskip

iv)\;\;Let $v\in D_2$. Then $v\in D$, $|T|^2v\in D$ and due to the assumption \eqref{Eq_Ti_range} also $T_0v\in D$. Hence we are allowed to expand the product of the following two brackets and rearrange the terms
\begin{align}
Q_{c,s}(T)Q_{c,p}(T)v&=(s^2-2sT_0+|T|^2)(p^2-2pT_0+|T|^2)v \notag \\
&=(sp-|T|^2)^2v+(2T_0-s-p)\big(2spT_0-(s+p)|T|^2\big)v, \label{Eq_Properties_Qs_3}
\end{align}
where we also used the commutation \eqref{Eq_Commutation_Ti_TTbar}. Since we assumed that $sp=ps$ commute, the right hand side of \eqref{Eq_Properties_Qs_3} stays the same when we replace $s\leftrightarrow p$. This proves the stated commutation of $Q_{c,s}(T)$ and $Q_{c,p}(T)$. \medskip

v)\;\;Plugging in $p=\overline{s}$ in \eqref{Eq_Properties_Qs_3}, gives
\begin{equation}\label{Eq_Properties_Qs_4}
Q_{c,s}(T)Q_{c,\overline{s}}(T)v=(|s|^2-|T|^2)^2v+4(T_0-s_0)(|s|^2T_0-s_0|T|^2)v,\qquad v\in D_2,
\end{equation}
which is obviously two-sided linear.
\end{proof}

\begin{defi}[$F$-spectrum]\label{defi_F_spectrum}
Let $T,\overline{T}\in\mathcal{KC}(V)$ and $D$ as in Assumption \ref{ass_D}. According to the invertibility of the operator \eqref{Eq_Q_operator}, we define the \textit{$F$-resolvent set}
\begin{equation}\label{Eq_F_resolvent_set}
\rho_F(T):=\Set{s\in\mathbb{H} | Q_{c,s}(T)\text{ is bijective}},
\end{equation}
and the \textit{$F$-spectrum} as the complement
\begin{equation}\label{Eq_F_spectrum}
\sigma_F(T):=\mathbb{H}\setminus\rho_F(T).
\end{equation}
\end{defi}

\begin{rem}
When dealing with bounded operators $T$, the $S$-spectrum
\begin{equation*}
\sigma_S(T):=\Set{s\in\mathbb{H} | Q_s(T)\text{ is not bijective}},\quad\text{with }Q_s(T):=T^2-2s_0T+|s|^2,
\end{equation*}
coincides with $F$-spectrum in \eqref{Eq_F_spectrum}. The $F$-spectrum can be viewed as a commutative counterpart of the $S$-spectrum. However, for unbounded operators, additional research is needed to establish the equivalence between the two quaternionic spectra, especially considering the various definitions of operators with commuting components.
\end{rem}

Next, we collect some basic properties of the $F$-spectrum and the inverse operator $Q_{c,s}^{-1}(T)$. The proof of next lemma is highly inspired by \cite[Theorem 3.1.2]{FJBOOK}, where similar results are proven for different operators $T$ with commuting components in Definition \ref{defi_Operators_with_commuting_components} and with a different operator domain $\dom(Q_{c,s}(T))$.

\begin{lem}\label{lem_Properties_rhoF}
Let $T,\overline{T}\in\mathcal{KC}(V)$ and $D$ as in Assumption \ref{ass_D}. Then the $F$-resolvent set \eqref{Eq_F_resolvent_set} and the inverse of the operator \eqref{Eq_Q_operator} have the following properties: \medskip

\begin{enumerate}
\item[i)] $\rho_F(T)$ is an open subset of $\mathbb{H}$. \\
\item[ii)] $\rho_F(T)$ is axially symmetric. \\
\item[iii)] For every $s\in\rho_F(T)$ and $i\in\{0,1,2,3\}$  we have the commutation relations
\begin{equation}\label{Eq_Commutation_Ti_Qsinv}
T_iQ_{c,s}^{-1}(T)=Q_{c,s}^{-1}(T)T_i,\qquad\text{on }\dom(T).
\end{equation}
\end{enumerate}
\end{lem}

\begin{proof}
i)\;\;Fix $s\in\rho_F(T)$. For every $p\in\mathbb{H}$ we consider the operator
\begin{equation*}
\Lambda_p(T):=\big(Q_{c,s}(T)-Q_{c,p}(T)\big)Q_{c,s}^{-1}(T)=(s^2-p^2)Q_{c,s}^{-1}(T)-2(s-p)T_0Q_{c,s}^{-1}(T).
\end{equation*}
Since $T,\overline{T}$ are closed, $TQ_{c,s}^{-1}(T),\overline{T}Q_{c,s}^{-1}(T)$ are closed and everywhere defined and hence bounded. Consequently also their sum
\begin{equation}\label{Eq_Properties_rhoF_11}
T_0Q_{c,s}^{-1}(T)=\frac{1}{2}(TQ_{c,s}^{-1}(T)+\overline{T}Q_{c,s}^{-1}(T))
\end{equation}
is a bounded operator. Choose now $\varepsilon>0$ small enough, such that
\begin{align*}
\Vert\Lambda_p(T)\Vert&\leq|s^2-p^2|\,\Vert Q_{c,s}^{-1}(T)\Vert+2|s-p|\,\Vert T_0Q_{c,s}^{-1}(T)\Vert \\
&\leq\varepsilon(2|s|+\varepsilon)\Vert Q_{c,s}^{-1}(T)\Vert+2\varepsilon\Vert T_0Q_{c,s}^{-1}(T)\Vert<1,\qquad p\in U_\varepsilon(s),
\end{align*}
where in the second inequality we used
\begin{equation*}
|s^2-p^2|=|s^2-sp+sp-p^2|\leq|s||s-p|+|s-p||p|\leq\varepsilon(2|s|+\varepsilon).
\end{equation*}
Using Neumann series, we then conclude that the operator $1-\Lambda_p(T)$ is boundedly invertible. The inverse operator then satisfies
\begin{align}
Q_{c,p}(T)\big(Q_{c,s}^{-1}(T)(1-\Lambda_p(T))^{-1}\big)&=\big(1-(Q_{c,s}(T)-Q_{c,p}(T))Q_{c,s}^{-1}(T)\big)(1-\Lambda_p(T))^{-1} \notag \\
&=(1-\Lambda_p(T))(1-\Lambda_p(T))^{-1}=1,\qquad\text{on }V, \label{Eq_Properties_rhoF_1}
\end{align}
but also
\begin{align}
\big(Q_{c,s}^{-1}(T)(1-\Lambda_p(T))^{-1}\big)Q_{c,p}(T)&=Q_{c,s}^{-1}(T)(1-\Lambda_p(T))^{-1}(1-\Lambda_p(T))Q_{c,s}(T) \notag \\
&=Q_{c,s}^{-1}(T)Q_{c,s}(T)=1,\qquad\text{on }D. \label{Eq_Properties_rhoF_2}
\end{align}
The two identities \eqref{Eq_Properties_rhoF_1} and \eqref{Eq_Properties_rhoF_2} now prove that $Q_{c,p}(T)$ is bijective for every $p\in U_\varepsilon(s)$, and this proves the first claim. \medskip

We now anticipate the proof of point iii) because it is necessary to prove point ii). \medskip

iii)\;\;Let $v\in\dom(T)$ and define $w:=Q_{c,s}^{-1}(T)v$. Then $(s^2-2sT_0+|T|^2)w=v\in\dom(T)$. Since moreover $w\in D\subseteq\dom(T_0)^2$, also $T_0w\in\dom(T)$, which gives $|T|^2w\in\dom(T)$, i.e. $w\in D_1$. From the assumption \eqref{Eq_Commutation_Ti_Qs} we then conclude the commutation
\begin{equation*}
T_iQ_{c,s}(T)w=Q_{c,s}(T)T_iw.
\end{equation*}
Applying $Q_{c,s}^{-1}(T)$ from the left and plugging in the element $w=Q_{c,s}^{-1}(T)v$, gives
\begin{equation*}
Q_{c,s}^{-1}(T)T_iv=T_iQ_{c,s}^{-1}(T)v.
\end{equation*}
ii)\;\;Let $s\in\rho_F(T)$. In the \textit{first step} we will prove that $\overline{s}\in\rho_F(T)$. Therefore, let us consider the components $A_0,A_1,A_2,A_3$ of $Q_{c,s}(T)$ from \eqref{Eq_Q_components}. Then from the commutation relation \eqref{Eq_Commutation_Ti_Qsinv} there follows
\begin{equation}\label{Eq_Properties_rhoF_4}
Q_{c,s}^{-1}(T)A_i=A_iQ_{c,s}^{-1}(T),\qquad\text{on }D.
\end{equation}
If we also decompose $Q_{c,s}^{-1}(T)=\sum_{j=0}^3e_jB_j$ with respect to its components $B_j\in\mathcal{B}(V)$, it follows from Lemma \ref{lem_Commutation_of_components}, that there also commute the components
\begin{equation}\label{Eq_Properties_rhoF_5}
B_jA_i=A_iB_j,\qquad\text{on }D.
\end{equation}
From the explicit form of the quaternionic product we then get for every $v\in D$
\begin{align*}
v&=Q_{c,s}^{-1}(T)Q_{c,s}(T)v \\
&=(B_0A_0-B_1A_1-B_2A_2-B_3A_3)v+e_1(B_2A_3-B_3A_2+B_0A_1+B_1A_0)v \\
&\quad+e_2(B_3A_1-B_1A_3+B_0A_2+B_2A_0)v+e_3(B_1A_2-B_2A_1+B_0A_3+B_3A_0)v,
\end{align*}
as well as
\begin{align*}
v&=Q_{c,s}(T)Q_{c,s}^{-1}(T)v \\
&=(B_0A_0-B_1A_1-B_2A_2-B_3A_3)v+e_1(B_3A_2-B_2A_3+B_1A_0+B_0A_1)v \\
&\quad+e_2(B_1A_3-B_3A_1+B_2A_0+B_0A_2)v+e_3(B_2A_1-B_1A_2+B_3A_0+B_0A_3)v,
\end{align*}
for every $v\in V$. For $v\in D$ we now add and subtract these two equations and obtain:
\begin{subequations}\label{Eq_Properties_rhoF_6}
\begin{align}
v&=(B_0A_0-B_1A_1-B_2A_2-B_3A_3)v \notag \\
&\quad+e_1(B_0A_1+B_1A_0)v+e_2(B_0A_2+B_2A_0)v+e_3(B_0A_3+B_3A_0)v,\quad\text{and} \\
0&=e_1(B_2A_3-B_3A_2)v+e_2(B_3A_1-B_1A_3)v+e_3(B_1A_2-B_2A_1)v,
\end{align}
\end{subequations}
which both have to be satisfied for every component individually due to the uniqueness of the components in Lemma \ref{lem_Components_of_operators}. If we now multiply $\overline{Q_{c,s}^{-1}(T)}=B_0-e_1B_1-e_2B_2-e_3B_3$ and $Q_{c,\overline{s}}(T)=A_0-e_1A_1-e_2A_2-e_3A_3$ and use the identities \eqref{Eq_Properties_rhoF_6}, we get
\begin{align*}
\overline{Q_{c,s}^{-1}(T)}Q_{c,\overline{s}}(T)v&=(B_0A_0-B_1A_1-B_2A_2-B_3A_3)v+e_1(B_2A_3-B_3A_2-B_0A_1-B_1A_0)v \\
&+e_2(B_3A_1-B_1A_3-B_0A_2-B_2A_0)v+e_3(B_1A_2-B_2A_1-B_0A_3-B_3A_0)v \\
&=v+e_1(0-0)v+e_2(0-0)v+e_3(0-0)v=v.
\end{align*}
Similarly, we also get
\begin{align*}
Q_{c,\overline{s}}(T)\overline{Q_{c,s}^{-1}(T)}v&=(B_0A_0-B_1A_1-B_2A_2-B_3A_3)v+e_1(B_3A_2-B_2A_3-B_1A_0-B_0A_1)v \\
&+e_2(B_1A_3-B_3A_1-B_2A_0-B_0A_2)v+e_3(B_2A_1-B_1A_2-B_3A_0-B_0A_3)v \\
&=v+e_1(0-0)v+e_2(0-0)v+e_3(0-0)v=v.
\end{align*}
Since $Q_{c,s}^{-1}(T)\in\mathcal{B}(V)$, also $\overline{Q_{c,s}^{-1}(T)}\in\mathcal{B}(V)$ by Lemma \ref{lem_Norm_of_components} and the product $Q_{c,s}(T)\overline{Q_{c,s}^{-1}(T)}$ is closed as the product of a bounded and a closed operator. Since it is also everywhere defined due to $\ran(B_j)\subseteq D$ by \eqref{Eq_Components_representation} and the two-sided linearity of $D$, we get
\begin{equation*}
Q_{c,s}(T)\overline{Q_{c,s}^{-1}(T)}\in\mathcal{B}(V).
\end{equation*}
This means, we can extend the identity $Q_{c,\overline{s}}(T)\overline{Q_{c,s}^{-1}(T)}v=v$ from the dense subspace $D$ to the whole space $V$ by continuity. Hence, $Q_{c,\overline{s}}(T)$ is bijective with $Q_{c,\overline{s}}^{-1}(T)=\overline{Q_{c,s}^{-1}(T)}$. \medskip

In the \textit{second step} we show that $p\in\rho_F(T)$ for every $p\in[s]$. To do so, we write $s=x+Jy$ for $x,y\in\mathbb{R}$, $J\in\mathbb{S}$ and decompose $Q_{c,s}(T)$ into
\begin{equation}\label{Eq_Properties_rhoF_10}
Q_{c,s}(T)=\underbrace{x^2-y^2-2xT_0+|T|^2}_{=:A_0}+J\underbrace{2y(x-T_0)}_{=:A_J}.
\end{equation}
Since we have already proven in the first step that $Q_{c,\overline{s}}(T)$ is bijective, we write its inverse as
\begin{equation}\label{Eq_Properties_rhoF_9}
Q_{c,s}^{-1}(T)=Q_{c,\overline{s}}(T)Q_{c,\overline{s}}^{-1}(T)Q_{c,s}^{-1}(T)=(A_0-JA_J)\big(Q_{c,s}(T)Q_{c,\overline{s}}(T)\big)^{-1},
\end{equation}
using the two-sided linear operator \eqref{Eq_Properties_Qs_4}, given by
\begin{equation}\label{Eq_Properties_rhoF_8}
Q_{c,s}(T)Q_{c,\overline{s}}(T)=(x^2+y^2-|T|^2)^2+4(T_0-x)\big((x^2+y^2)T_0-x|T|^2\big),\quad\text{on }D_2.
\end{equation}
Since $p\in[s]$, we can decompose $p=x+Iy$ for some $I\in\mathbb{S}$. Consequently there holds $Q_{c,p}(T)=A_0+IA_J$, with $A_0$, $A_J$ from \eqref{Eq_Properties_rhoF_10}, and also
\begin{equation*}
Q_{c,p}(T)Q_{c,\overline{p}}(T)=Q_{c,s}(T)Q_{c,\overline{s}}(T),
\end{equation*}
since \eqref{Eq_Properties_rhoF_8} does not depend on the imaginary units $J$. Using this, we get
\begin{equation}\label{Eq_Properties_rhoF_3}
Q_{c,p}(T)\big(Q_{c,\overline{p}}(T)(Q_{c,s}(T)Q_{c,\overline{s}}(T))^{-1}\big)=Q_{c,p}(T)Q_{c,\overline{p}}(T)(Q_{c,p}(T)Q_{c,\overline{p}}(T))^{-1}=1,\quad\text{on }V.
\end{equation}
Due to \eqref{Eq_Commutation_Ti_Qsinv} every individual term in (\ref{Eq_Properties_rhoF_10}) commutes with $Q_{c,s}^{-1}(T)$ and $Q_{c,\overline{s}}^{-1}(T)$ on $D$.
Consequently also  $A_0$ and $A_J$ and hence $Q_{c,p}(T)$ commute with $Q_{c,s}^{-1}(T)$ and $Q_{c,\overline{s}}^{-1}(T)$ on $D$. We can now rearrange the left hand side of \eqref{Eq_Properties_rhoF_3} to
\begin{equation}\label{Eq_Properties_rhoF_7}
\big(Q_{c,\overline{p}}(T)(Q_{c,s}(T)Q_{c,\overline{s}}(T))^{-1}\big)Q_{c,p}(T)=Q_{c,\overline{p}}(T)Q_{c,p}(T)(Q_{c,\overline{p}}(T)Q_{c,p}(T))^{-1}=1,\quad\text{on }D.
\end{equation}
The equations \eqref{Eq_Properties_rhoF_3} and \eqref{Eq_Properties_rhoF_7} then show that $Q_{c,p}(T)$ is bijective and hence $p\in\rho_F(T)$.
\end{proof}

\begin{lem}\label{lem_Commutation_B}
Let $T,\overline{T}\in\mathcal{KC}(V)$ and $D$ as in Assumption \ref{ass_D}. Then every $B\in\mathcal{B}(V)$ which commutes with the components $BT_i=T_iB$ on $\dom(T)$, $i\in\{0,1,2,3\}$, also commutes with \medskip

\begin{enumerate}
\item[i)] $B|T|^2=|T|^2B$,\qquad on $\dom(|T|^2)$, \\
\item[ii)] $B\big(Q_{c,s}(T)Q_{c,\overline{s}}(T)\big)^{-1}=\big(Q_{c,s}(T)Q_{c,\overline{s}}(T)\big)^{-1}B$,\qquad for every $s\in\rho_F(T)$.
\end{enumerate}
\end{lem}

\begin{proof}
i)\;\;Let $v\in\dom(|T|^2)$. Then $v\in\dom(T)$ and $T_iv\in\dom(T)$ for every $i\in\{0,1,2,3\}$. This means we are allowed to use the assumption $BT_i=T_iB$ for $T_iv$ and for $v$, which gives
\begin{equation*}
B|T|^2v=\sum\limits_{i=0}^3BT_iT_iv=\sum\limits_{i=0}^3T_iBT_iv=\sum\limits_{k=0}^3T_iT_iBv=|T|^2Bv.
\end{equation*}
ii)\;\;Let $v\in V$ and define $w:=\big(Q_{c,s}(T)Q_{c,\overline{s}}(T)\big)^{-1}v$. Then $w\in D_2$ by \eqref{Eq_Domain_QpQs}, i.e. $w\in D$ with $|T|^2w\in D$. This in particular implies $w\in\dom(|T|^2)\subseteq\dom(T)$ and it follows from the assumption and from i) the commutation relations
\begin{equation}\label{Eq_Commutation_B_1}
T_0Bw=BT_0w\qquad\text{and}\qquad|T|^2Bw=B|T|^2w.
\end{equation}
Moreover, we also have $(s_0|T|^2-|s|^2T_0)w\in\dom(T)$ as well as $(|s|^2-|T|^2)w\in D$. Again by assumption and by i), there also commutes
\begin{align}
BT_0(s_0|T|^2-|s|^2T_0)w&=T_0B(s_0|T|^2-|s|^2T_0)w \qquad\text{and} \label{Eq_Commutation_B_2} \\
B|T|^2(|s|^2-|T|^2)w&=|T|^2B(|s|^2-|T|^2)w. \label{Eq_Commutation_B_3}
\end{align}
Combining now \eqref{Eq_Commutation_B_1}, \eqref{Eq_Commutation_B_2} and \eqref{Eq_Commutation_B_3} in the representation \eqref{Eq_QsQsbar}, we obtain
\begin{equation*}
BQ_{c,s}(T)Q_{c,\overline{s}}(T)w=Q_{c,s}(T)Q_{c,\overline{s}}(T)Bw.
\end{equation*}
Applying $\big(Q_{c,s}(T)Q_{c,\overline{s}}(T)\big)^{-1}$ from the left and plugging in $w=\big(Q_{c,\overline{s}}(T)Q_{c,s}(T)\big)^{-1}v$, gives the stated commutation ii).
\end{proof}

With the operator \eqref{Eq_Q_operator} and motivated by the Cauchy integral formula \eqref{Eq_Cauchy_formula}, we define for every $T,\overline{T}\in\mathcal{KC}(V)$, $D$ as in Assumption \ref{ass_D} and $s\in\rho_F(T)$ the \textit{left} and the \textit{right} \textit{$S$-resolvent}
\begin{equation}\label{Eq_S_resolvent}
S_L^{-1}(s,T):=(s-\overline{T})Q_{c,s}^{-1}(T)\qquad\text{and}\qquad S_R^{-1}(s,T):=sQ_{c,s}^{-1}(T)-\sum\limits_{i=0}^3T_iQ_{c,s}^{-1}(T)\overline{e_i}.
\end{equation}
Note, that on $\dom(T)$ we are allowed to interchange $T_i$ and $Q_{c,s}^{-1}(T)$ due to \eqref{Eq_Commutation_Ti_Qsinv}, which gives the more elegant form of the right $S$-resolvent
\begin{equation}\label{Eq_S_resolvent_on_domT}
S_R^{-1}(s,T)=Q_{c,s}^{-1}(T)(s-\overline{T}),\qquad\text{on }\dom(T).
\end{equation}
It is straight forward to verify, that the Cauchy-Fueter operator $\mathcal D$ from \eqref{Eq_Cauchy_Fueter_operator}, its conjugate $\overline{\mathcal D}$ from \eqref{Eq_Cauchy_Fueter_operator_conjugate} and the Laplace operator $\Delta$ from \eqref{Eq_Laplace_operator}, applied to the Cauchy-kernel $S_L^{-1}(s,q)$ from \eqref{Eq_Cauchy_formula}, are given by
\begin{align*}
\mathcal DS_L^{-1}(s,q)&=-2Q_{c,s}^{-1}(q), \\
\overline{\mathcal D}S_L^{-1}(s,q)&=2S_L^{-1}(s,q)\big(S_L^{-1}(s,q)+S_L^{-1}(s,\overline{q})\big) \\
\Delta S_L^{-1}(s,q)&=-4S_L^{-1}(s,q)Q_{c,s}^{-1}(q).
\end{align*}
We see, that the kernel $-2Q_{c,s}^{-1}(T)$ of the $Q$-functional calculus is already given by the inverse of the $Q$-operator \eqref{Eq_Q_operator}. Moreover, the above relations motivate the $P_2$-resolvents
\begin{subequations}\label{Eq_P_resolvent}
\begin{align}
P_2^L(s,T):=&2S_L^{-1}(s,T)\big(S_L^{-1}(s,T)+S_L^{-1}(s,\overline{T})\big), \label{Eq_P2L_resolvent}
\\
P_2^R(s,T):=&2\big(S_R^{-1}(s,T)+S_R^{-1}(s,\overline{T})\big)S_R^{-1}(s,T),
\end{align}
\end{subequations}
as well as the $F$-resolvents
\begin{equation}\label{Eq_F_resolvent}
F_L(s,T):=-4S_L^{-1}(s,T)Q_{c,s}^{-1}(T)\qquad\text{and}\qquad F_R(s,T):=-4Q_{c,s}^{-1}(T)S_R^{-1}(s,T).
\end{equation}
Next, we will find a similar structure for all four integral kernels \eqref{Eq_S_resolvent}, \eqref{Eq_Q_operator}, \eqref{Eq_P_resolvent} and \eqref{Eq_F_resolvent}. In Section \ref{sec_Functional_calculi_for_decaying_functions}, this structure will play a crucial role in many important properties of the respective functional calculi.

\begin{lem}\label{lem_K_decomposition}
Let $T,\overline{T}\in\mathcal{KC}(V)$ and $D$ as in Assumption \ref{ass_D}. Then for every $s\in\rho_F(T)$ we consider the following pairs of operators: \medskip

\begin{enumerate}
\item[i)] $K_L(s,T):=S_L^{-1}(s,T)$\hspace{0.43cm} and \quad $K_R(s,T):=S_R^{-1}(s,T)$,\hspace{0.45cm} or \\
\item[ii)] $K_L(s,T):=-2Q_{c,s}^{-1}(T)$\quad and \quad $K_R(s,T):=-2Q_{c,s}^{-1}(T)$,\quad or \\
\item[iii)] $K_L(s,T):=P_2^L(s,T)$\hspace{0.6cm} and \quad $K_R(s,T):=P_2^R(s,T)$,\hspace{0.53cm} or \\
\item[iv)] $K_L(s,T):=F_L(s,T)$\hspace{0.65cm} and \quad $K_R(s,T):=F_R(s,T)$.
\end{enumerate}

\medskip

Then for any pair i) -- iv), these operators admit for $s=x+Jy\in\rho_F(T)$ the decomposition
\begin{equation}\label{Eq_K_decomposition}
K_L(s,T)=A(x,y,T)+B(x,y,T)J\quad\text{and}\quad K_R(s,T)=A(x,y,T)+JB(x,y,T),
\end{equation}
with operators $A(x,y,T),B(x,y,T)\in\mathcal{B}(V)$, satisfying
\begin{subequations}\label{Eq_AB_properties}
\begin{align}
\circ\;\;&A(x,-y,T)=A(x,y,T), \qquad B(x,-y,T)=-B(x,y,T), \label{Eq_AB_sbar_property} \\
\circ\;\;&\overline{A(x,y,T)}=A(x,y,\overline{T}),\hspace{1.08cm} \overline{B(x,y,T)}=B(x,y,\overline{T}), \label{Eq_AB_Tbar_property} \\
\circ\;\;&A(x,y,T)\text{ and }B(x,y,T)\text{ commute with }T,\overline{T},T_0,T_1,T_2,T_3\text{ on}\,\dom(T). \label{Eq_AB_Ti_commutation}
\end{align}
\end{subequations}
Let $C\in\mathcal{B}(V)$ and suppose that it commutes with $T,T_0,T_1,T_2,T_3$ on $\dom(T)$, then
\begin{equation}\label{Eq_AB_C_commutation}
C\text{ commutes with }A(x,y,T)\text{ and }B(x,y,T).
\end{equation}
\end{lem}

\begin{proof}
For simplicity, we will write $A=A(x,y,T)$ and $B=B(x,y,T)$ in this proof. \medskip

We now anticipate the proof of point ii) because it is necessary to prove point i). \medskip

ii)\;\;For the $Q$-resolvent, we obtain from \eqref{Eq_Properties_rhoF_10} and \eqref{Eq_Properties_rhoF_9} the decomposition
\begin{equation}\label{Eq_Q_decomposition}
Q_{c,s}^{-1}(T)=\underbrace{(x^2-y^2-2xT_0+|T|^2)\big(Q_{c,s}(T)Q_{c,\overline{s}}(T)\big)^{-1}}_{=:A_1}\underbrace{-2y(x-T_0)\big(Q_{c,s}(T)Q_{c,\overline{s}}(T)\big)^{-1}}_{=:B_1}J.
\end{equation}
We have $T_0Q_{c,s}^{-1}(T)\in\mathcal{B}(V)$ by \eqref{Eq_Properties_rhoF_11}, and so $B_1\in\mathcal{B}(V)$. Because $Q_{c,s}^{-1}(T)\in\mathcal{B}(V)$, also $A_1\in\mathcal{B}(V)$ has to be bounded. Since
\begin{equation*}
Q_{c,s}(T)Q_{c,\overline{s}}(T)=(x^2+y^2-|T|^2)^2+4(T_0-x)\big((x^2+y^2)T_0-x|T|^2\big),\qquad\text{on }D_2,
\end{equation*}
by \eqref{Eq_QsQsbar}, the operators $A_1,B_1$ are two-sided linear, also the representation $Q_{c,s}^{-1}(T)=A_1+JB_1$ of the right kernel follows. The properties \eqref{Eq_AB_sbar_property} and \eqref{Eq_AB_Tbar_property} are clearly satisfied. The property \eqref{Eq_AB_Ti_commutation} follows from the commutation properties \eqref{Eq_Commuting_components}, \eqref{Eq_Commutation_Ti_TTbar}, \eqref{Eq_Commutation_Ti_Qsinv} and the fact that $A_1$ and $B_1$ are two-sided linear. Finally, \eqref{Eq_AB_C_commutation} follows from Lemma \ref{lem_Commutation_B}. \medskip

i)\;\;From the representation \eqref{Eq_Q_decomposition} of the $Q$-resolvent, we immediately get the representation
\begin{equation}\label{Eq_SL_decomposition}
S_L^{-1}(s,T)=(x+Jy-\overline{T})(A_1+B_1J)=\underbrace{(x-\overline{T})A_1-yB_1}_{=:A_2}+\underbrace{\big(yA_1+(x-\overline{T})B_1\big)}_{=:B_2}J,
\end{equation}
for the $S$-resolvent \eqref{Eq_S_resolvent}, where we used that $A_1$ and $B_1$ commute with $J$. Since $\overline{T}$ is closed, the products $\overline{T}A_1$ and $\overline{T}B_1$ are closed operators  and everywhere defined and hence they are bounded. This implies that $A_2,B_2\in\mathcal{B}(V)$. For the right $S$-resolvent \eqref{Eq_S_resolvent}, we similarly obtain the representation
\begin{align}
S_R^{-1}(s,T)&=(x+Jy)(A_1+B_1J)-\sum\limits_{i=0}^3T_i(A_1+B_1J)\overline{e_i} \notag \\
&=(x-\overline{T})A_1-yB_1+J\big(yA_1+(x-\overline{T})B_1\big)=A_2+JB_2, \label{Eq_SR_decomposition}
\end{align}
using the same operators $A_2$ and $B_2$ as in \eqref{Eq_SL_decomposition}. The properties \eqref{Eq_AB_properties} and \eqref{Eq_AB_C_commutation} of $A_2$ and $B_2$ follow from the respective properties of $A_1$ and $B_1$, where for \eqref{Eq_AB_Ti_commutation} we additionally need that $\overline{T}T_i=T_i\overline{T}$ and $\overline{T}T=T\overline{T}$ on $\ran(A_1)\cup\ran(B_1)\subseteq D$, see \eqref{Eq_Commuting_components}, \eqref{Eq_Action_TT} and \eqref{Eq_Domain_QpQs}. For \eqref{Eq_AB_C_commutation} we note that since $C$ commutes with $T$ and $T_0$, it also commutes with $\overline{T}=2T_0-T$ on $\dom(T)$. Hence \eqref{Eq_AB_C_commutation} for $A_2$ and $B_2$ follows from the same property of $A_1$ and $B_1$. \medskip

iii)\;\;We can now use \eqref{Eq_Q_decomposition}, \eqref{Eq_SL_decomposition} and \eqref{Eq_SR_decomposition}, to write the left and the right $F$-resolvent \eqref{Eq_F_resolvent} as
\begin{equation}\label{Eq_FL_decomposition}
F_L(s,T)=-4(A_2+B_2J)(A_1+B_1J)=\underbrace{-4(A_2A_1-B_2B_1)}_{=:A_3}\underbrace{-4(A_2B_1+B_2A_1)}_{=:B_3}J,
\end{equation}
and using the same operators $A_3$ and $B_3$, we also get
\begin{align}
F_R(s,T)&=-4(A_1+B_1J)(A_2+JB_2) \notag \\
&=-4(A_1A_2-B_1B_2)-4J(B_1A_2+A_1B_2)=A_3+JB_3, \label{Eq_FR_decomposition}
\end{align}
where in the last equation we used that $A_1B_2=B_2A_1$ and $A_2B_1=B_1A_2$ commute due to \eqref{Eq_AB_Ti_commutation} and \eqref{Eq_AB_C_commutation}. The properties \eqref{Eq_AB_properties} and \eqref{Eq_AB_C_commutation} of $A_3,B_3$ follow immediately from the respective properties of $A_1,B_1,A_2,B_2$. \medskip

iv)\;\;Using the $F$-resolvents \eqref{Eq_F_resolvent}, it is straight forward to rewrite the $P_2$-resolvents \eqref{Eq_P_resolvent} as
\begin{equation*}
P_2^L(s,T)=T_0F_L(s,T)-F_L(s,T)s\qquad\text{and}\qquad P_2^R(s,T)=(T_0-s)F_R(s,T).
\end{equation*}
Hence we can use \eqref{Eq_FL_decomposition} and \eqref{Eq_FR_decomposition}, to write
\begin{align}
P_2^L(s,T)&=T_0(A_3+B_3J)-(A_3+B_3J)(x+Jy) \notag \\
&=\underbrace{(T_0-x)A_3+yB_3}_{=:A_4}+\underbrace{\big((T_0-x)B_3-yA_3\big)}_{=:B_4}J. \label{Eq_PL_decomposition}
\end{align}
Since $TA_3$ and $\overline{T}A_3$ are everywhere defined and closed, they are also bounded. Consequently, also $T_0A_3=\frac{1}{2}(TA_3+\overline{T}A_3)\in\mathcal{B}(V)$ is bounded. For the same reason also $T_0B_3\in\mathcal{B}(V)$ and consequently $A_4,B_4\in\mathcal{B}(V)$. Using the same operators $A_4$ and $B_4$, we also get
\begin{align*}
P_2^R(s,T)&=(T_0-x-Jy)(A_3+JB_3) \\
&=(T_0-x)A_3+yB_3+J\big((T_0-x)B_3-yA_3\big)=A_4+JB_4.
\end{align*}
The properties \eqref{Eq_AB_properties} and \eqref{Eq_AB_C_commutation} of $A_4$ and $B_4$ now follow from the respective properties of $A_3$ and $B_3$, where for \eqref{Eq_AB_Ti_commutation} we additionally needed that $T_iT_0=T_0T_i$ commutes due to \eqref{Eq_Commuting_components}.
\end{proof}

To draw a parallel with intrinsic slice hyperholomorphic functions, we will in the following define slice holomorphicity of operator valued functions. In particular we will prove in Corollary \ref{cor_Holomorphic_kernels}, that our resolvent kernels \eqref{Eq_S_resolvent}, \eqref{Eq_Q_operator}, \eqref{Eq_P_resolvent} and \eqref{Eq_F_resolvent} are indeed slice hyperholomorphic functions in the variable $s$.

\begin{defi}
Let $U\subseteq\mathbb{H}$ be an axially symmetric open set. An operator valued function $K:U\rightarrow\mathcal{B}(V)$ is called \textit{left} (resp. \textit{right}) \textit{slice hyperholomorphic}, if there exists operator valued functions $A,B:\mathcal{U}\rightarrow\mathcal{B}(V)$, with $\mathcal{U}$ in \eqref{Eq_Axially_symmetric_reduced_set}, such that for every $(x,y)\in\mathcal{U}$:

\begin{enumerate}
\item[i)] The operators $K$ admit for every $J\in\mathbb{S}$ the representation
\begin{equation}\label{Eq_Holomorphic_decomposition_operators}
K(x+Jy)=A(x,y)+JB(x,y),\quad\Big(\text{resp.}\;K(x+Jy)=A(x,y)+B(x,y)J\Big).
\end{equation}

\item[ii)] The operators $A,B$ satisfy the even-odd conditions
\begin{equation}\label{Eq_Symmetry_condition_operators}
A(x,-y)=A(x,y)\quad\text{and}\quad B(x,-y)=-B(x,y).
\end{equation}

\item[iii)] The operators $A,B$ satisfy the Cauchy-Riemann equations
\begin{equation}\label{Eq_Cauchy_Riemann_equations_operators}
\frac{\partial}{\partial x}A(x,y)=\frac{\partial}{\partial y}B(x,y)\quad\text{and}\quad\frac{\partial}{\partial y}A(x,y)=-\frac{\partial}{\partial x}B(x,y),
\end{equation}
where the derivatives are understood in the norm convergence sense.
\end{enumerate}

We moreover, call $s\mapsto K(s)$ \textit{intrinsic}, if the operators $A,B$ are two-sided linear.
\end{defi}

\begin{cor}\label{cor_Holomorphic_kernels}
Let $T,\overline{T}\in\mathcal{KC}(V)$ and $D$ as in Assumption \ref{ass_D}. Then \medskip

\begin{enumerate}
\item[i)] $Q_{c,s}^{-1}(T)$ is intrinsic, \\
\item[ii)] $S_L^{-1}(s,T)$, $P_2^L(s,T)$, $F_L(s,T)$ are right-slice hyperholomorphic, \\
\item[iii)] $S_R^{-1}(s,T)$, $P_2^R(s,T)$, $F_R(s,T)$ are left-slice hyperholomorphic.
\end{enumerate}
\end{cor}

\begin{proof}
i)\;\;In order to show that $Q_{c,s}^{-1}(T)$ is intrinsic, we use the decomposition
\begin{equation*}
Q_{c,s}^{-1}(T)=A_1(x,y,T)+B_1(x,y,T)J,
\end{equation*}
from \eqref{Eq_Q_decomposition}. It is then obvious, that $A_1$, $B_1$ satisfy the symmetry relation \eqref{Eq_Symmetry_condition_operators}. One can also straight forward calculate the derivatives
\begin{align*}
\frac{\partial A_1}{\partial x}=\frac{\partial B_1}{\partial y}=&2(T_0-x)\big(Q_{c,s}(T)Q_{c,\overline{s}}(T)\big)^{-1} \\
&-8y^2(T_0-x)\big(x^2+y^2-2xT_0+2T_0^2-|T|^2\big)\big(Q_{c,s}(T)Q_{c,\overline{s}}(T)\big)^{-2},
\end{align*}
as well as
\begin{align*}
\frac{\partial A_1}{\partial y}=-\frac{\partial B_1}{\partial x}=&2y\big(Q_{c,s}(T)Q_{c,\overline{s}}(T)\big)^{-1} \\
&-8y(T_0-x)^2\big(x^2+y^2-2xT_0+|T|^2\big)\big(Q_{c,s}(T)Q_{c,\overline{s}}(T)\big)^{-2}.
\end{align*}
Hence the Cauchy-Riemann equations \eqref{Eq_Cauchy_Riemann_equations_operators} are satisfied for $A_1$ and $B_1$. Since the operators $A_1$, $B_1$ are also two-sided linear we have verified that the operator $Q_{c,s}^{-1}(T)$ is intrinsic. \medskip

ii) Since $S_L^{-1}(s,T)$ in \eqref{Eq_S_resolvent} is defined as
\begin{equation*}
S_L^{-1}(s,T)=(s-\overline{T})Q_{c,s}^{-1}(T),
\end{equation*}
it is an intrinsic function $Q_{c,s}^{-1}(T)$ multiplied from the left with $(s-\overline{T})$. Hence it is automatically right slice-hyperholomorphic. \medskip

If we write the $P_2$-resolvent \eqref{Eq_P2L_resolvent} in the form
\begin{equation*}
P_2^L(s,T)=4S_L^{-1}(s,T)(s-T_0)Q_{c,s}^{-1}(T),
\end{equation*}
it is the product of the right slice-hyperholomorphic function $S_L^{-1}(s,T)$ with the intrinsic function $(s-T_0)Q_{c,s}^{-1}(T)$, and hence again right slice-hyperholomorphic. \medskip

Also the $F$-resolvent \eqref{Eq_F_resolvent} is written in the form
\begin{equation*}
F_L(s,T)=-4S_L^{-1}(s,T)Q_{c,s}^{-1}(T),
\end{equation*}
as the product of the right slice-hyperholomorphic function $S_L^{-1}(s,T)$ with the intrinsic function $Q_{c,s}^{-1}(T)$. Hence also $F_L(s,T)$ turns out to be right slice-hyperholomorphic. \medskip

iii)\;\;The left slice-holomorphicity of the respective right kernels follows from the fact that they admit the common decomposition \eqref{Eq_K_decomposition} with the left kernels.
\end{proof}

In the last part of this section, we specify the class of operators, for which the functional calculus will be established in this paper. Therefore, we define for every angle $\omega\in(0,\pi)$ the open sector
\begin{equation}\label{Eq_Somega}
S_\omega:=\Set{s\in\mathbb{H}\setminus\{0\} | \vert\Arg(s)\vert<\omega},
\end{equation}
where $\Arg(s)\in[-\pi,\pi]$ is the usual argument of complex numbers when we treat $s$ as an element in the complex plane $\mathbb{C}_J$. Note that, since $\mathbb{C}_J=\mathbb{C}_{-J}$, the imaginary unit of the complex plane is not uniquely defined, and so also the argument of a quaternionic number is only unique up to a sign. However, this does not affect the sector $S_\omega$ in \eqref{Eq_Somega}.

\begin{defi}\label{defi_Operators_of_type_omega}
Let $\alpha,\beta\in\mathbb{R}$, $\omega\in(0,\pi)$. An operator $T$ and a set $D$ as in Assumption \ref{ass_D} is called \textit{of type} $(\alpha,\beta,\omega)$, if $T,\overline{T}\in\mathcal{KC}(V)$, the spectrum is contained in the sector
\begin{equation*}
\sigma_F(T)\subseteq\overline{S_\omega},
\end{equation*}
and for every $\varphi\in(\omega,\pi)$ there exists $C_\varphi\geq 0$, such that
\begin{equation}\label{Eq_S_resolvent_estimate}
\Vert S_L^{-1}(s,T)\Vert\leq C_\varphi\begin{cases} |s|^{-\alpha}, & |s|\leq 1, \\ |s|^{-\beta}, & |s|\geq 1, \end{cases}\qquad s\in S_\varphi^c\setminus\{0\}
\end{equation}
where $S_\varphi^c:=\mathbb{H}\setminus S_\varphi$ is the complement of the sector $S_\varphi$.
\end{defi}

We will now show that the estimates \eqref{Eq_S_resolvent_estimate} on the left $S$-resolvent imply similar estimates on the right $S$-resolvents as well as on the left and right $Q$-, $P_2$- and $F$-resolvents. These estimates will then be crucial in defining the convergence of the integrals in Definition \ref{defi_Functional_calculus_decaying}.

\begin{lem}\label{lem_Resolvent_estimates}
Let $\alpha,\beta\in\mathbb{R}$, $\omega\in(0,\pi)$ and $T$ of type $(\alpha,\beta,\omega)$ (see  Definition \ref{defi_Operators_of_type_omega}). Then also $\overline{T}$ is of type $(\alpha,\beta,\omega)$ and for every $\varphi\in(\omega,\pi)$ there exists $C_\varphi\geq 0$, such that for every $s\in S_\varphi^c\setminus\{0\}$, there holds
\begin{align*}
\Vert S_R^{-1}(s,T)\Vert &\leq C_\varphi\begin{cases} |s|^{-\alpha}, & |s|\leq 1, \\ |s|^{-\beta}, & |s|\geq 1, \end{cases} \hspace{0.8cm} \Vert P_2^L(s,T)\Vert,\Vert P_2^R(s,T)\Vert\leq C_\varphi\begin{cases} |s|^{-2\alpha}, & |s|\leq 1, \\ |s|^{-2\beta}, & |s|\geq 1, \end{cases} \\
\Vert Q_{c,s}^{-1}(T)\Vert&\leq C_\varphi\begin{cases} |s|^{-2\alpha}, & |s|\leq 1, \\ |s|^{-2\beta}, & |s|\geq 1, \end{cases} \qquad \Vert F_L(s,T)\Vert,\Vert F_R(s,T)\Vert\leq C_\varphi\begin{cases} |s|^{-3\alpha}, & |s|\leq 1, \\ |s|^{-3\beta}, & |s|\geq 1. \end{cases}
\end{align*}
\end{lem}

\begin{proof}
Let us start by using \eqref{Eq_K_decomposition} to decompose the $S$-resolvent as
\begin{equation*}
S_L^{-1}(s,T)=A(x,y,T)+B(x,y,T)J,\qquad s=x+Jy\in\rho_F(T).
\end{equation*}
Due to the property \eqref{Eq_AB_sbar_property}, we can write
\begin{equation*}
A(x,y,T)=\frac{1}{2}\big(S_L^{-1}(s,T)+S_L^{-1}(\overline{s},T)\big)\quad\text{and}\quad B(x,y,T)=\big(S_L^{-1}(s,T)-S_L^{-1}(\overline{s},T)\big)\frac{1}{2J},
\end{equation*}
which by \eqref{Eq_S_resolvent_estimate} leads to the norm estimates
\begin{equation*}
\Vert A(x,y,T)\Vert\leq C_\varphi\begin{cases} |s|^{-\alpha}, & |s|\leq 1, \\ |s|^{-\beta}, & |s|\geq 1, \end{cases}\quad\text{and}\quad\Vert B(x,y,T)\Vert\leq C_\varphi\begin{cases} |s|^{-\alpha}, & |s|\leq 1, \\ |s|^{-\beta}, & |s|\geq 1. \end{cases}
\end{equation*}
Since we can write $S_R^{-1}(s,T)=A(x,y,T)+JB(x,y,T)$ by \eqref{Eq_K_decomposition}, we also get the estimate
\begin{equation}\label{Eq_Resolvent_estimates_2}
\Vert S_R^{-1}(s,T)\Vert\leq 2C_\varphi\begin{cases} |s|^{-\alpha}, & |s|\leq 1, \\ |s|^{-\beta}, & |s|\geq 1. \end{cases}
\end{equation}
Moreover, from \eqref{Eq_AB_sbar_property} and \eqref{Eq_AB_Tbar_property} we obtain $S_L^{-1}(s,\overline{T})=\overline{S_R^{-1}(\overline{s},T)}$ and $S_R^{-1}(s,\overline{T})=\overline{S_L^{-1}(\overline{s},T)}$. Together with the norm estimate of the conjugate operator in Lemma \ref{lem_Norm_of_components}, it then follows from \eqref{Eq_Resolvent_estimates_2} and \eqref{Eq_S_resolvent_estimate}, that
\begin{equation}\label{Eq_Resolvent_estimates_3}
\Vert S_L^{-1}(s,\overline{T})\Vert\leq 4C_\varphi\begin{cases} |s|^{-\alpha}, & |s|\leq 1, \\ |s|^{-\beta}, & |s|\geq 1. \end{cases}\quad\text{and}\quad\Vert S_R^{-1}(s,\overline{T})\Vert\leq 2C_\varphi\begin{cases} |s|^{-\alpha}, & |s|\leq 1, \\ |s|^{-\beta}, & |s|\geq 1. \end{cases}
\end{equation}
Hence, \eqref{Eq_S_resolvent_estimate} is shown for the operator $\overline{T}$, and since $\sigma_F(\overline{T})=\sigma_F(T)\subseteq\overline{S_\omega}$ is trivial, it is indeed of type $(\alpha,\beta,\omega)$. Next, it is shown in \cite[Lemma 2.9]{MPS23}, that the $Q$-resolvent admits the representation
\begin{align*}
Q_{c,s}^{-1}(T)=&\frac{1}{4}\big(S_R^{-1}(s,T)+S_R^{-1}(s,\overline{T})\big)\big(S_L^{-1}(s,T)+S_L^{-1}(s,\overline{T})\big) \\
&-\frac{1}{4}\big(S_R^{-1}(s,T)-S_R^{-1}(s,\overline{T})\big)\big(S_L^{-1}(s,T)-S_L^{-1}(s,\overline{T})\big).
\end{align*}
Hence by \eqref{Eq_S_resolvent_estimate}, \eqref{Eq_Resolvent_estimates_2} and \eqref{Eq_Resolvent_estimates_3} it admits the norm estimate
\begin{equation}\label{Eq_Resolvent_estimates_1}
\Vert Q_{c,s}^{-1}(T)\Vert\leq 10C_\varphi^2\begin{cases} |s|^{-2\alpha}, & |s|\leq 1, \\ |s|^{-2\beta}, & |s|\geq 1. \end{cases}
\end{equation}
For the estimate of the $P_2$-resolvents \eqref{Eq_P_resolvent}, we also combine \eqref{Eq_S_resolvent_estimate}, \eqref{Eq_Resolvent_estimates_2} and \eqref{Eq_Resolvent_estimates_3}, and we get
\begin{equation*}
\Vert P_2^L(s,T)\Vert\leq 10C_\varphi^2\begin{cases} |s|^{-2\alpha}, & |s|\leq 1, \\ |s|^{-2\beta}, & |s|\geq 1, \end{cases}\quad\text{and}\quad\Vert P_2^R(s,T)\Vert\leq 16C_\varphi^2\begin{cases} |s|^{-2\alpha}, & |s|\leq 1, \\ |s|^{-2\beta}, & |s|\geq 1, \end{cases}
\end{equation*}
Finally, it follows from \eqref{Eq_S_resolvent_estimate}, \eqref{Eq_Resolvent_estimates_2} and \eqref{Eq_Resolvent_estimates_1} that the $F$-resolvents \eqref{Eq_F_resolvent} can be estimated by
\begin{equation*}
\Vert F^L(s,T)\Vert\leq 40C_\varphi^3\begin{cases} |s|^{-3\alpha}, & |s|\leq 1, \\ |s|^{-3\beta}, & |s|\geq 1, \end{cases}\quad\text{and}\quad\Vert F^R(s,T)\Vert\leq 80C_\varphi^3\begin{cases} |s|^{-3\alpha}, & |s|\leq 1, \\ |s|^{-3\beta}, & |s|\geq 1. \end{cases} \qedhere
\end{equation*}
\end{proof}

\section{The $\omega$-functional calculus for decaying functions}\label{sec_Functional_calculi_for_decaying_functions}

In this section we introduce the $S$-, the $Q$-, the $P_2$- and the $F$-functional calculus for operators of type $(\alpha,\beta,\omega)$, see Definition \ref{defi_Operators_of_type_omega}, by giving a direct meaning to the integrals \eqref{Eq_S_functional_calculus_formal} -- \eqref{Eq_F_functional_calculus_formal}. In order to make these integrals converge, we need to assume certain decay properties on the function $f$. In particular, we treat for every $\alpha\geq 1$, $\beta\leq 1$, $\theta\in(0,\pi)$ and $S_\theta$ the sector \eqref{Eq_Somega}, we consider the following classes of slice hyperholomorphic functions:

\begin{enumerate}
\item[i)] $\Psi_L^{\alpha,\beta}(S_\theta):=\Set{f\in\mathcal{SH}_L(S_\theta) | \exists\delta>0,\,C_f\geq 0: |f(s)|\leq C_f\begin{cases} |s|^{\alpha-1+\delta}, & |s|\leq 1, \\ |s|^{\beta-1-\delta}, & |s|\geq 1 \end{cases}},$
\item[ii)] $\Psi^{\alpha,\beta}(S_\theta):=\Set{f\in\mathcal{N}(S_\theta) | \exists\delta>0,\,C_f\geq 0: |f(s)|\leq C_f\begin{cases} |s|^{\alpha-1+\delta}, & |s|\leq 1, \\ |s|^{\beta-1-\delta}, & |s|\geq 1 \end{cases}}.$
\end{enumerate}

The next theorem is crucial for the welldefinedness of the functional calculi in Definition~\ref{defi_Functional_calculus_decaying}.

\begin{thm}\label{thm_Integral_independence}
Let $\omega\in(0,\pi)$ and consider a family of bounded linear operators $K(s)\in\mathcal{B}(V)$, for $s\in\mathbb{H}\setminus\overline{S_\omega}$ , such that $s\mapsto K(s)$ is right slice hyperholomorphic. Moreover, suppose that there exists $\alpha\geq 1$, $\beta\in(0,1]$, such that for every $\varphi\in(\omega,\pi)$ there exists some $C_\varphi$ with
\begin{equation}\label{Eq_K_estimate}
\Vert K(s)\Vert\leq C_\varphi\begin{cases} |s|^{-\alpha}, & |s|\leq 1, \\ |s|^{-\beta} & |s|\geq 1, \end{cases}\qquad s\in S_\varphi^c\setminus\{0\}.
\end{equation}
Then for any $\theta\in(\omega,\pi)$ and $f\in\Psi_L^{\alpha,\beta}(S_\theta)$, the integral \medskip

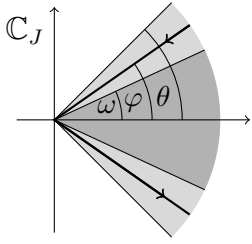
\begin{minipage}{0.25\textwidth}
\begin{center}
\begin{tikzpicture}
\fill[black!15] (1.56,1.56)--(0,0)--(1.56,-1.56) arc (-45:45:2.2);
\fill[black!30] (2,0.93)--(0,0)--(2,-0.93) arc (-25:25:2.2);
\draw (2,0.93)--(0,0)--(2,-0.93);
\draw (1.56,1.56)--(0,0)--(1.56,-1.56);
\draw (0.9,0) arc (0:25:0.9) (0.7,-0.05) node[anchor=south] {$\omega$};
\draw (1.3,0) arc (0:35:1.3) (1.05,-0.03) node[anchor=south] {$\varphi$};
\draw (1.7,0) arc (0:45:1.7) (1.45,0.05) node[anchor=south] {$\theta$};
\draw[thick] (1.8,-1.26)--(0,0)--(1.8,1.26);
\draw[thick,->] (1.8,1.26)--(1.47,1.03);
\draw[thick,->] (0,0)--(1.47,-1.03);
\draw[->] (-0.5,0)--(2.6,0);
\draw[->] (0,-1.5)--(0,1.5) node[anchor=north east] {\large{$\mathbb{C}_J$}};
\end{tikzpicture}
\end{center}
\end{minipage}
\begin{minipage}{0.74\textwidth}
\begin{equation}\label{Eq_Integral_independence}
\int_{\partial(S_\varphi\cap\mathbb{C}_J)}K(s)ds_Jf(s),
\end{equation}
is absolute convergent and it depends neither of the angle $\varphi\in(\omega,\theta)$ nor on the imaginary unit $J\in\mathbb{S}$.
\end{minipage}
\end{thm}

\begin{proof}
For the \textit{absolute convergence of the integral} \eqref{Eq_Integral_independence}, we use the integration path
\begin{equation}\label{Eq_gamma}
\gamma(t):=\begin{cases} -te^{J\varphi}, & t<0, \\ te^{-J\varphi}, & t>0, \end{cases}
\end{equation}
along the boundary of $S_\varphi\cap\mathbb{C}_J$. Then the estimate \eqref{Eq_K_estimate} of the operator $K(s)$ and the decay of the function $f\in\Psi_L^{\alpha,\beta}(S_\theta)$, gives the absolute convergence of the integral
\begin{equation}\label{Eq_Integral_independence_13}
\int_{\mathbb{R}\setminus\{0\}}\Vert K(s)\Vert\,|\gamma'(t)|\,|f(\gamma(t))|dt\leq 2C_\varphi C_f\bigg(\int_0^1t^{-1+\delta}dt+\int_1^\infty t^{-1-\delta}dt\bigg)=\frac{4C_\varphi C_f}{\delta}<\infty.
\end{equation}
For the \textit{independence of the angle} $\varphi$, let us consider two angles $\varphi_1<\varphi_2\in(\omega,\theta)$ and for every $0<\varepsilon<R$ the curves \medskip

\begin{minipage}{0.29\textwidth}
\begin{center}
\begin{tikzpicture}
\fill[black!15] (1.25,2.17)--(0,0)--(1.25,-2.17) arc (-60:60:2.5);
\draw (1.25,2.17)--(0,0)--(1.25,-2.17);
\fill[black!30] (2.27,1.06)--(0,0)--(2.27,-1.06) arc (-25:25:2.5);
\draw (2.27,1.06)--(0,0)--(2.27,-1.06);
\draw[->] (-0.3,0)--(2.9,0);
\draw[->] (0,-2)--(0,2.3) node[anchor=north east] {\large{$\mathbb{C}_J$}};
\draw (1,0) arc (0:35:1) (0.75,-0.09) node[anchor=south] {\small{$\varphi_1$}};
\draw (1.4,0) arc (0:50:1.4) (1.15,0.06) node[anchor=south] {\small{$\varphi_2$}};
\draw[thick] (0.45,0.54)--(1.45,1.73) arc (50:35:2.26)--(0.58,0.41);
\draw[thick] (0.45,-0.54)--(1.45,-1.73) arc (-50:-35:2.26)--(0.58,-0.41);
\draw[thick,->] (0.58,0.41)--(0.45,0.54);
\draw (0.7,0.55) node[anchor=north east] {\small{$\sigma_\varepsilon$}};
\draw[thick,->] (0.45,-0.54)--(0.58,-0.41);
\draw (0.65,-0.55) node[anchor=south east] {\small{$\sigma_\varepsilon$}};
\draw[thick,->] (1.85,1.3)--(1.28,0.89);
\draw (1.28,0.78) node[anchor=west] {\small{$\gamma_{1,\varepsilon,R}$}};
\draw[thick,->] (0.58,-0.41)--(1.28,-0.89);
\draw (1.2,-0.7) node[anchor=west] {\small{$\gamma_{1,\varepsilon,R}$}};
\draw[thick,->] (1.45,1.73)--(1,1.2);
\draw (1.1,1.3) node[anchor=east] {\small{$\gamma_{2,\varepsilon,R}$}};
\draw[thick,->] (0.45,-0.54)--(1,-1.2);
\draw (1,-1.3) node[anchor=east] {\small{$\gamma_{2,\varepsilon,R}$}};
\draw[thick,->] (1.63,1.57)--(1.6,1.6);
\draw (1.6,1.6) node[anchor=west] {\small{$\sigma_R$}};
\draw[thick,->] (1.71,-1.48)--(1.73,-1.45);
\draw (1.65,-1.6) node[anchor=west] {\small{$\sigma_R$}};
\end{tikzpicture}
\end{center}
\end{minipage}
\begin{minipage}{0.7\textwidth}
\begin{align*}
\sigma_\varepsilon(\varphi)&:=\varepsilon e^{J\varphi},\hspace{1.03cm}\varphi\in(-\varphi_2,-\varphi_1)\cup(\varphi_1,\varphi_2), \\
\sigma_R(\varphi)&:=Re^{J\varphi},\hspace{0.9cm}\varphi\in(-\varphi_2,-\varphi_1)\cup(\varphi_1,\varphi_2), \\
\gamma_{1,\varepsilon,R}(t)&:=\begin{cases} -te^{J\varphi_1}, & t\in(-R,-\varepsilon), \\ te^{-J\varphi_1}, & t\in(\varepsilon,R), \end{cases} \\
\gamma_{2,\varepsilon,R}(t)&:=\begin{cases} -te^{J\varphi_2}, & t\in(-R,-\varepsilon), \\ te^{-J\varphi_2}, & t\in(\varepsilon,R). \end{cases}
\end{align*}
\end{minipage}

\medskip

Then the Cauchy integral theorem gives
\begin{equation}\label{Eq_Integral_independence_1}
\int_{\gamma_{1,\varepsilon,R}}K(s)ds_Jf(s)=\int_{\sigma_R\oplus\gamma_{2,\varepsilon,R}\ominus\sigma_\varepsilon}K(s)ds_Jf(s).
\end{equation}
In the limit $\varepsilon\rightarrow 0^+$, the integral along $\sigma_\varepsilon$ vanishes, since for $\varepsilon\leq 1$ we get
\begin{equation}\label{Eq_Integral_independence_2}
\bigg\Vert\int_{\sigma_\varepsilon}K(s)ds_Jf(s)\bigg\Vert\leq C_{\varphi_1}C_f\int_{\varphi_1<|\varphi|<\varphi_2}\frac{1}{\varepsilon^\alpha}\varepsilon^{\alpha-1+\delta}\varepsilon d\varphi=2C_{\varphi_1}C_f(\varphi_2-\varphi_1)\varepsilon^\delta\overset{\varepsilon\rightarrow 0^+}{\longrightarrow}0.
\end{equation}
Similarly, also the integral along $\sigma_R$ vanishes in the limit $R\rightarrow\infty$, since for $R\geq 1$ we get
\begin{equation}\label{Eq_Integral_independence_3}
\bigg\Vert\int_{\sigma_R}K(s)ds_Jf(s)\bigg\Vert\leq 2C_{\varphi_1}C_f(\varphi_2-\varphi_1)R^{-\delta}\overset{R\rightarrow\infty}{\longrightarrow}0.
\end{equation}
Performing now the limits $\varepsilon\rightarrow 0^+$ and $R\rightarrow\infty$ in \eqref{Eq_Integral_independence_1} and using the fact that the integrals \eqref{Eq_Integral_independence_2} and \eqref{Eq_Integral_independence_3} vanish, we obtain the independence of the angle
\begin{equation*}
\int_{\partial(S_{\varphi_1}\cap\mathbb{C}_J)}K(s)ds_Jf(s)=\int_{\partial(S_{\varphi_2}\cap\mathbb{C}_J)}K(s)ds_Jf(s).
\end{equation*}
For the \textit{independence on the imaginary unit}, we consider $J,I\in\mathbb{S}$. For any three angles $\varphi_1<\varphi_2<\varphi_3\in(\omega,\theta)$ and $\varepsilon>0$ we define the paths \medskip

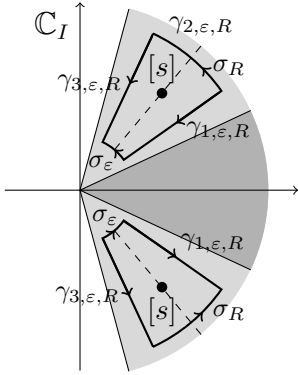
\begin{minipage}{0.29\textwidth}
\begin{center}
\begin{tikzpicture}
\fill[black!15] (0.65,2.41)--(0,0)--(0.65,-2.41) arc (-75:75:2.5);
\draw (0.65,2.41)--(0,0)--(0.65,-2.41);
\fill[black!30] (2.27,1.06)--(0,0)--(2.27,-1.06) arc (-25:25:2.5);
\draw (2.27,1.06)--(0,0)--(2.27,-1.06);
\draw[->] (-1,0)--(2.9,0);
\draw[->] (0,-2.4)--(0,2.5) node[anchor=north east] {\large{$\mathbb{C}_I$}};
\draw[thick] (0.97,2.08)--(0.3,0.64) arc (65:35:0.71)--(1.88,1.32) arc (35:65:2.3);
\draw[thick] (0.97,-2.08)--(0.3,-0.64) arc (-65:-35:0.71)--(1.88,-1.32) arc (-35:-65:2.3);
\draw[thick,->] (0.47,0.53)--(0.46,0.54);
\draw (0.6,0.6) node[anchor=north east] {\small{$\sigma_\varepsilon$}};
\draw[thick,->] (0.45,-0.55)--(0.46,-0.54);
\draw (0.65,-0.6) node[anchor=south east] {\small{$\sigma_\varepsilon$}};
\draw[thick,->] (1.64,1.62)--(1.63,1.63);
\draw (1.63,1.63) node[anchor=west] {\small{$\sigma_R$}};
\draw[thick,->] (1.62,-1.64)--(1.63,-1.63);
\draw (1.63,-1.63) node[anchor=west] {\small{$\sigma_R$}};
\draw[thick,->] (1.85,1.3)--(1.28,0.89);
\draw (1.28,0.78) node[anchor=west] {\small{$\gamma_{1,\varepsilon,R}$}};
\draw[thick,->] (0.58,-0.41)--(1.28,-0.89);
\draw (1.2,-0.7) node[anchor=west] {\small{$\gamma_{1,\varepsilon,R}$}};
\draw[thick,->] (0.96,2.05)--(0.66,1.41);
\draw (0.69,1.41) node[anchor=east] {\small{$\gamma_{3,\varepsilon,R}$}};
\draw[thick,->] (0.3,-0.64)--(0.66,-1.41);
\draw (0.66,-1.41) node[anchor=east] {\small{$\gamma_{3,\varepsilon,R}$}};
\draw[dashed] (1.61,-1.92)--(0.46,-0.54);
\draw[dashed] (0.46,0.54)--(1.61,1.92) node[anchor=south] {\small{$\gamma_{2,\varepsilon,R}$}};
\fill[black] (1.09,1.29) circle (0.07cm) node[anchor=south] {$[s]$};
\fill[black] (1.09,-1.29) circle (0.07cm) node[anchor=north] {$[s]$};
\end{tikzpicture}
\end{center}
\end{minipage}
\begin{minipage}{0.7\textwidth}
\begin{align}
\gamma_{1,\varepsilon,R}(t)&:=\begin{cases} -te^{I\varphi_1}, & t\in(-R,-\varepsilon), \\ te^{-I\varphi_1}, & t\in(\varepsilon,R), \end{cases} \notag \\
\gamma_{2,\varepsilon,R}(t)&:=\begin{cases} -te^{J\varphi_2}, & t\in(-R,-\varepsilon), \\ te^{-J\varphi_2}, & t\in(\varepsilon,R), \end{cases} \label{Eq_Integral_independence_9} \\
\gamma_{3,\varepsilon,R}(t)&:=\begin{cases} -te^{I\varphi_3}, & t\in(-R,-\varepsilon), \\ te^{-I\varphi_3}, & t\in(\varepsilon,R), \end{cases} \notag \\
\sigma_\varepsilon(\varphi)&:=\varepsilon e^{I\varphi},\hspace{0.5cm}\varphi\in(-\varphi_3,-\varphi_1)\cup(\varphi_1,\varphi_3), \notag \\
\sigma_R(\varphi)&:=Re^{I\varphi},\quad \varphi\in(-\varphi_3,-\varphi_1)\cup(\varphi_1,\varphi_3). \notag
\end{align}
\hspace{-1cm} Note, that $\gamma_{1,\varepsilon,R},\gamma_{3,\varepsilon,R},\sigma_\varepsilon,\sigma_R$ are curves in $\mathbb{C}_I$, while $\gamma_{2,\varepsilon,R}$ is in $\mathbb{C}_J$.
\end{minipage}

\medskip

For $s\in\ran(\gamma_{2,\varepsilon,\infty})$ we choose $R>|s|$, such that the Cauchy formula \eqref{Eq_Cauchy_formula} gives
\begin{equation}\label{Eq_Integral_independence_4}
f(s)=\frac{1}{2\pi}\int_{\gamma_{3,\varepsilon,R}\ominus\sigma_\varepsilon\ominus\gamma_{1,\varepsilon,R}\oplus\sigma_R}S_L^{-1}(p,s)dp_If(p).
\end{equation}
In the limit $R\rightarrow\infty$, the integral along $\sigma_R$ vanishes because of
\begin{equation}\label{Eq_Integral_independence_14}
\bigg|\int_{\sigma_R}S_L^{-1}(p,s)dp_If(p)\bigg|\leq C_f\int_{\varphi_1<|\varphi|<\varphi_3}\frac{R+|s|}{(R-|s|)^2}R^{\beta-1-\delta}Rd\varphi\overset{R\rightarrow\infty}{\longrightarrow}0,
\end{equation}
where the integral vanishes since the integrand asymptotically behaves as $\mathcal{O}(R^{\beta-1-\delta})$ and we assumed $\beta\leq 1$. Hence \eqref{Eq_Integral_independence_4} becomes
\begin{equation}\label{Eq_Integral_independence_12}
f(s)=\frac{1}{2\pi}\int_{\gamma_{3,\varepsilon,\infty}\ominus\sigma_\varepsilon\ominus\gamma_{1,\varepsilon,\infty}}S_L^{-1}(p,s)dp_If(p),\qquad s\in\ran(\gamma_{2,\varepsilon,\infty}).
\end{equation}
Next, we consider the curves \medskip

\begin{minipage}{0.34\textwidth}
\begin{center}
\begin{tikzpicture}
\fill[black!30] (1.99,0.93)--(0,0)--(1.99,-0.93) arc (-25:25:2.2);
\draw (1.99,0.93)--(0,0)--(1.99,-0.93);
\draw[->] (-2.3,0)--(2.6,0);
\draw[->] (0,-2.3)--(0,2.4);
\draw (0,1.8) node[anchor=north east] {\large{$\mathbb{C}_J$}};
\fill[black] (0.66,1.4) circle (0.07cm) node[anchor=east] {$[p]$};
\fill[black] (0.66,-1.4) circle (0.07cm) node[anchor=east] {$[p]$};
\draw[dashed] (0.85,1.81)--(0.42,0.91) arc (65:35:1)--(1.64,1.15) arc (35:65:2);
\draw (0.85,1.75) node[anchor=south] {\small{$\gamma_{3,\varepsilon,R}$}};
\draw[dashed] (0.85,-1.81)--(0.42,-0.91) arc (-65:-35:1)--(1.64,-1.15) arc (-35:-65:2);
\draw (1.6,1.26) node[anchor=west] {\small{$\gamma_{1,\varepsilon,R}$}};
\draw[thick] (1.29,1.53)--(0.32,0.38) arc (50:310:0.5)--(1.29,-1.53) arc (310:50:2);
\draw[thick,->] (0.64,-0.77)--(1.09,-1.3);
\draw (0.7,-1) node[anchor=west] {\small{$\gamma_{2,\varepsilon,R}$}};
\draw[thick,->] (1.29,1.53)--(0.91,1.08);
\draw (0.8,0.9) node[anchor=west] {\small{$\gamma_{2,\varepsilon,R}$}};
\draw[thick,->] (0.51,0.61)--(0.45,0.54);
\draw (0.6,0.75) node[anchor=east] {\small{$\gamma_{2,\frac{\varepsilon}{2},\varepsilon}$}};
\draw[thick,->] (0.45,-0.54)--(0.51,-0.61);
\draw (0.6,-0.7) node[anchor=east] {\small{$\gamma_{2,\frac{\varepsilon}{2},\varepsilon}$}};
\draw[thick,->] (-0.45,0.19)--(-0.46,0.17);
\draw (-0.41,0.17) node[anchor=east] {\small{$\tau_{\frac{\varepsilon}{2}}$}};
\draw[thick,->] (-1.86,0.73)--(-1.88,0.69);
\draw (-1.88,0.69) node[anchor=west] {\small{$\tau_R$}};
\end{tikzpicture}
\end{center}
\end{minipage}
\begin{minipage}{0.65\textwidth}
\begin{align}
\tau_{\frac{\varepsilon}{2}}(\varphi)&:=\frac{\varepsilon}{2}e^{J\varphi},\qquad\varphi\in(\varphi_2,2\pi-\varphi_2), \notag \\
\tau_R(\varphi)&:=Re^{J\varphi},\qquad\varphi\in(\varphi_2,2\pi-\varphi_2), \label{Eq_Integral_independence_10} \\
\gamma_{2,\frac{\varepsilon}{2},\varepsilon}(t)&:=\begin{cases} -te^{J\varphi_2}, & t\in(-\varepsilon,-\frac{\varepsilon}{2}), \\ te^{-J\varphi_2}, & t\in(\frac{\varepsilon}{2},\varepsilon). \end{cases} \notag
\end{align}
\end{minipage}

\medskip

For $p\in\ran(\gamma_{3,\varepsilon,\infty})$, we choose $R>|p|$, such that the Cauchy formula gives
\begin{equation}\label{Eq_Integral_independence_6}
K(p)=\frac{-1}{2\pi}\int_{\gamma_{2,\frac{\varepsilon}{2},R+1}\oplus\tau_{\frac{\varepsilon}{2}}\ominus\tau_R}K(s)ds_JS_R^{-1}(s,p)=\frac{1}{2\pi}\int_{\gamma_{2,\frac{\varepsilon}{2},R+1}\oplus\tau_{\frac{\varepsilon}{2}}\ominus\tau_R}K(s)ds_JS_L^{-1}(p,s),
\end{equation}
where the negative sign in front of the first integral above comes from the negative orientation of integration path, and in the second equality we used the connection $S_R^{-1}(s,p)=-S_L^{-1}(p,s)$ between the left and the right Cauchy kernel. So we obtain that
\begin{equation}\label{EQ_PETER}
\bigg\Vert\int_{\tau_R}K(s)ds_JS_L^{-1}(p,s)\bigg\Vert\leq C_{\varphi_2}\int_{\varphi_2}^{2\pi-\varphi_2}R^{-\beta}\frac{|p|+R}{(R-|p|)^2}Rd\varphi\overset{R\rightarrow\infty}{\longrightarrow}0,
\end{equation}
where the integral vanishes since the integrand asymptotically behaves as $\mathcal{O}(R^{-\beta})$ and we assumed $\beta>0$. This reduces the equation \eqref{Eq_Integral_independence_6} to
\begin{equation}\label{Eq_Integral_independence_11}
K(p)=\frac{1}{2\pi}\int_{\gamma_{2,\frac{\varepsilon}{2},\infty}\oplus\tau_{\frac{\varepsilon}{2}}}K(s)ds_JS_L^{-1}(p,s),\qquad p\in\ran(\gamma_{3,\varepsilon,\infty}).
\end{equation}

Let $p\in\ran(\gamma_{1,\varepsilon,\infty})$ we now reason as in \eqref{Eq_Integral_independence_6}, \eqref{EQ_PETER}, \eqref{Eq_Integral_independence_11} with the difference that the left hand side of \eqref{Eq_Integral_independence_6} equals zero instead of $K(p)$ because the points $[p]\cap\mathbb{C}_J$ lie outside the integration path,  so we obtain the formula
\begin{equation}\label{Eq_Integral_independence_5}
0=\frac{1}{2\pi}\int_{\gamma_{2,\frac{\varepsilon}{2},\infty}\oplus\tau_{\frac{\varepsilon}{2}}}K(s)ds_JS_L^{-1}(p,s),\qquad p\in\ran(\gamma_{1,\varepsilon,\infty}).
\end{equation}
Combining now \eqref{Eq_Integral_independence_12}, \eqref{Eq_Integral_independence_11} and \eqref{Eq_Integral_independence_5}, leads to the formula
\begin{align}
\int_{\gamma_{2,\varepsilon,\infty}}K(s)ds_Jf(s)&=\frac{1}{2\pi}\int_{\gamma_{2,\varepsilon,\infty}}K(s)ds_J\bigg(\int_{\gamma_{3,\varepsilon,\infty}\ominus\sigma_\varepsilon\ominus\gamma_{1,\varepsilon,\infty}}S_L^{-1}(p,s)dp_If(p)\bigg) \notag \\
&=\int_{\gamma_{3,\varepsilon,\infty}}\bigg(K(p)-\frac{1}{2\pi}\int_{\gamma_{2,\frac{\varepsilon}{2},\varepsilon}\oplus\tau_{\frac{\varepsilon}{2}}}K(s)ds_JS_L^{-1}(p,s)\bigg)dp_If(p) \notag \\
&\quad+\frac{1}{2\pi}\int_{\gamma_{1,\varepsilon,\infty}}\bigg(\int_{\gamma_{2,\frac{\varepsilon}{2},\varepsilon}\oplus\tau_{\frac{\varepsilon}{2}}}K(s)ds_JS_L^{-1}(p,s)\bigg)dp_If(p) \notag \\
&\quad-\frac{1}{2\pi}\int_{\sigma_\varepsilon}\bigg(\int_{\gamma_{2,\varepsilon,\infty}}K(s)ds_JS_L^{-1}(p,s)\bigg)dp_If(p). \label{Eq_Integral_independence_7}
\end{align}
Note, that in the above manipulations we were allowed to interchange the order of integration since the double integrals over two unbounded paths are absolute convergent due to
\begin{align*}
\bigg\Vert\int_{\gamma_{3,1,\infty}}&\bigg(\int_{\gamma_{2,1,\infty}}K(s)ds_JS_L^{-1}(p,s)\bigg)dp_If(p)\bigg\Vert \\
&\leq 4C_{\varphi_2}C_f\int_1^\infty\int_1^\infty r^{-\beta}\frac{|te^{I\varphi_3}-re^{-J\varphi_2}|}{|te^{I\varphi_3}-re^{I\varphi_2}||te^{I\varphi_3}-re^{-I\varphi_2}|}t^{\beta-1-\delta}drdt \\
&\leq\frac{16C_{\varphi_2}C_f}{|e^{I\varphi_3}-e^{I\varphi_2}||e^{I\varphi_3}-e^{-I\varphi_2}|}\int_1^\infty\int_1^\infty r^{-\beta}\frac{1}{t+r}t^{\beta-1-\delta}drdt \\
&\leq\frac{16C_{\varphi_2}C_f\theta^\theta(1-\theta)^{1-\theta}}{|e^{I\varphi_3}-e^{I\varphi_2}||e^{I\varphi_3}-e^{-I\varphi_2}|}\int_1^\infty\int_1^\infty\frac{1}{r^{\beta+\theta}t^{2+\delta-\theta-\beta}}drdt<\infty,
\end{align*}
where in the second inequality we used $\frac{1}{|te^{I\varphi_3}-re^{\pm I\varphi_2}|}\leq\frac{2}{|e^{I\varphi_3}-e^{\pm I\varphi_2}|(t+r)}$ and in the third inequality $\frac{1}{t+r}\leq\frac{\theta^\theta(1-\theta)^{1-\theta}}{r^\theta t^{1-\theta}}$, for some arbitrary $\theta\in(0,1)$ with $1<\theta+\beta<1+\delta$. Analogously the integral along $\gamma_{1,1,\infty}$ and $\gamma_{2,1,\infty}$ is absolute convergent. Since every $s\in\ran(\gamma_{2,\frac{\varepsilon}{2},\varepsilon}\oplus\tau_{\frac{\varepsilon}{2}})$ lies outside $\gamma_{1,\varepsilon,R}\ominus\gamma_{3,\varepsilon,R}\ominus\sigma_R\oplus\sigma_\varepsilon$, see also the graphic in \eqref{Eq_Integral_independence_9}, the Cauchy formula \eqref{Eq_Cauchy_formula} gives
\begin{equation*}
\int_{\gamma_{1,\varepsilon,\infty}\ominus\gamma_{3,\varepsilon,\infty}}S_L^{-1}(p,s)dp_If(p)=-\int_{\sigma_\varepsilon}S_L^{-1}(p,s)dp_If(p),\qquad s\in\ran(\gamma_{2,\frac{\varepsilon}{2},\varepsilon}\oplus\tau_{\frac{\varepsilon}{2}}),
\end{equation*}
where we used the fact that in the limit $R\rightarrow\infty$ the integral along $\sigma_R$ vanishes, see \eqref{Eq_Integral_independence_14}. This then reduces equation \eqref{Eq_Integral_independence_7} to
\begin{equation}\label{Eq_Integral_independence_8}
\int_{\gamma_{2,\varepsilon,\infty}}K(s)ds_Jf(s)=\int_{\gamma_{3,\varepsilon,\infty}}K(p)dp_If(p)-\frac{1}{2\pi}\int_{\sigma_\varepsilon}\bigg(\int_{\gamma_{2,\frac{\varepsilon}{2},\infty}\oplus\tau_{\frac{\varepsilon}{2}}}K(s)ds_JS_L^{-1}(p,s)\bigg)dp_If(p).
\end{equation}
Finally, we perform the limits when $\varepsilon\rightarrow 0^+$ in this equation and show that the double integral vanishes. For the first part of the claim, we get
\begin{align*}
\bigg\Vert&\int_{\sigma_\varepsilon}\bigg(\int_{\gamma_{2,\frac{\varepsilon}{2},\infty}}K(s)ds_JS_L^{-1}(p,s)\bigg)dp_If(p)\bigg\Vert \\
&\leq 4C_{\varphi_2}C_f\int_{\varphi_1}^{\varphi_3}\bigg(\int_{\frac{\varepsilon}{2}}^1\frac{\varepsilon+r}{r^\alpha|\varepsilon e^{I\varphi}-re^{I\varphi_2}||\varepsilon e^{I\varphi}-re^{-I\varphi_2}|}dr+\int_1^\infty\frac{\varepsilon+r}{r^\beta(\varepsilon-r)^2}dr\bigg)\varepsilon^{\alpha-1+\delta}\varepsilon d\varphi \\
&=4C_{\varphi_2}C_f\varepsilon^\delta\bigg(\int_{\varphi_1}^{\varphi_3}\int_{\frac{1}{2}}^{\frac{1}{\varepsilon}}\frac{1+\rho}{\rho^\alpha|e^{I\varphi}-\rho e^{I\varphi_2}||e^{I\varphi}-\rho e^{-I\varphi_2}|}d\rho d\varphi+\varepsilon^\alpha(\varphi_3-\varphi_1)\int_1^\infty\frac{\varepsilon+r}{r^\beta(\varepsilon-r)^2}dr\bigg) \\
&=4C_{\varphi_2}C_f\varepsilon^\delta\int_{\varphi_1}^{\varphi_3}\bigg(\int_{\frac{1}{2}}^{\frac{1}{\varepsilon}}\frac{1+\rho}{\rho^\alpha|e^{I\varphi}-\rho e^{I\varphi_2}||e^{I\varphi}-\rho e^{-I\varphi_2}|}d\rho+\varepsilon^\alpha\int_1^\infty\frac{\varepsilon+r}{r^\beta(\varepsilon-r)^2}dr\bigg)d\varphi\overset{\varepsilon\rightarrow 0^+}{\longrightarrow}0.
\end{align*}
The second part of the double integral in \eqref{Eq_Integral_independence_8} vanishes in the limit $\varepsilon\rightarrow 0^+$ because of
\begin{align*}
\bigg\Vert\int_{\sigma_\varepsilon}\bigg(\int_{\tau_{\frac{\varepsilon}{2}}}&K(s)ds_JS_L^{-1}(p,s)\bigg)dp_If(p)\bigg\Vert \\
&\leq 4C_{\varphi_1}C_f\int_{\varphi_1}^{\varphi_3}\int_{\varphi_2}^{2\pi-\varphi_2}\frac{\varepsilon+\frac{\varepsilon}{2}}{(\frac{\varepsilon}{2})^\alpha|\varepsilon e^{I\varphi}-\frac{\varepsilon}{2}e^{I\phi}||\varepsilon e^{I\varphi}-\frac{\varepsilon}{2}e^{-I\phi}|}\varepsilon^{\alpha-1+\delta}\frac{\varepsilon}{2}d\phi\varepsilon d\varphi \\
&=2^\alpha 3C_{\varphi_1}C_f\varepsilon^\delta\int_{\varphi_1}^{\varphi_3}\int_{\varphi_2}^{2\pi-\varphi_2}\frac{1}{|e^{I\varphi}-\frac{1}{2}e^{I\phi}||e^{I\varphi}-\frac{1}{2}e^{-I\phi}|}d\phi d\varphi\overset{\varepsilon\rightarrow 0^+}{\longrightarrow}0.
\end{align*}
Therefore, the limit $\varepsilon\rightarrow 0^+$ turns \eqref{Eq_Integral_independence_8} into the desired independence of the imaginary unit
\begin{equation*}
\int_{\partial(S_{\varphi_2}\cap\mathbb{C}_J)}K(s)ds_Jf(s)=\int_{\partial(S_{\varphi_3}\cap\mathbb{C}_I)}K(p)dp_If(p). \qedhere
\end{equation*}
\end{proof}

For the welldefinedness of the functional calculi in Definition \ref{defi_Functional_calculus_decaying} ii) -- iv), we also have to show that there are no two functions $f_1\neq f_2$ for which $\mathcal Df_1=\mathcal Df_2$, $\overline{\mathcal D}f_1=\overline{\mathcal D}f_2$ or $\Delta f_1=\Delta f_2$.

\begin{lem}\label{lem_Kernel_independence}
Let $\alpha\geq 1$, $\beta\leq 1$. Then for every $\theta\in(0,\pi)$, $f\in\Psi_L^{\alpha,\beta}(S_\theta)$ there holds \medskip

\begin{enumerate}
\item[i)] $\forall s\in S_\theta: \mathcal Df(s)=0\Rightarrow\forall s\in S_\theta: f(s)=0$, \\
\item[ii)] $\forall s\in S_\theta: \overline{\mathcal D}f(s)=0\Rightarrow\forall s\in S_\theta: f(s)=0$, \\
\item[iii)] $\forall s\in S_\theta: \Delta f(s)=0\Rightarrow\forall s\in S_\theta: f(s)=0$.
\end{enumerate}
\end{lem}

\begin{proof}
We will prove that in all three cases there exists some $c\in\mathbb{H}$, such that the function $\beta$ in the decomposition \eqref{Eq_Holomorphic_decomposition} is of the form $\beta(x,y)=cy$, for every $(x,y)\in\mathcal{U}$ with $y>0$. This is sufficient, since by the symmetry condition \eqref{Eq_Symmetry_condition} and the continuity of $\beta$ then also $\beta(x,y)=cy$ for every $(x,y)\in\mathcal{U}$. From \eqref{Eq_Cauchy_Riemann_equations} we then get the two differential equations
\begin{equation*}
\frac{\partial}{\partial x}\alpha(x,y)=c\qquad\text{and}\qquad\frac{\partial}{\partial y}\alpha(x,y)=0,
\end{equation*}
which admit the explicit solution $\alpha(x,y)=cx+d$, for some $d\in\mathbb{H}$. Altogether, the function $f$ is then of the form $f(s)=cs+d$, which is only possible for $c=d=0$, since it has to vanish in the limits $|s|\rightarrow 0$ and $|s|\rightarrow\infty$ due to the assumption $f\in\Psi_L^{\alpha,\beta}(S_\theta)$. \medskip

i)\;\;We write the $s_1$-, $s_2$-, $s_3$-derivatives of the Cauchy-Fueter operator $\mathcal D$ in spherical coordinates, i.e. with respect to the decomposition $s=x+Jy$, with $x\in\mathbb{R}$ and $y>0$, as
\begin{equation}\label{Eq_Kernel_independence_1}
\mathcal D=\frac{\partial}{\partial s_0}+e_1\frac{\partial}{\partial s_1}+e_2\frac{\partial}{\partial s_2}+e_3\frac{\partial}{\partial s_3}=\frac{\partial}{\partial x}+J\frac{\partial}{\partial y}+\frac{J\Gamma_J}{y},
\end{equation}
where $\Gamma_J$ is a symbol for the angular derivatives. Using now the identity
\begin{equation}\label{Eq_Angular_derivative}
J\Gamma_JJ=\Gamma_J-2,
\end{equation}
see \cite[Paragraph 1.12.1]{DSS92}, it follows from the assumption $\mathcal Df(s)=0$, that $\alpha$ and $\beta$ satisfy
\begin{equation}\label{Eq_Kernel_independence_2}
\Big(\frac{\partial}{\partial x}+J\frac{\partial}{\partial y}+\frac{J\Gamma_J}{y}\Big)\alpha(u,v)+\Big(J\frac{\partial}{\partial x}-\frac{\partial}{\partial y}+\frac{\Gamma_J-2}{y}\Big)\beta(u,v)=0.
\end{equation}
Since $\alpha$ and $\beta$ also satisfy the Cauchy-Riemann equations \eqref{Eq_Cauchy_Riemann_equations} and $\Gamma_J\alpha=\Gamma_J\beta=0$ vanish since $\alpha$ and $\beta$ only depend on $x$ and the radial variable $y$, this turns \eqref{Eq_Kernel_independence_2} into $\frac{2}{y}\beta(x,y)=0$ and hence $\beta(x,y)=0$. \medskip

ii)\;\;Similar to \eqref{Eq_Kernel_independence_1}, also the conjugate Cauchy-Fueter operator $\overline{\mathcal D}$ can be written as
\begin{equation}\label{Eq_Kernel_independence_3}
\overline{\mathcal D}=\frac{\partial}{\partial s_0}-e_1\frac{\partial}{\partial s_1}-e_2\frac{\partial}{\partial s_2}-e_3\frac{\partial}{\partial s_3}=\frac{\partial}{\partial x}-J\frac{\partial}{\partial y}-\frac{J\Gamma_J}{y},
\end{equation}
and from the assumption $\overline{\mathcal D}f(s)=0$, together with \eqref{Eq_Angular_derivative}, there follows
\begin{equation*}
\Big(\frac{\partial}{\partial x}-J\frac{\partial}{\partial y}-\frac{J\Gamma_J}{y}\Big)\alpha(x,y)+\Big(J\frac{\partial}{\partial x}+\frac{\partial}{\partial y}-\frac{\Gamma_J-2}{y}\Big)\beta(x,y)=0.
\end{equation*}
With the Cauchy-Riemann equations \eqref{Eq_Cauchy_Riemann_equations} and $\Gamma_J\alpha=\Gamma_J\beta=0$, this equation reduces to
\begin{equation*}
\Big(J\frac{\partial}{\partial x}+\frac{\partial}{\partial y}+\frac{1}{y}\Big)\beta(x,y)=0.
\end{equation*}
Since this equation has to be satisfied for every $J\in\mathbb{S}$, while the function $\beta$ may not depend on $J$, the real and the imaginary part of this equation has to be satisfied separately. This leads to the two ordinary differential equations
\begin{equation*}
\frac{\partial}{\partial x}\beta(x,y)=0\qquad\text{and}\qquad\frac{\partial}{\partial y}\beta(x,y)=-\frac{1}{y}\beta(x,y),
\end{equation*}
which have the explicit solution $\beta(x,y)=\frac{c}{y}$, $c\in\mathbb{H}$. However, since the positive real line is contained in the domain $S_\theta$, where $f$ is holomorphic, this is only possible for $c=0$. \medskip

iii)\;\;Combining \eqref{Eq_Kernel_independence_1}, \eqref{Eq_Kernel_independence_3} and using \eqref{Eq_Angular_derivative}, gives the Laplace operator in spherical coordinates
\begin{equation*}
\Delta=\mathcal D\overline{\mathcal D}=\frac{\partial^2}{\partial x^2}+\frac{\partial^2}{\partial y^2}+\frac{2}{y}\frac{\partial}{\partial y}+\frac{\Gamma_J-\Gamma_J^2}{y^2}.
\end{equation*}
From the assumption $\Delta f(s)=0$ and with $\Gamma_JJ=2J-J\Gamma_J$, there follows
\begin{equation*}
\Big(\frac{\partial^2}{\partial x^2}+\frac{\partial^2}{\partial y^2}+\frac{2}{y}\frac{\partial}{\partial y}+\frac{\Gamma_J-\Gamma_J^2}{y^2}\Big)\alpha(x,y)+J\Big(\frac{\partial^2}{\partial x^2}+\frac{\partial^2}{\partial y^2}+\frac{2}{y}\frac{\partial}{\partial y}-\frac{2-3\Gamma_J+\Gamma_J^2}{y^2}\Big)\beta(x,y)=0.
\end{equation*}
Since the Cauchy-Riemann equations \eqref{Eq_Cauchy_Riemann_equations} in particular imply
\begin{equation*}
\frac{\partial^2\alpha}{\partial x^2}=\frac{\partial^2\alpha}{\partial y^2}=\frac{\partial^2\beta}{\partial x^2}=\frac{\partial^2\beta}{\partial y^2}=0,
\end{equation*}
together with $\Gamma_J\alpha=\Gamma_J\beta=0$, they reduce the above equation to
\begin{equation*}
\Big(\frac{\partial}{\partial x}-J\frac{\partial}{\partial y}+\frac{J}{y}\Big)\beta(x,y)=0.
\end{equation*}
Since this equation has to be satisfied for every $J\in\mathbb{S}$, while the function $\beta$ may not depend on $J$, the real and the imaginary part of this equation has to be satisfied separately. This leads to the two ordinary differential equations
\begin{equation*}
\frac{\partial}{\partial x}\beta(x,y)=0\qquad\text{and}\qquad\frac{\partial}{\partial y}\beta(u,v)=\frac{1}{y}\beta(x,y),
\end{equation*}
which has the explicit solution $\beta(x,y)=cy$, for some constant $c\in\mathbb{H}$.
\end{proof}

Next we give a proper meaning to the functional calculi  \eqref{Eq_Functional_calculus_formal}. In particular, Theorem \ref{thm_Integral_independence} together with Lemma \ref{lem_Resolvent_estimates} show that the integrals converge and are independent of the integration path $\partial(S_\varphi\cap\mathbb{C}_J)$ for every $J\in \mathbb{S}$. It is moreover proven in Lemma \ref{lem_Kernel_independence}, that the functional calculus is independent of the chosen representative $f$ in the spaces \eqref{Eq_Functionspaces}.

\begin{defi}\label{defi_Functional_calculus_decaying}
Let $\alpha\geq\frac{1}{3}$, $\beta\in(0,\frac{1}{3}]$, $\omega\in(0,\pi)$ and $T$ of type $(\alpha,\beta,\omega)$. Then for every $\theta\in(\omega,\pi)$, $f\in\Psi_L^{3\alpha,3\beta}(S_\theta)$, we define
\begin{align*}
f(T):=&\frac{1}{2\pi}\int_{\partial(S_\varphi\cap\mathbb{C}_J)}S_L^{-1}(s,T)ds_Jf(s), && \textit{($S$-functional calculus)} \\
\mathcal Df(T):=&\frac{-1}{\pi}\int_{\partial(S_\varphi\cap\mathbb{C}_J)}Q_{c,s}^{-1}(T)ds_Jf(s), && \textit{($Q$-functional calculus)} \\
\overline{\mathcal D}f(T):=&\frac{1}{2\pi}\int_{\partial(S_\varphi\cap\mathbb{C}_J)}P_2^L(s,T)ds_Jf(s), && \textit{($P_2$-functional calculus)} \\
\Delta f(T):=&\frac{1}{2\pi}\int_{\partial(S_\varphi\cap\mathbb{C}_J)}F_L(s,T)ds_Jf(s). && \textit{($F$-functional calculus)}
\end{align*}
\end{defi}

For $\widehat{f}\in\{f,\mathcal Df,\overline{\mathcal D}f,\Delta f\}$ we will also use the kernels $K_L(s,T)$ from Lemma \ref{lem_K_decomposition} and write
\begin{equation}\label{Eq_K_functional_calculus}
\widehat{f}(T)=\frac{1}{2\pi}\int_{\partial(S_\varphi\cap\mathbb{C}_J)}K_L(s,T)ds_Jf(s).
\end{equation}
The next lemma collects some basic properties of the functional calculi in Definition \ref{defi_Functional_calculus_decaying}.

\begin{lem}\label{lem_Properties_decaying}
Let $\alpha\geq\frac{1}{3}$, $\beta\in(0,\frac{1}{3}]$, $\omega\in(0,\pi)$ and $T$ of type $(\alpha,\beta,\omega)$. Then for every $\theta\in(\omega,\pi)$, $f\in\Psi_L^{3\alpha,3\beta}(S_\theta)$ and any choice $\widehat{f}\in\{f,\mathcal Df,\overline{\mathcal D}f,\Delta f\}$, there holds \medskip

\begin{enumerate}
\item[i)] $\widehat{f}(T)\in\mathcal{BC}(V)$, \\
\item[ii)] If $f$ is intrinsic, then $\overline{\widehat{f}(T)}=\widehat{f}(\overline{T})$; \\
\item[iii)] If $f$ is intrinsic, we can use the right resolvent $K_R(s,T)$ from Lemma \ref{lem_K_decomposition}, to write
\begin{equation}\label{Eq_Integral_KR}
\widehat{f}(T)=\frac{1}{2\pi}\int_{\partial(S_\varphi\cap\mathbb{C}_J)}f(s)ds_JK_R(s,T).
\end{equation}
\end{enumerate}
\end{lem}

\begin{proof}
i)\;\;The boundedness of the operator $\widehat{f}(T)$ follows immediately from the estimate \eqref{Eq_Integral_independence_13}. In order to show that the components of $\widehat{f}(T)$ commute, we decompose the kernel $K_L(s,T)$ according to \eqref{Eq_K_decomposition} into
\begin{equation}\label{Eq_Properties_decaying_1}
K_L(te^{\pm J\varphi},T)=A(t\cos\varphi,t\sin\varphi,T)\pm B(t\cos\varphi,t\sin\varphi,T)J.
\end{equation}
With the property \eqref{Eq_AB_sbar_property} of the operators $A$ and $B$, we can write \eqref{Eq_K_functional_calculus} as
\begin{align}
\widehat{f}(T)&=\frac{1}{2\pi}\int_{-\infty}^0K_L(-te^{J\varphi},T)Je^{J\varphi}f(-te^{J\varphi})dt-\frac{1}{2\pi}\int_0^\infty K_L(te^{-J\varphi},T)Je^{-J\varphi}f(te^{-J\varphi})dt \notag \\
&=\frac{1}{2\pi}\int_0^\infty\Big(K_L(te^{J\varphi},T)Je^{J\varphi}f(te^{J\varphi})-K_L(te^{-J\varphi},T)Je^{-J\varphi}f(te^{-J\varphi})\Big)dt \notag \\
&=\frac{1}{2\pi}\int_0^\infty\Big(A(t\cos\varphi,t\sin\varphi,T)J\big(e^{J\varphi}f(te^{J\varphi})-e^{-J\varphi}f(te^{-J\varphi})\big) \notag \\
&\hspace{3cm}-B(t\cos\varphi,t\sin\varphi,T)\big(e^{J\varphi}f(te^{J\varphi})+e^{-J\varphi}f(te^{-J\varphi})\big)\Big)dt. \label{Eq_Properties_decaying_2}
\end{align}
This representation shows that the components of $\widehat{f}(T)$ are integrals over linear combinations of the components of $A$ and $B$. However, the components of $A$ and $B$ do pairwise commute, which can either be seen from their explicit form in \eqref{Eq_Q_decomposition}, \eqref{Eq_SL_decomposition}, \eqref{Eq_FL_decomposition} and \eqref{Eq_PL_decomposition}, or it is also a consequence of \eqref{Eq_AB_Ti_commutation}, \eqref{Eq_AB_C_commutation} and Lemma \ref{lem_Commutation_of_components}. Hence also the components of $\widehat{f}(T)$ commute. \medskip

ii)\;\;If we assume that $f$ is intrinsic, we know that $f(te^{-J\varphi})=\overline{f(te^{J\varphi})}\in\mathbb{C}_J$ and hence the integral \eqref{Eq_Properties_decaying_2} simplifies to
\begin{align}
\widehat{f}(T)=\frac{1}{\pi}\int_0^\infty\Big(A(t\cos\varphi,&t\sin\varphi,T)\Re\big(Je^{J\varphi}f(te^{J\varphi})\big) \notag \\
&-B(t\cos\varphi,t\sin\varphi,T)\Re\big(e^{J\varphi}f(te^{J\varphi})\big)\Big)dt. \label{Eq_Properties_decaying_3}
\end{align}
Hence it follows from \eqref{Eq_AB_Tbar_property}, that $\overline{\widehat{f}(T)}=\widehat{f}(\overline{T})$. \medskip

iii)\;\;Similar to \eqref{Eq_Properties_decaying_1}, we can write the right kernel $K_R(s,T)$ from \eqref{Eq_K_decomposition} as
\begin{equation*}
K_R(te^{\pm J\varphi},T)=A(t\cos\varphi,t\sin\varphi,T)\pm JB(t\cos\varphi,t\sin\varphi,T).
\end{equation*}
With the same calculations as in \eqref{Eq_Properties_decaying_2} and \eqref{Eq_Properties_decaying_3} we get
\begin{align}
\frac{1}{2\pi}\int_{\partial(S_\varphi\cap\mathbb{C}_J)}f(s)ds_JK_R(s,T)=\frac{1}{\pi}\int_0^\infty\Big(&\Re\big(f(te^{J\varphi})Je^{J\varphi}\big)A(t\cos\varphi,t\sin\varphi,T) \notag \\
&-\Re\big(f(te^{J\varphi})e^{J\varphi}\big)B(t\cos\varphi,t\sin\varphi,T)\Big)dt. \label{Eq_Properties_decaying_4}
\end{align}
Since the right hand sides of \eqref{Eq_Properties_decaying_3} and \eqref{Eq_Properties_decaying_4} coincide, the representation \eqref{Eq_Integral_KR} is proven.
\end{proof}

\begin{prop}\label{prop_Commutation_B}
Let $\alpha\geq\frac{1}{3}$, $\beta\in(0,\frac{1}{3}]$, $\omega\in(0,\pi)$ and $T$ of type $(\alpha,\beta,\omega)$. Moreover, let $B\in\mathcal{B}(V)$ which commutes with $T,T_0,T_1,T_2,T_3$ on $\dom(T)$. Then for every $\theta\in(\omega,\pi)$, $g\in\Psi^{3\alpha,3\beta}(S_\theta)$ and any choice $\widehat{g}\in\{g, \mathcal Dg,\overline{\mathcal D}g,\Delta g\}$, there also commutes
\begin{equation}\label{Eq_Commutation_B}
B\,\widehat{g}(T)=\widehat{g}(T)B.
\end{equation}
\end{prop}

\begin{proof}
Since $B$ commutes with $T,T_0,T_1,T_2,T_3$, it is stated in \eqref{Eq_AB_C_commutation}, that $B$ also commutes with the operators $A(t,x,y)$ and $B(t,x,y)$ in the decomposition \eqref{Eq_K_decomposition}. It follows then from the representation \eqref{Eq_Properties_decaying_3} of the functional calculus, that $B$ also commutes with $\widehat{g}(T)$.
\end{proof}

\begin{cor}\label{cor_Commutation_decaying}
Let $\alpha\geq\frac{1}{3}$, $\beta\in(0,\frac{1}{3}]$, $\omega\in(0,\pi)$ and $T$ of type $(\alpha,\beta,\omega)$. Then for $\theta\in(\omega,\pi)$, $f,g\in\Psi_L^{3\alpha,3\beta}(S_\theta)$ and any choice $\widehat{f}\in\{f,\mathcal Df,\overline{\mathcal D}f,\Delta f\}$ and $\widehat{g}\in\{g,\mathcal Dg,\overline{\mathcal D}g,\Delta g\}$, there holds \medskip

\begin{enumerate}
\item[i)] $\widehat{f}(T)_i\widehat{g}(T)_j=\widehat{g}(T)_j\widehat{f}(T)_i$\quad and \quad $\widehat{f}(\overline{T})_i\widehat{g}(T)_j=\widehat{g}(T)_j\widehat{f}(\overline{T})_i$,\qquad $i,j\in\{0,1,2,3\}$. \\
\item[ii)] If $f,g$ are intrinsic, then \quad $\widehat{f}(T)\widehat{g}(T)=\widehat{g}(T)\widehat{f}(T)$\quad and\quad $\widehat{f}(\overline{T})\widehat{g}(T)=\widehat{g}(T)\widehat{f}(\overline{T})$. \\
\item[iii)] $\widehat{f}(T)T_j=T_j\widehat{f}(T)$,\quad on $\dom(T)$,\qquad $j\in\{0,1,2,3\}$. \\
\item[iv)] If $f$ is intrinsic, then\quad $\widehat{f}(T)T=T\widehat{f}(T)$\quad and\quad $\widehat{f}(T)\overline{T}=\overline{T}\widehat{f}(T)$,\quad on $\dom(T)$.
\end{enumerate}
\end{cor}

\begin{proof}
iii), iv)\;\;Since, by \eqref{Eq_AB_Ti_commutation}, the operators $A$ and $B$ of the decomposition of the kernel of $\widehat{f}(T)$ commute with $T_0,T_1,T_2,T_3$, it follows from the integral representation \eqref{Eq_Properties_decaying_2}, that also $\widehat{f}(T)$ commutes with $T_0,T_1,T_2,T_3$, on $\dom(T)$. If $f$ is intrinsic, it follows from \eqref{Eq_Properties_decaying_3} that $\widehat{f}(T)$ even commutes with $T,\overline{T}$, on $\dom(T)$. \medskip

i)\;\;It is shown in iii) that $\widehat{g}(T)$ commutes with $T_0,T_1,T_2,T_3$, on $\dom(T)$. It follows then from Lemma \ref{lem_Commutation_of_components} that its components $\widehat{g}(T)_j$ commute with $T,T_0,T_1,T_2,T_3$ as well. The property \eqref{Eq_AB_C_commutation} shows that $\widehat{g}(T)_j$ commutes with the operators $A(x,y,T)$ and $B(x,y,T)$ from the decomposition of the resolvent of $\widehat{f}(T)$. Again, by the representation \eqref{Eq_Properties_decaying_2} of $\widehat{f}(T)$ we get $\widehat{g}(T)_j\widehat{f}(T)=\widehat{f}(T)\widehat{g}(T)_j$ and by Lemma \ref{lem_Commutation_of_components} then the commutation of the components $\widehat{g}(T)_j\widehat{f}(T)_i=\widehat{f}(T)_i\widehat{g}(T)_j$. The commutation $\widehat{g}(T)_j\widehat{f}(\overline{T})_i=\widehat{f}(\overline{T})_i\widehat{g}(T)_j$, follows analogously. \medskip

ii)\;\;It is already shown in iii), iv) that $\widehat{g}(T)$ commutes with $T,T_0,T_1,T_2,T_3$, on $\dom(T)$. By \eqref{Eq_AB_C_commutation} it then also commutes with $A(x,y,T)$ and $B(x,y,T)$. Since $f$ is intrinsic as well, the commutation of $\widehat{f}(T)$ and $\widehat{g}(T)$ then follows from the representation \eqref{Eq_Properties_decaying_3} of the integrals. The commutation of $\widehat{f}(\overline{T})$ and $\widehat{g}(T)$ follows analogously.
\end{proof}

Next we want to derive the very important product rules \eqref{Eq_Product_rules_formal} of the four functional calculi. The basic ingredient will be the following resolvent identities.

\begin{lem}\label{lem_Resolvent_identities}
Let $T,\overline{T}\in\mathcal{KC}(V)$ and $D$ as in Assumption \ref{ass_D}. Then for every $s,p\in\rho_F(T)$ with $s\notin[p]$, there holds the following resolvent identities
\begin{subequations}
\begin{align}
\text{i)}\;\;&\big(S_R^{-1}(s,T)p-S_L^{-1}(p,T)p-\overline{s}S_R^{-1}(s,T)+\overline{s}S_L^{-1}(p,T)\big)(p^2-2s_0p+|s|^2)^{-1} \notag \\
&\hspace{9.55cm}=S_R^{-1}(s,T)S_L^{-1}(p,T), \label{Eq_S_resolvent_identity} \\
\text{ii)}\;\;&\big(Q_{c,s}^{-1}(T)p-Q_{c,p}^{-1}(T)p-\overline{s}Q_{c,s}^{-1}(T)+\overline{s}Q_{c,p}^{-1}(T)\big)(p^2-2s_0p+|s|^2)^{-1} \notag \\
&\hspace{6.5cm}=Q_{c,s}^{-1}(T)S_L^{-1}(p,T)+S_R^{-1}(p,\overline{T})Q_{c,p}^{-1}(T) \label{Eq_Q_resolvent_identity} \\
&\hspace{6.5cm}=Q_{c,s}^{-1}(T)S_L^{-1}(p,\overline{T})+S_R^{-1}(p,T)Q_{c,p}^{-1}(T), \notag \\
\text{iii)}\;\;&\big(P_2^R(s,T)p-P_2^L(p,T)p-\overline{s}P_2^R(s,T)+\overline{s}P_2^L(p,T)\big)(p^2-2s_0p+|s|^2)^{-1} \notag \\
&\hspace{0.2cm}=P_2^R(s,T)S_L^{-1}(p,T)+S_R^{-1}(s,T)P_2^L(p,T)-2Q_{c,s}^{-1}(T)\big(S_L^{-1}(p,T)-S_L^{-1}(p,\overline{T})\big), \label{Eq_P_resolvent_identity} \\
\text{iv)}\;\;&\big(F_R(s,T)p-F_L(p,T)p-\overline{s}F_R(s,T)+\overline{s}F_L(p,T)\big)(p^2-2s_0p+|s|^2)^{-1} \notag \\
&\hspace{3cm}=F_R(s,T)S_L^{-1}(p,T)+S_R^{-1}(s,T)F_L(p,T)-4Q_{c,s}^{-1}(T)Q_{c,p}^{-1}(T). \label{Eq_F_resolvent_identity}
\end{align}
\end{subequations}
\end{lem}

\begin{proof}
Although the assumptions on the operator $T$ and the domain of the operator $Q_{c,s}(T)$ are different, the proof of the $S$-resolvent identity \eqref{Eq_S_resolvent_identity} and the $Q$-resolvent identities \eqref{Eq_Q_resolvent_identity} follow the same steps as in \cite[Theorem 2.33]{CG18_2} and \cite[Lemma 3.9]{MPS23}. For the proof of the $P_2$-resolvent identity \eqref{Eq_P_resolvent_identity}, we use \eqref{Eq_S_resolvent_identity} and
\begin{align}
P_2^R(s,&T)S_L^{-1}(p,T)(p^2-2s_0p+|s|^2) \notag \\
&=2\big(S_R^{-1}(s,T)+S_R^{-1}(s,\overline{T})\big)S_R^{-1}(s,T)S_L^{-1}(p,T)(p^2-2s_0p+|s|^2) \notag \\
&=2\big(S_R^{-1}(s,T)+S_R^{-1}(s,\overline{T})\big)\big(S_R^{-1}(s,T)p-S_L^{-1}(p,T)p-\overline{s}S_R^{-1}(s,T)+\overline{s}S_L^{-1}(p,T)\big) \notag \\
&=P_2^R(s,T)p-\overline{s}P_2^R(s,T)-2\big(S_R^{-1}(s,T)+S_R^{-1}(s,\overline{T})\big)
\big(S_L^{-1}(p,T)p-\overline{s}S_L^{-1}(p,T)\big) \notag \\
&=P_2^R(s,T)p-\overline{s}P_2^R(s,T)-4Q_{c,s}^{-1}(T)(s-T_0)
\big((p-\overline{T})p-\overline{s}(p-\overline{T})\big)Q_{c,p}^{-1}(T). \label{Eq_P_resolvent_identity_1}
\end{align}
The same calculation also gives
\begin{align}
S_R^{-1}&(s,T)P_2^L(p,T)(p^2-2s_0p+|s|^2) \notag \\
&=\overline{s}P_2^L(p,T)-P_2^L(s,T)p+4Q_{c,s}^{-1}(T)\big((s-\overline{T})p-\overline{s}(s-\overline{T})\big)(p-T_0)Q_{c,s}^{-1}(T). \label{Eq_P_resolvent_identity_2}
\end{align}
Adding now \eqref{Eq_P_resolvent_identity_1} and \eqref{Eq_P_resolvent_identity_2} leads to the stated $P_2$-resolvent identity
\begin{align*}
\big(P_2^R&(s,T)S_L^{-1}(p,T)+S_R^{-1}(s,T)P_2^L(p,T)\big)(p^2-2s_0p+|s|^2) \\
&=P_2^R(s,T)p-\overline{s}P_2^R(s,T)+\overline{s}P_2^L(p,T)-P_2^L(p,T)p \\
&\quad-4Q_{c,s}^{-1}(T)\Big((s-T_0)\big((p-\overline{T})p-\overline{s}(p-\overline{T})\big)-\big((s-\overline{T})p-\overline{s}(s-\overline{T})\big)(p-T_0)\Big)Q_{c,p}^{-1}(T) \\
&=P_2^R(s,T)p-\overline{s}P_2^R(s,T)+\overline{s}P_2^L(p,T)-P_2^L(p,T)p \\
&\quad+2Q_{c,s}^{-1}(T)(T-\overline{T})(p^2-2s_0p+|s|^2)Q_{c,p}^{-1}(T) \\
&=P_2^R(s,T)p-\overline{s}P_2^R(s,T)+\overline{s}P_2^L(p,T)-P_2^L(p,T)p \\
&\quad+2Q_{c,s}^{-1}(T)\big(S_L^{-1}(p,T)-S_L^{-1}(p,\overline{T})\big)(p^2-2s_0p+|s|^2).
\end{align*}
For the proof of the $F$-resolvent identity \eqref{Eq_F_resolvent_identity}, we again use \eqref{Eq_S_resolvent_identity} and get
\begin{align*}
\big(&F_R(s,T)S_L^{-1}(p,T)+S_R^{-1}(s,T)F_L(p,T)\big)(p^2-2s_0p+|s|^2) \\
&=-4\big(Q_{c,s}^{-1}(T)S_R^{-1}(s,T)S_L^{-1}(p,T)+S_R^{-1}(s,T)S_L^{-1}(p,T)Q_{c,p}^{-1}(T)\big)(p^2-2s_0p+|s|^2) \\
&=-4Q_{c,s}^{-1}(T)\big(S_R^{-1}(s,T)p-S_L^{-1}(p,T)p-\overline{s}S_R^{-1}(s,T)+\overline{s}S_L^{-1}(p,T)\big) \\
&\quad-4\big(S_R^{-1}(s,T)p-S_L^{-1}(p,T)p-\overline{s}S_R^{-1}(s,T)+\overline{s}S_L^{-1}(p,T)\big)Q_{c,p}^{-1}(T) \\
&=F_R(s,T)p-\overline{s}F_R(s,T)-F_L(p,T)p+\overline{s}F_L(p,T) \\
&\quad+4Q_{c,s}^{-1}\big((p-\overline{T})p-\overline{s}(p-\overline{T})-(s-\overline{T})p+\overline{s}(s-\overline{T})\big)Q_{c,p}^{-1}(T) \\
&=F_R(s,T)p-\overline{s}F_R(s,T)-F_L(p,T)p+\overline{s}F_L(p,T)+4Q_{c,s}^{-1}(T)(p^2-2s_0p+|s|^2)Q_{c,p}^{-1}(T). \qedhere
\end{align*}
\end{proof}

The forthcoming lemma holds significant importance for the upcoming product rule in Theorem \ref{thm_Product_rule_decaying}.

\begin{lem}\label{lem_Integral_identity}
Let $B\in\mathcal{B}(V)$, $g\in\Psi^{\alpha,\beta}(S_\theta)$ for some $\alpha\geq 1$, $\beta\leq 1$, $\theta\in(0,\pi)$. Then for every $\varphi\in(0,\theta)$, $J\in\mathbb{S}$ there holds
\begin{equation}\label{Eq_Integral_identity}
Bg(p)=\frac{1}{2\pi}\int_{\partial(S_\varphi\cap\mathbb{C}_J)}g(s)ds_J(\overline{s}B-Bp)(p^2-2s_0p+|s|^2)^{-1},\qquad p\in S_\varphi.
\end{equation}
\end{lem}

\begin{proof}
First, we note that the integral \eqref{Eq_Integral_identity} is absolute convergent due to the asymptotics
\begin{align}
g(s)(\overline{s}B-Bp)(p^2-2s_0p+|s|^2)^{-1}&=\mathcal{O}(|s|^{\alpha-1+\delta}),\quad\text{as }|s|\rightarrow 0^+, \notag \\
g(s)(\overline{s}B-Bp)(p^2-2s_0p+|s|^2)^{-1}&=\mathcal{O}(|s|^{\beta-\delta-2}),\quad\text{as }|s|\rightarrow\infty. \label{Eq_Integral_identity_4}
\end{align}
Next, one immediately verifies the identity
\begin{equation*}
(s^2-2p_0s+|p|^2)(\overline{s}B-Bp)=(sB-B\overline{p})(p^2-2s_0p+|s|^2),
\end{equation*}
by expanding both sides of the equation. Multiplying $(s^2-2p_0s+|p|^2)^{-1}$ from the left, $(p^2-2s_0p+|s|^2)^{-1}$ from the right and plugging it into \eqref{Eq_Integral_identity}, gives
\begin{equation*}
\int_{\partial(S_\varphi\cap\mathbb{C}_J)}g(s)ds_J
(\overline{s}B-Bp)(p^2-2s_0p+|s|^2)^{-1}=\int_{\partial(S_\varphi\cap\mathbb{C}_J)}
\frac{g(s)}{(s-p_J)(s-\overline{p_J})}ds_J(sB-B\overline{p}),
\end{equation*}
where we factorized $p^2-2s_0p+|s|^2=(s-p_J)(s-\overline{p_J})$, using those two zeros $\{p_J,\overline{p_J}\}:=[p]\cap\mathbb{C}_J$ which lie in the complex plane $\mathbb{C}_J$. Apart from the constant factors $B$ and $\overline{p}$, this makes the right hand side a classical complex path integral in the complex plane $\mathbb{C}_J$. If we close the path $\partial(S_\varphi\cap\mathbb{C}_J)$ on the right at infinity, the integral along this path vanishes due to the asymptotic decay \eqref{Eq_Integral_identity_4}. Since this closed path surrounds both singularities $p_J$ and $\overline{p_J}$, we are able to evaluate the integrals using the Cauchy formula and distinguish two cases: \medskip

$\circ$\;\;If $\Im(p_J)\neq 0$, i.e. $p_J\neq\overline{p_J}$, we get
\begin{align*}
\frac{1}{2\pi}\int_{\partial(S_\varphi\cap\mathbb{C}_J)}\frac{g(s)}{(s-p_J)(s-\overline{p_J})}ds_J(sB-B\overline{p})&=\frac{g(p_J)}{p_J-\overline{p_J}}(p_JB-B\overline{p})+\frac{g(\overline{p_J})}{\overline{p_J}-p_J}(\overline{p_J}B-B\overline{p}) \\
&=\frac{g(p_J)p_J-g(\overline{p_J})\overline{p_J}}{p_J-\overline{p_J}}B-\frac{g(p_J)-g(\overline{p_J})}{p_J-\overline{p_J}}B\overline{p} \\
&=B\Big(\frac{g(p)p-g(\overline{p})\overline{p}}{p-\overline{p}}-\frac{g(p)-g(\overline{p})}{p-\overline{p}}\overline{p}\Big)=Bg(p),
\end{align*}
where we are allowed to replace $p_J$ by $p$ and shift $B$ to the left since
\begin{equation*}
\frac{g(p_J)p_J-g(\overline{p_J})\overline{p_J}}{p_J-\overline{p_J}}=\frac{g(p)p-g(\overline{p})\overline{p}}{p-\overline{p}}\qquad\text{and}\qquad\frac{g(p_J)-g(\overline{p_J})}{p_J-\overline{p_J}}=\frac{g(p)-g(\overline{p})}{p-\overline{p}},
\end{equation*}
and both equations are real valued. This can be seen by decomposing $p=u+Iv$ for $u,v\in\mathbb{R}$, $I\in\mathbb{S}$, which leads to $p_J=u+Jv$ and according to \eqref{Eq_Holomorphic_decomposition} also to $g(p_J)=\alpha(u,v)+J\beta(u,v)$ and $g(p)=\alpha(u,v)+I\beta(u,v)$ with real valued functions $\alpha$ and $\beta$. Moreover, in the fourth line we are allowed to shift the operator $B$ all the way to the left since the two fractions are real valued. \medskip

$\circ$\;\;If $\Im(p_J)=0$, i.e. $p_J=\overline{p_J}=p\in\mathbb{R}$, we get by the Cauchy formula of the derivative
\begin{align*}
\frac{1}{2\pi}\int_{\partial(S_\varphi\cap\mathbb{C}_J)}\frac{g(s)}{(s-p)^2}ds_J(sB-Bp)&=\frac{d}{ds}\big(g(s)(sB-Bp)\big)\Big|_{s=p} \\
&=g'(p)(pB-Bp)+g(p)B=Bg(p),
\end{align*}
where we interchange $pB=Bp$ and $g(p)B=Bg(p)$ since both $p$ and $g(p)$ are real valued.
\end{proof}

Equipped with the resolvent formulas in Lemma \ref{lem_Resolvent_identities} and the integral identity of Lemma \ref{lem_Integral_identity}, we are now ready to prove the product rules of the four functional calculi in Definition \ref{defi_Functional_calculus_decaying}. In particular such product rules will be the starting point for the $H^\infty$-versions of the functional calculi for the fine structure, see Definition \ref{defi_Functional_calculus_growing}.

\begin{thm}\label{thm_Product_rule_decaying}
Let $\alpha\geq\frac{1}{3}$, $\beta\in(0,\frac{1}{3}]$, $\omega\in(0,\pi)$ and $T$ of type $(\alpha,\beta,\omega)$. Then for any $\theta\in(\omega,\pi)$, $g\in\Psi^{3\alpha,3\beta}(S_\theta)$, $f\in\Psi_L^{3\alpha,3\beta}(S_\theta)$ we obtain the product rules
\begin{subequations}
\begin{align}
\text{i)}\;\;\;\;\;(gf)(T)&=g(T)f(T), \label{Eq_S_product_rule_decaying} \\
\text{ii)}\;\; \mathcal D(gf)(T)&=\mathcal Dg(T)f(T)+g(\overline{T})\mathcal Df(T) \notag \\
&=\mathcal Dg(T)f(\overline{T})+g(T)\mathcal Df(T), \label{Eq_Q_product_rule_decaying} \\
\text{iii)}\;\;\overline{\mathcal D}(gf)(T)&=\overline{\mathcal D}g(T)f(T)+g(T)\overline{\mathcal D}f(T)+\mathcal Dg(T)\big(f(T)-f(\overline{T})\big), \label{Eq_P_product_rule_decaying} \\
\text{iv)}\;\;\Delta(gf)(T)&=\Delta g(T)f(T)+g(T)\Delta f(T)-\mathcal Dg(T) \mathcal Df(T). \label{Eq_F_product_rule_decaying}
\end{align}
\end{subequations}
\end{thm}

\begin{proof}
It is obvious that for $g\in\Psi^{3\alpha,3\beta}(S_\theta)$ and $f\in\Psi_L^{3\alpha,3\beta}(S_\theta)$ also their product is in $gf\in\Psi_L^{3\alpha,3\beta}(S_\theta)$ and hence all the functional calculi in \eqref{Eq_S_product_rule_decaying} -- \eqref{Eq_F_product_rule_decaying} are well defined. \medskip

In the \textit{first step} we fix $\varphi_2<\varphi_1\in(\omega,\theta)$ and use the functional calculi of Definition \ref{defi_Functional_calculus_decaying} and the resolvent identities of Lemma \ref{lem_Resolvent_identities}, to write all the right hand sides of \eqref{Eq_S_product_rule_decaying} -- \eqref{Eq_F_product_rule_decaying} in a similar form. For the right hand side of \eqref{Eq_S_product_rule_decaying} we use \eqref{Eq_S_resolvent_identity} and get
\begin{align*}
g&(T)f(T)=\frac{1}{4\pi^2}\int_{\partial(S_{\varphi_1}\cap\mathbb{C}_J)}\int_{\partial(S_{\varphi_2}\cap\mathbb{C}_I)}g(s)ds_JS_R^{-1}(s,T)S_L^{-1}(p,T)dp_If(p) \\
&=\frac{1}{4\pi^2}\int_{\partial(S_{\varphi_1}\cap\mathbb{C}_J)}\int_{\partial(S_{\varphi_2}\cap\mathbb{C}_I)}g(s)ds_J\big(S_R^{-1}(s,T)p-S_L^{-1}(p,T)p-\overline{s}S_R^{-1}(s,T)+\overline{s}S_L^{-1}(p,T)\big) \\
&\hspace{10.0cm}\times(p^2-2s_0p+|s|^2)^{-1}dp_If(p).
\end{align*}
Regarding the $Q$-, the $P_2$- and the $F$-resolvent identities \eqref{Eq_Q_resolvent_identity}, \eqref{Eq_P_resolvent_identity} and \eqref{Eq_F_resolvent_identity}, we rewrite the right hand sides of \eqref{Eq_Q_product_rule_decaying} -- \eqref{Eq_F_product_rule_decaying} in a similar way. Considering all the additional terms we end up with all the right hand sides (RHS) of \eqref{Eq_S_product_rule_decaying} -- \eqref{Eq_F_product_rule_decaying} written in the form
\begin{align}
\text{RHS}&=\frac{1}{4\pi^2}\int_{\partial(S_{\varphi_1}\cap\mathbb{C}_J)}\int_{\partial(S_{\varphi_2}\cap\mathbb{C}_I)}g(s)ds_J\big(K_R(s,T)p-K_L(p,T)p-\overline{s}K_R(s,T)+\overline{s}K_L(p,T)\big) \notag \\
&\hspace{8cm}\times(p^2-2s_0p+|s|^2)^{-1}dp_If(p), \label{Eq_Product_rule_decaying_5}
\end{align}
using the operators $K_L(p,T)$ and $K_R(s,T)$ from Lemma \ref{lem_K_decomposition}. \medskip

In the \textit{second step} we will further simplify \eqref{Eq_Product_rule_decaying_5}. Since $\varphi_2<\varphi_1$, for every $s\in\partial(S_{\varphi_1}\cap\mathbb{C}_J)$, all the singularities $[s]\cap\mathbb{C}_I$ of $(p^2-2s_0p+|s|^2)^{-1}$ in the plane $\mathbb{C}_I$, lie outside the integration path $\partial(S_{\varphi_2}\cap\mathbb{C}_I)$, if we close the path on the right at infinity. Since the integrals along this closing path vanish due to the asymptotics
\begin{equation*}
\big(K_R(s,T)p-\overline{s}K_R(s,T)\big)(p^2-2s_0p+|s|^2)^{-1}g(p)=\mathcal{O}\big(|p|^{-2+3\beta-\delta}\big),\qquad\text{as }|p|\rightarrow\infty,
\end{equation*}
and the assumption $\beta\leq\frac{1}{3}$. Hence the Cauchy integral formula gives
\begin{equation*}
\int_{\partial(S_{\varphi_2}\cap\mathbb{C}_I)}\big(K_R(s,T)p-\overline{s}K_R(s,T)\big)(p^2-2s_0p+|s|^2)^{-1}dp_If(p)=0,
\end{equation*}
which reduces \eqref{Eq_Product_rule_decaying_5} to
\begin{equation*}
\text{RHS}=\frac{1}{4\pi^2}\int_{\partial(S_{\varphi_1}\cap\mathbb{C}_J)}\int_{\partial(S_{\varphi_2}\cap\mathbb{C}_I)}g(s)ds_J\big(\overline{s}K_L(p,T)-K_L(p,T)p\big)(p^2-2s_0p+|s|^2)^{-1}dp_If(p).
\end{equation*}
Next, the integral identity \eqref{Eq_Integral_identity} further reduces this integral to
\begin{equation*}
\text{RHS}=\frac{1}{2\pi}\int_{\partial(S_{\varphi_2}\cap\mathbb{C}_I)}K_L(p,T)f(p)dp_If(p)=\frac{1}{2\pi}\int_{\partial(S_{\varphi_2}\cap\mathbb{C}_I)}K_L(p,T)p_Ig(p)f(p),
\end{equation*}
where in the second equation we were allowed to interchange $g(p)dp_I=dp_Ig(p)$ since $g$ is intrinsic. However, by the possible choices of $K_L(p,T)$ in Lemma \ref{lem_K_decomposition}, it turns out that the right hand side is exactly the left hand side in the product rules \eqref{Eq_S_product_rule_decaying} -- \eqref{Eq_F_product_rule_decaying}.
\end{proof}

Next we want to derive a connection between the $S$-, the $Q$- and the $P_2$-functional calculus, which is motivated by the fact that $$\mathcal D+\overline{\mathcal D}=2\frac{\partial}{\partial x_0}.$$

\begin{prop}
Let $\alpha\geq\frac{1}{3}$, $\beta\in(0,\frac{1}{3}]$, $\omega\in(0,\pi)$ and $T$ of type $(\alpha,\beta,\omega)$. Then for any $\theta\in(\omega,\pi)$ and $f\in\Psi_L^{3\alpha,3\beta}(S_\theta)$ with $f'\in\Psi_L^{3\alpha,3\beta}(S_\theta)$, there holds
\begin{equation*}
\overline{ \mathcal D}f(T)=2f'(T)-\mathcal Df(T),
\end{equation*}
where $f'(T)$ is understood as the $S$-functional calculus in Definition \ref{defi_Functional_calculus_decaying} i) for the function $f'$.
\end{prop}

\begin{proof}
First, we use the path \eqref{Eq_gamma} to parametrize the integral of the $S$-functional calculus as
\begin{equation*}
f'(T)=\frac{1}{2\pi}\int_{\mathbb{R}\setminus\{0\}}S_L^{-1}\big(\gamma(t),T\big)\frac{\gamma'(t)}{J}f'(\gamma(t))dt=\frac{1}{2\pi}\int_{\mathbb{R}\setminus\{0\}}S_L^{-1}\big(\gamma(t),T\big)\frac{1}{J}\frac{d}{dt}f(\gamma(t))dt.
\end{equation*}
Integration by parts then gives
\begin{equation}\label{Eq_Combination_of_S_Q_P2}
f'(T)=-\frac{1}{2\pi}\int_{\mathbb{R}\setminus\{0\}}\frac{d}{dt}S_L^{-1}(\gamma(t),T)\frac{1}{J}f(\gamma(t))dt=-\frac{1}{2\pi}\int_{\partial(S_\varphi\cap\mathbb{C}_J)}\partial_sS_L^{-1}(s,T)ds_Jf(s),
\end{equation}
where there are no boundary terms since the integrand vanishes for $t=0$ and $|t|\rightarrow\infty$. Explicitly, calculating the derivative, we get
\begin{align*}
\partial_sS_L^{-1}(s,T)&=\frac{\partial}{\partial s}\big((s-\overline{T})Q_{c,s}^{-1}(T)\big)=Q_{c,s}^{-1}(T)-2(s-\overline{T})(s-T_0)Q_{c,s}^{-2}(T) \\
&=Q_{c,s}^{-1}(T)-S_L^{-1}(s,T)\big(S_L^{-1}(s,T)+S_L^{-1}(s,\overline{T})\big)=Q_{c,s}^{-1}(T)-\frac{1}{2}P_2^L(s,T).
\end{align*}
and plugging it into \eqref{Eq_Combination_of_S_Q_P2} we obtain
\begin{equation*}
2f'(T)=\frac{1}{2\pi}\int_{\partial(S_\varphi\cap\mathbb{C}_J)}\big(-2Q_{c,s}^{-1}(T)+P_2^L(s,T)\big)ds_Jf(s)=\mathcal Df(T)+\overline{\mathcal D}f(T). \qedhere
\end{equation*}
\end{proof}

We conclude this section investigating how the functional calculi of Definition \ref{defi_Functional_calculus_decaying} act on functions of the form $s^nf(s)$. In particular, we give a recurrence relation between the functional calculi associated with $s^nf(s)$ and with  $s^{n-1}f(s)$. An implication of the following Proposition \ref{prop_Recurrence_powers}, contained in Proposition \ref{prop_Functional_calculus_of_powers}, demonstrates that the $H^\infty$-functional calculus of powers yields the operators defined using the quaternionic derivatives.

\begin{prop}\label{prop_Recurrence_powers}
Let $\alpha\geq\frac{1}{3}$, $\beta\in(0,\frac{1}{3}]$, $\omega\in(0,\pi)$ and $T$ of type $(\alpha,\beta,\omega)$. Consider now $\theta\in(\omega,\pi)$, $N\in\mathbb{N}$, $f\in\Psi_L^{3\alpha,3\beta-N}(S_\theta)$. Then for every $n\in\{1,\dots,N\}$ there holds \medskip

\begin{enumerate}
\item[i)] $(s^nf)(T)=T^nf(T)$, \\
\item[ii)] $\mathcal D(s^nf)(T)=T\mathcal D(s^{n-1}f)(T)-2\overline{T}^{n-1}f(\overline{T})=\overline{T}\mathcal D(s^{n-1}f)(T)-2T^{n-1}f(T)$, \\
\item[iii)] $\overline{\mathcal D}(s^nf)(T)=T\overline{\mathcal D}(s^{n-1}f)(T)+2\overline{T}^{n-1}f(\overline{T})+2T^{n-1}f(T)$, \\
\item[iv)] $\Delta(s^nf)(T)=T\Delta(s^{n-1}f)(T)+2\mathcal D(s^{n-1}f)(T)$.
\end{enumerate}
\end{prop}

\begin{proof}
First, we observe that  $s^nf\in\Psi_L^{3\alpha+n,3\beta-N+n}(S_\theta)\subseteq\Psi_L^{3\alpha,3\beta}(S_\theta)$ for every $n\in\{0,\dots,N\}$ and hence all the functional calculi in i) -- iv) are well defined. \medskip

i)\;\; We reason by induction and the step $n=0$ is trivial. For the induction step $n-1$ to $n$ we use the identity
\begin{equation}\label{Eq_Recurrence_powers_1}
S_L^{-1}(s,T)s=TS_L^{-1}(s,T)+1,
\end{equation}
to get
\begin{align*}
(s^nf)(T)&=\frac{1}{2\pi}\int_{\partial(S_\varphi\cap\mathbb{C}_J)}S_L^{-1}(s,T)ds_Js^nf(s)=\frac{1}{2\pi}\int_{\partial(S_\varphi\cap\mathbb{C}_J)}\big(TS_L^{-1}(s,T)+1\big)ds_Js^{n-1}f(s) \\
&=\frac{1}{2\pi}T\int_{\partial(S_\varphi\cap\mathbb{C}_J)}S_L^{-1}(s,T)ds_Js^{n-1}f(s)=T(s^{n-1}f)(T)=T^nf(T),
\end{align*}
where in the third term above we used Hills theorem to carry the closed operator $T$ outside the integral and then we observe that the integral over $s^{n-1}f(s)$ vanishes due to the holomorphicity of the function on $S_\varphi$. \medskip

ii)\;\;For the $Q$-functional calculus we use the identity
\begin{equation*}
Q_{c,s}^{-1}(T)s=TQ_{c,s}^{-1}(T)+S_L^{-1}(s,\overline{T}),
\end{equation*}
which immediately follows from the definition \eqref{Eq_S_resolvent} of the left $S$-resolvent, to get the recurrence relation
\begin{align*}
\mathcal D(s^nf)(T)&=\frac{-1}{\pi}\int_{\partial(S_\varphi\cap\mathbb{C}_J)}Q_{c,s}^{-1}(T)ds_Js^nf(s) \\
&=\frac{-1}{\pi}\int_{\partial(S_\varphi\cap\mathbb{C}_J)}\big(TQ_{c,s}^{-1}(T)+S_L^{-1}(s,\overline{T})\big)ds_Js^{n-1}f(s) \\
&=T\mathcal D(s^{n-1}f)(T)-2(s^{n-1}f)(\overline{T})=T\mathcal D(s^{n-1}f)(T)-2\overline{T}^{n-1}f(\overline{T}),
\end{align*}
where in the last line we used i). Analogously, from the identity
\begin{equation*}
Q_{c,s}^{-1}(T)s=\overline{T}Q_{c,s}^{-1}(T)+S_L^{-1}(s,T),
\end{equation*}
there get the second recurrence relation
\begin{equation*}
\mathcal D(s^nf)(T)=\overline{T}\mathcal D(s^{n-1}f)(T)+T^{n-1}f(T).
\end{equation*}
iii)\;\;For the $P_2$-functional calculus we use \eqref{Eq_Recurrence_powers_1} and the fact that $s$ commutes with the sum
\begin{equation*}
S_L^{-1}(s,T)+S_L^{-1}(s,\overline{T})=2(s-T_0)Q_{c,s}^{-1}(T).
\end{equation*}
This leads to
\begin{align*}
P_2^L(s,T)s&=2S_L^{-1}(s,T)\big(S_L^{-1}(s,T)+S_L^{-1}(s,\overline{T})\big)s \\
&=2\big(TS_L^{-1}(s,T)+1\big)\big(S_L^{-1}(s,T)+S_L^{-1}(s,\overline{T})\big) \\
&=TP_2^L(s,T)+2S_L^{-1}(s,T)+2S_L^{-1}(s,\overline{T}),
\end{align*}
and consequently we obtain the recurrence relation
\begin{align*}
\overline{\mathcal D}(s^nf)(T)&=\frac{1}{2\pi}\int_{\partial(S_\varphi\cap\mathbb{C}_J)}P_2^L(s,T)ds_Js^nf(s) \\
&=\frac{1}{2\pi}\int_{\partial(S_\varphi\cap\mathbb{C}_J)}\big(TP_2^L(s,T)+2S_L^{-1}(s,\overline{T})+2S_L^{-1}(s,T)\big)ds_Js^{n-1}f(s) \\
&=T\overline{\mathcal D}(s^{n-1}f)(T)+2(s^{n-1}f)(\overline{T})+2(s^{n-1}f)(T) \\
&=T\overline{\mathcal D}(s^{n-1}f)(T)+2\overline{T}^{n-1}f(\overline{T})+2T^{n-1}f(T).
\end{align*}
iv)\;\;For the $F$-functional calculus we use \eqref{Eq_Recurrence_powers_1}, to get
\begin{equation*}
F_L(s,T)s=-4S_L^{-1}(s,T)Q_{c,s}^{-1}(T)s=-4\big(TS_L^{-1}(s,T)+1\big)Q_{c,s}^{-1}(T)=TF_L(s,T)-4Q_{c,s}^{-1}(T).
\end{equation*}
This then gives the final relation
\begin{align*}
\Delta(s^nf)(T)&=\frac{1}{2\pi}\int_{\partial(S_\varphi\cap\mathbb{C}_J)}F_L(s,T)ds_Js^nf(s) \\
&=\frac{1}{2\pi}\int_{\partial(S_\varphi\cap\mathbb{C}_J)}\big(TF_L(s,T)-4Q_{c,s}^{-1}(T)\big)ds_Js^{n-1}f(s) \\
&=T\Delta(s^{n-1}f)(T)+2\mathcal D(s^{n-1}f)(T). \qedhere
\end{align*}
\end{proof}

\section{The quaternionic $H^\infty$-functional calculi}\label{HINFTY}

In this section we will extend the quaternionic functional calculi of Definition \ref{defi_Functional_calculus_decaying} for decaying functions, to slice hyperholomorphic functions on a sector which are polynomially growing at $0$ and at $\infty$. More precisely, we consider the following spaces of functions.

\begin{defi}\label{defi_Space_of_increasing_functions}
For every $\theta\in(0,\pi)$ we define the function spaces

\begin{enumerate}
\item[i)] $\mathcal{F}_L(S_\theta):=\Set{f\in\mathcal{SH}_L(S_\theta) | \exists k>0,\,C_k\geq 0: |f(s)|\leq C_k\big(|s|^k+\frac{1}{|s|^k}\big),\,s\in S_\theta}$
\item[ii)] $\mathcal{F}(S_\theta):=\Set{f\in\mathcal{N}(S_\theta) | \exists k>0,\,C_k\geq 0: |f(s)|\leq C_k\big(|s|^k+\frac{1}{|s|^k}\big),\,s\in S_\theta}$
\end{enumerate}
\end{defi}

The main idea behind the $H^\infty$-functional calculi is to choose a regularizer function $e$ which imposes enough decay at $0$ and at $\infty$, such that $ef$ is regular enough in order to apply the functional calculus of Definition \ref{defi_Functional_calculus_decaying}. Motivated by the product rules in Theorem \ref{thm_Product_rule_decaying}, we define the so called \textit{quaternionic $H^\infty$-functional calculus}, where it is crucial for $e(T)$ and $e(\overline{T})$ to be injective, and hence in contrast to the bounded functional calculus of Section \ref{sec_Functional_calculi_for_decaying_functions} we are only able to treat injective operators $T$ and $\overline{T}$ by using the choice \eqref{Eq_Regularizer} for the regularizing function.

\begin{defi}\label{defi_Functional_calculus_growing}
Let $\alpha\geq\frac{1}{3}$, $\beta\in(0,\frac{1}{3}]$, $\omega\in(0,\pi)$ and $T$ of type $(\alpha,\beta,\omega)$ with $T,\overline{T}$ injective. Then for every $\theta\in(\omega,\pi)$, $f\in\mathcal{F}_L(S_\theta)$ we define the \textit{$H^\infty$-functional calculi}
\begin{align*}
\text{i)}\;\;\;\;f(T):=&e(T)^{-1}(ef)(T), \hspace{7cm} (S\textit{-functional calculus}) \\
\text{ii)}\;\mathcal Df(T):=&\big(e(T)e(\overline{T})\big)^{-1}\Big(e(T)\mathcal D(ef)(T)-\mathcal De(T)(ef)(T)\Big), \hspace{1.45cm} (Q\textit{-functional calculus}) \\
\text{iii)}\;\overline{\mathcal D}f(T):=&\big(e(T)^2e(\overline{T})\big)^{-1}\Big(e(T)e(\overline{T})\overline{\mathcal D}(ef)(T)-e(\overline{T})\overline{\mathcal D}e(T)(ef)(T) \hspace{0.05cm} (P_2\textit{-functional calculus}) \\
&\hspace{2.5cm}+e(T)\mathcal De(T)(ef)(\overline{T})-e(\overline{T})\mathcal De(T)(ef)(T)\Big),
\end{align*}
\begin{align*}
\text{iv)}\;\Delta f(T):=&\big(e(T)^2e(\overline{T})\big)^{-1}\Big(e(T)e(\overline{T})\Delta(ef)(T)-e(\overline{T})\Delta e(T)(ef)(T) \hspace{0.25cm} (F\textit{-functional calculus}) \\
&\hspace{2.4cm}+e(T) \mathcal De(T) \mathcal D(ef)(T)-(\mathcal De(T))^2(ef)(T)\Big),
\end{align*}
where $e\in\Psi^{3\alpha,3\beta}(S_\theta)$ is such that $e(T),e(\overline{T})$ are injective and $ef\in\Psi_L^{3\alpha,3\beta}(S_\theta)$. Here $e(T)$, $e(\overline{T})$, $(ef)(T)$, $(ef)(\overline{T})$ are understood as the $S$-functional calculus, $\mathcal De(T)$, $\mathcal D(ef)(T)$ as the $Q$-functional calculus, $\overline{\mathcal D}e(T)$, $\overline{\mathcal D}(ef)(T)$ as the $P_2$-functional calculus and $\Delta e(T)$, $\Delta(ef)(T)$ as the $F$-functional calculus of Definition \ref{defi_Functional_calculus_decaying}.
\end{defi}

\begin{thm}
Let $\alpha\geq\frac{1}{3}$, $\beta\in(0,\frac{1}{3}]$, $\omega\in(0,\pi)$ and $T$ of type $(\alpha,\beta,\omega)$ with $T,\overline{T}$ injective. Then for every $\theta\in(\omega,\pi)$, $f\in\mathcal{F}_L(S_\theta)$ there exists a regularizer function $e$ in the sense of Definition~\ref{defi_Functional_calculus_growing} and no one of the functional calculi i) -- iv) in Definition \ref{defi_Functional_calculus_growing} depend on the choice of the regularizer $e$.
\end{thm}

\begin{proof}
Let $f\in\mathcal{F}_L(S_\theta)$, i.e. there exists $k>0$, $C_k\geq 0$, such that
$$
|f(s)|\leq C_k(|s|^k+\frac{1}{|s|^k}).
$$
 Choosing now $n\in\mathbb{N}$ with $n>\max\{k+3\alpha-1,k-3\beta+1\}$ and consider the function
\begin{equation}\label{Eq_Regularizer}
e(s):=\frac{s^n}{(1+s)^{2n}},
\end{equation}
it is obvious that $e\in\Psi^{3\alpha,3\beta}(S_\theta)$ as well as $ef\in\Psi^{3\alpha,3\beta}(S_\theta)$. From \cite[Equation (58)]{MPS23} and Proposition \ref{prop_Recurrence_powers} i), it then follows that $(1+T)^{2n}$ is bijective and $e(T)=T^n(1+T)^{-2n}$. Since $T$ is injective, $e(T)$ is then injective as well. Analogously, the injectivity of $e(\overline{T})=\overline{T}^n(1+\overline{T})^{2n}$ follows from the assumed injectivity of $\overline{T}$. \medskip

For the independence of the regularizer, let $e_1,e_2\in\Psi^{3\alpha,3\beta}(S_\theta)$, with $e_1(T),e_1(\overline{T}),e_2(T),e_2(\overline{T})$ are injective and $e_1f,e_2f\in\Psi_L^{3\alpha,3\beta}(S_\theta)$. Note, that the symbols $e_1$ and $e_2$ for the regularizers are the same as the one for the imaginary units of the quaternions. This fact doesn't pose an issue, as they do not appear in the proofs of the following theorems. \medskip

To enhance the clarity of the calculations in this proof, we will adopt the following notation
\begin{equation}\label{Eq_Notation}
\begin{array}{llll} e\text{ for }e(T), & (ef)\text{ for }(ef)(T), & e^{-1}\text{ for }e(T)^{-1}, & \mathcal De\text{ for }\mathcal De(T), \\ \overline{e}\text{ for }e(\overline{T}), & (\overline{ef})\text{ for }(ef)(\overline{T}), & (\overline{e})^{-1}\text{ for }e(\overline{T})^{-1}, & \mathcal D(ef)\text{ for }\mathcal D(ef)(T), \\ \overline{\mathcal D}e\text{ for }\overline{\mathcal D}e(T), & \overline{\mathcal D}(ef)\text{ for }\overline{\mathcal D}(ef)(T), & \Delta e\text{ for }\Delta e(T), & \Delta(ef)\text{ for }\Delta(ef)(T). \end{array}
\end{equation}
Throughout this proof we will always use the fact that (which is due to Corollary \ref{cor_Commutation_decaying} ii)):
\begin{equation*}
e_i,\overline{e_i}\quad\text{commute with}\quad e_j, \mathcal De_j,\overline{\mathcal D}e_j,\Delta e_j,\qquad\text{for every }i,j\in\{1,2\}.
\end{equation*}
i)\;\;For the $S$-functional calculus, it follows from the product rule \eqref{Eq_S_product_rule_decaying}, that
\begin{equation}\label{Eq_Independence_regularizer_2}
e_2(e_1f)=(e_2e_1f)=(e_1e_2f)=e_1(e_2f).
\end{equation}
Multiplying $(e_1e_2)^{-1}$ from the left gives the independence of the regularizer
\begin{equation*}
e_1^{-1}(e_1f)=e_2^{-1}(e_2f).
\end{equation*}
ii)\;\;For the $Q$-functional calculus, the two versions of the product rule \eqref{Eq_Q_product_rule_decaying} state, that
\begin{equation*}
\mathcal De_2(\overline{e_1f})+e_2\mathcal D(e_1f)=\mathcal D(e_2e_1f)=\mathcal D(e_1e_2f)=\mathcal De_1(e_2f)+\overline{e_1}\mathcal D(e_2f).
\end{equation*}
Rearranging this equation, gives
\begin{equation}\label{Eq_Independence_regularizer_3}
e_2\mathcal D(e_1f)-\mathcal De_1(e_2f)=\overline{e_1}\mathcal D(e_2f)-\mathcal De_2(\overline{e_1f}).
\end{equation}
Multiplying $e_1\overline{e_2}$ from the left and using \eqref{Eq_Independence_regularizer_2}, gives
\begin{align}
e_2\overline{e_2}\big(e_1\mathcal D(e_1f)-\mathcal De_1(e_1f)\big)&=e_1\overline{e_1}\big(\overline{e_2}\mathcal D(e_2f)-\mathcal De_2(\overline{e_2f})\big) \notag \\
&=e_1\overline{e_1}\big(e_2\mathcal D(e_2f)-\mathcal De_2(e_2f)\big), \label{Eq_Independence_regularizer_4}
\end{align}
where the second equation comes from \eqref{Eq_Independence_regularizer_3} with the choice $e_1=e_2$. Multiplying $(e_1e_2\overline{e_1e_2})^{-1}$ from the left, then gives the independence of the regularizer
\begin{equation*}
(e_1\overline{e_1})^{-1}\big(e_1\mathcal D(e_1f)-\mathcal De_1(e_1f)\big)=(e_2\overline{e_2})^{-1}\big(e_2\mathcal D(e_2f)-\mathcal De_2(e_2f)\big).
\end{equation*}
iii)\;\;For the $P_2$-functional calculus we apply \eqref{Eq_P_product_rule_decaying} to both sides of \eqref{Eq_Independence_regularizer_2} and get
\begin{equation*}
\overline{\mathcal D}e_2(e_1f)+e_2\overline{\mathcal D}(e_1f)+\mathcal De_2\big((e_1f)-(\overline{e_1f})\big)=\overline{\mathcal D}e_1(e_2f)+e_1\overline{\mathcal D}(e_2f)+\mathcal De_1\big((e_2f)-(\overline{e_2f})\big).
\end{equation*}
Rearranging the terms leads to
\begin{equation*}
e_2\overline{\mathcal D}(e_1f)+\mathcal De_1(\overline{e_2f})-(\mathcal De_1+\overline{\mathcal D}e_1)(e_2f)=e_1\overline{\mathcal D}(e_2f)+\mathcal De_2(\overline{e_1f})-(\mathcal De_2+\overline{\mathcal D}e_2)(e_1f).
\end{equation*}
Multiplying $e_1e_2\overline{e_1e_2}$ from the left and using again \eqref{Eq_Independence_regularizer_2}, gives
\begin{align}
e_2\overline{e_2}\big(&e_1e_2\overline{e_1}\overline{\mathcal D}(e_1f)+e_1\overline{e_2}\mathcal De_1(\overline{e_1f})-e_2\overline{e_1}(\mathcal De_1+\overline{\mathcal D}e_1)(e_1f)\big) \notag \\
&=e_1\overline{e_1}\big(e_1e_2\overline{e_2}\overline{\mathcal D}(e_2f)+e_2\overline{e_1}\mathcal De_2(\overline{e_2f})-e_1\overline{e_2}(\mathcal De_2+\overline{\mathcal D}e_2)(e_2f)\big). \label{Eq_Independence_regularizer_1}
\end{align}
Furthermore, similar to \eqref{Eq_Independence_regularizer_3}, the two versions of the product rule \eqref{Eq_Q_product_rule_decaying} give the identity
\begin{equation*}
\overline{e_2}\mathcal De_1-\overline{e_1}\mathcal De_2=e_2\mathcal De_1-e_1\mathcal De_2,
\end{equation*}
which, multiplied by $e_1e_2(\overline{e_1e_2f})$ from the right, becomes
\begin{equation*}
e_2\overline{e_2}e_1\overline{e_2}\mathcal De_1(\overline{e_1f})-e_1\overline{e_1}e_2\overline{e_1}\mathcal De_2(\overline{e_2f})=e_2^2\overline{e_2}e_1\mathcal De_1(\overline{e_1f})-e_1^2\overline{e_1}e_2\mathcal De_2(\overline{e_2f}).
\end{equation*}
Using this in \eqref{Eq_Independence_regularizer_1} gives
\begin{align*}
e_2^2\overline{e_2}\big(&e_1\overline{e_1}\overline{\mathcal D}(e_1f)+e_1\mathcal De_1(\overline{e_1f})-\overline{e_1}(\mathcal De_1+\overline{\mathcal D}e_1)(e_1f)\big) \\
&=e_1^2\overline{e_1}\big(e_2\overline{e_2}\overline{\mathcal D}(e_2f)+e_2\mathcal De_2(\overline{e_2f})-\overline{e_2}(\mathcal De_2+\overline{\mathcal D}e_2)(e_2f)\big).
\end{align*}
Multiplying now $(e_1^2e_2^2\overline{e_1e_2})^{-1}$ from the left, gives the independence of the regularizer
\begin{align*}
(e_1^2\overline{e_1})^{-1}\big(&e_1\overline{e_1}\overline{\mathcal D}(e_1f)+e_1\mathcal De_1(\overline{e_1f})-\overline{e_1}(\mathcal De_1+\overline{\mathcal D}e_1)(e_1f)\big) \\
&=(e_2^2\overline{e_2})^{-1}\big(e_2\overline{e_2}\overline{\mathcal D}(e_2f)+e_2\mathcal De_2(\overline{e_2f})-\overline{e_2}(\mathcal De_2+\overline{\mathcal D}e_2)(e_2f)\big).
\end{align*}
iv)\;\;For the $F$-functional calculus, we apply \eqref{Eq_F_product_rule_decaying} to both sides of \eqref{Eq_Independence_regularizer_2} and get
\begin{equation*}
\Delta e_2(e_1f)+e_2\Delta(e_1f)-\mathcal De_2\mathcal D(e_1f)=\Delta e_1(e_2f)+e_1\Delta(e_2f)-\mathcal De_1 \mathcal D(e_2f).
\end{equation*}
Rearranging the terms and multiplying the resulting equation with $e_1e_2\overline{e_1e_2}$, gives
\begin{equation*}
e_1e_2\overline{e_1e_2}\big(e_2\Delta(e_1f)-\Delta e_1(e_2f)+\mathcal De_1 \mathcal D(e_2f)\big)=e_1e_2\overline{e_1e_2}\big(e_1\Delta(e_2f)-\Delta e_2(e_1f)+\mathcal De_2\mathcal D(e_1f)\big).
\end{equation*}
Using the identity \eqref{Eq_Independence_regularizer_3} on the left hand side, and the same one with $e_1\leftrightarrow e_2$ exchanged on the right hand side, turns this equation into
\begin{align*}
e_1e_2\overline{e_2}\Big(&\overline{e_1}e_2\Delta(e_1f)-\overline{e_1}\Delta e_1(e_2f)+\mathcal De_1\big(\mathcal De_2(\overline{e_1f})+e_2\mathcal D(e_1f)-\mathcal De_1(e_2f)\big)\Big) \\
&=e_1e_2\overline{e_1}\Big(\overline{e_2}e_1\Delta(e_2f)-\overline{e_2}\Delta e_2(e_1f)+\mathcal De_2\big(\mathcal De_1(\overline{e_2f})+e_1\mathcal D(e_2f)-\mathcal De_2(e_1f)\big)\Big).
\end{align*}
Since the term $e_1e_2\overline{e_2} \mathcal De_1 \mathcal De_2(\overline{e_1f})$ cancels with $e_1e_2\overline{e_1}\mathcal De_2\mathcal De_1(\overline{e_2f})$ on the right due, this equation reduces to
\begin{align*}
e_2^2\overline{e_2}\big(&e_1\overline{e_1}\Delta(e_1f)-\overline{e_1}\Delta e_1(e_1f)+e_1\mathcal De_1\mathcal D(e_1f)-(\mathcal De_1)^2(e_1f)\big) \\
&=e_1^2\overline{e_1}\big(e_2\overline{e_2}\Delta(e_2f)-\overline{e_2}\Delta e_2(e_2f)+e_2\mathcal De_2\mathcal D(e_2f)-(\mathcal De_2)^2(e_2f)\big),
\end{align*}
where we once more used \eqref{Eq_Independence_regularizer_2}. Multiplying both sides with $(e_1^2e_2^2\overline{e_1e_2})^{-1}$ gives the independence of the regularizer
\begin{align*}
(e_1^2\overline{e_1})^{-1}\big(&e_1\overline{e_1}\Delta(e_1f)-\overline{e_1}\Delta e_1(e_1f)+e_1\mathcal De_1\mathcal D(e_1f)-(\mathcal De_1)^2(e_1f)\big) \\
&=(e_2^2\overline{e_2})^{-1}\big(e_2\overline{e_2}\Delta(e_2f)-\overline{e_2}\Delta e_2(e_2f)+e_2\mathcal De_2\mathcal D(e_2f)-(\mathcal De_2)^2(e_2f)\big). \qedhere
\end{align*}
\end{proof}

\begin{lem}\label{lem_Properties_growing}
Let $\alpha\geq\frac{1}{3}$, $\beta\in(0,\frac{1}{3}]$, $\omega\in(0,\pi)$ and $T$ of type $(\alpha,\beta,\omega)$ with $T,\overline{T}$ injective. Then for every $\theta\in(0,\pi)$, $f\in\mathcal{F}_L(S_\theta)$ the functional calculi
\begin{equation*}
f(T),\,\mathcal Df(T),\,\overline{\mathcal D}f(T),\,\Delta f(T)\quad\text{are closed operators}.
\end{equation*}
\end{lem}

\begin{proof}
Since in the Definition \ref{defi_Functional_calculus_growing} i) -- iv) the operators $e(T),e(\overline{T})$ are bounded and injective, the inverses $e(T)^{-1},\big(e(T)e(\overline{T})\big)^{-1}$ and $\big(e(T)^2e(\overline{T})\big)^{-1}$ are closed operators. Moreover, the remaining term $(ef)(T)$ in i) and the large brackets in ii) -- iv), are bounded operators. Altogether, $f(T),\mathcal Df(T),\overline{\mathcal D}f(T),\Delta f(T)$ are the product of a bounded and a closed operator and hence are closed operators themselves.
\end{proof}

\begin{prop}\label{prop_Commutation_growing}
Let $\alpha\geq\frac{1}{3}$, $\beta\in(0,\frac{1}{3}]$, $\omega\in(0,\pi)$ and $T$ of type $(\alpha,\beta,\omega)$ with $T,\overline{T}$ injective. Moreover, let $B\in\mathcal{B}(V)$ which commutes with $T,T_0,T_1,T_2,T_3$ on $\dom(T)$. Then for every $\theta\in(\omega,\pi)$, $g\in\mathcal{F}(S_\theta)$ and any choice $\widehat{g}\in\{g,\mathcal Dg,\overline{\mathcal D}g,\Delta g\}$, there holds
\begin{equation*}
B\widehat{g}(T)\subseteq\widehat{g}(T)B.
\end{equation*}
\end{prop}

\begin{proof}
First of all, we note that for any injective operator $A\in\mathcal{B}(V)$ which commutes with $B$, we get the inclusion
\begin{equation*}
BA^{-1}=A^{-1}ABA^{-1}=A^{-1}BAA^{-1}\subseteq A^{-1}B.
\end{equation*}
Due to Proposition \ref{prop_Commutation_B} this in particular holds true for all prefactors in Definition \ref{defi_Functional_calculus_growing} i) -- iv), i.e.  $A=e(T)$, $A=e(T)e(\overline{T})$ and $A=e(T)^2e(\overline{T})$. Also by Proposition \ref{prop_Commutation_growing}, $B$ commutes with $(ef)(T)$ from Definition \ref{defi_Functional_calculus_growing} i) and with the big bracket terms from Definition \ref{defi_Functional_calculus_growing} ii) -- iv). Altogether, this shows the commutation $B\widehat{g}(T)\subseteq\widehat{g}(T)B$.
\end{proof}

\begin{thm}\label{thm_Product_rule_growing}
Let $\alpha\geq\frac{1}{3}$, $\beta\in(0,\frac{1}{3}]$, $\omega\in(0,\pi)$ and $T$ of type $(\alpha,\beta,\omega)$ with $T,\overline{T}$ injective. Then for any $\theta\in(0,\pi)$, $g\in\mathcal{F}(S_\theta)$, $f\in\mathcal{F}_L(S_\theta)$, there holds the product rules
\begin{subequations}
\begin{align}
\text{i)}\;\;\;\;(gf)(T)&\supseteq g(T)f(T), \label{Eq_S_product_rule_growing} \\
\text{ii)}\;\mathcal D(gf)(T)&\supseteq \mathcal Dg(T)f(T)+g(\overline{T})\mathcal Df(T)\quad\text{and} \notag \\
\mathcal D(gf)(T)&\supseteq \mathcal Dg(T)f(\overline{T})+g(T)\mathcal Df(T), \label{Eq_Q_product_rule_growing} \\
\text{iii)}\;\overline{\mathcal D}(gf)(T)&\supseteq\overline{\mathcal D}g(T)f(T)+g(T)\overline{\mathcal D}f(T)+\mathcal Dg(T)\big(f(T)-f(\overline{T})\big). \label{Eq_P_product_rule_growing} \\
\text{iv)}\;\Delta(gf)(T)&\supseteq\Delta g(T)f(T)+g(T)\Delta f(T)-\mathcal Dg(T)\mathcal Df(T). \label{Eq_F_product_rule_growing}
\end{align}
\end{subequations}
\end{thm}

\begin{proof}
In this proof we will use again the notation \eqref{Eq_Notation}. Let $e_1$ be a regularizer of $g$ and $e_2$ a regularizer of $f$ according to Definition \ref{defi_Functional_calculus_growing}. Then it is clear that $e_1e_2$ is a regularizer of the product $gf$. Note, that the symbols $e_1$ and $e_2$ for the regularizers are the same as the one for the imaginary units of the quaternions. This fact doesn't pose an issue, as they do not appear in the proofs of the following theorems. \medskip

In the following we will do many manipulations like interchanging the order of operators which will be allowed by Corollary \ref{cor_Commutation_decaying}. We will also use the following list of operator identities (or inclusions) and we will not mention them any more at any point where they appear in the proof:

\begin{enumerate}
\item[$\circ$] $A^{-1}B\supseteq BA^{-1}$\hfill for any $A,B\in\mathcal{B}(V,V)$ with $A$ injective and $AB=BA$,
\item[$\circ$] $(AB)^{-1}=A^{-1}B^{-1}=B^{-1}A^{-1}$\hfill for any injective $A,B\in\mathcal{B}(V)$ with $AB=BA$.
\item[$\circ$] $A^{-1}(B+C)\supseteq A^{-1}B+A^{-1}C$\hfill for any $A,B,C\in\mathcal{B}(V)$ with $A$ injective.
\end{enumerate}

i)\;\;For \eqref{Eq_S_product_rule_growing}, we use the product rule \eqref{Eq_S_product_rule_decaying} for $(e_1ge_2f)$ in the Definition \ref{defi_Functional_calculus_growing} i), to get
\begin{equation*}
(gf)=(e_1e_2)^{-1}(e_1e_2gf)=e_1^{-1}e_2^{-1}(e_1g)(e_2f)\supseteq e_1^{-1}(e_1g)e_2^{-1}(e_2f)=gf.
\end{equation*}
ii)\;\;For the $Q$-functional calculus, we only prove the first equation in \eqref{Eq_Q_product_rule_growing}, while the second one follows the same steps. Using the product rule \eqref{Eq_S_product_rule_decaying} for $(e_1ge_2f)$ as well as \eqref{Eq_Q_product_rule_decaying} for $\mathcal D(e_1ge_2f)$ and $\mathcal D(e_1e_2)$, we can rewrite the Definition \ref{defi_Functional_calculus_growing} ii) as
\begin{align*}
\mathcal D(gf)(T)&=(e_1e_2\overline{e_1e_2})^{-1}\big((e_1e_2)\mathcal D(e_1e_2gf)-(e_1e_2gf)\mathcal D(e_1e_2)\big) \\
&=(e_1e_2\overline{e_1e_2})^{-1}\Big(e_1e_2\big(\mathcal D(e_1g)(e_2f)+(\overline{e_1g})\mathcal D(e_2f)\big)-(e_1g)\big(\overline{e_2}\mathcal De_1+e_1\mathcal De_2\big)(e_2f)\Big).
\end{align*}
The identity \eqref{Eq_Independence_regularizer_3} with $g$ instead of $f$, then turns the above equation into
\begin{align*}
\mathcal D(gf)(T)&=(e_1e_2\overline{e_1e_2})^{-1}\big(e_1\overline{e_2}\mathcal D(e_1g)(e_2f)+e_1e_2(\overline{e_1g})\mathcal D(e_2f) \\
&\hspace{3cm}-(e_1g)\overline{e_2}\mathcal De_1(e_2f)-(\overline{e_1g})e_1\mathcal De_2(e_2f)\big) \\
&\supseteq(e_1e_2\overline{e_1})^{-1}\big(e_1\mathcal D(e_1g)-(e_1g)\mathcal De_1\big)(e_2f)+(e_2\overline{e_1e_2})^{-1}(\overline{e_1g})\big(e_2\mathcal D(e_2f)-\mathcal De_2(e_2f)\big) \\
&\supseteq(e_1\overline{e_1})^{-1}\big(e_1\mathcal D(e_1g)-(e_1g)\mathcal De_1\big)e_2^{-1}(e_2f) \\
&\quad+(\overline{e_1})^{-1}(\overline{e_1g})(e_2\overline{e_2})^{-1}\big(e_2\mathcal D(e_2f)-\mathcal De_2(e_2f)\big)=\mathcal Dg\,f+\overline{g}\,\mathcal Df.
\end{align*}
iii)\;\;For \eqref{Eq_P_product_rule_growing}, we use the product rule \eqref{Eq_S_product_rule_decaying} for $(e_1ge_2f)$ and $(\overline{e_1ge_2f})$, \eqref{Eq_Q_product_rule_decaying} for $D(e_1e_2)$, as well as \eqref{Eq_P_product_rule_decaying} for $\overline{\mathcal D}(e_1ge_2f)$ and $\overline{\mathcal D}(e_1e_2)$, we can rewrite Definition \ref{defi_Functional_calculus_growing} iii) as
\begin{align}
\overline{\mathcal D}(gf)&=(e_1^2e_2^2\overline{e_1e_2})^{-1}\big(e_1e_2\overline{e_1e_2}\overline{\mathcal D}(e_1e_2gf)-\overline{e_1e_2}\overline{\mathcal D}(e_1e_2)(e_1e_2gf) \notag \\
&\hspace{2.7cm}+e_1e_2\mathcal D(e_1e_2)(\overline{e_1e_2gf})-\overline{e_1e_2}\mathcal D(e_1e_2)(e_1e_2gf)\big) \notag \\
&=(e_1^2e_2^2\overline{e_1e_2})^{-1}\Big(e_1e_2\overline{e_1e_2}\big(\overline{\mathcal D}(e_1g)(e_2f)+(e_1g)\overline{\mathcal D}(e_2f)+\mathcal D(e_1g)(e_2f)-\mathcal D(e_1g)(\overline{e_2f})\big) \notag \\
&\hspace{2.7cm}-\overline{e_1e_2}(e_1g)\big(\overline{\mathcal D}e_1e_2+e_1\overline{\mathcal D}e_2+\mathcal De_1e_2-\mathcal De_1\overline{e_2}\big)(e_2f) \notag \\
&\hspace{2.7cm}+e_1e_2(\overline{e_1g})\big(\mathcal De_1e_2+\overline{e_1}\mathcal De_2\big)(\overline{e_2f})-\overline{e_1e_2}(e_1g)\big(\mathcal De_1\overline{e_2}+e_1\mathcal De_2\big)(e_2f)\Big) \notag \\
&=(e_1^2e_2^2\overline{e_1e_2})^{-1}\Big(e_2\overline{e_2}\big(e_1\overline{e_1}\overline{\mathcal D}(e_1g)-\overline{e_1}\overline{\mathcal D}e_1(e_1g)+e_1\mathcal De_1(\overline{e_1g})-\overline{e_1}\mathcal De_1(e_1g)\big)(e_2f) \notag \\
&\hspace{2.7cm}+e_1\overline{e_1}(e_1g)\big(e_2\overline{e_2}\overline{\mathcal D}(e_2f)-\overline{e_2}\overline{\mathcal D}e_2(e_2f)+e_2\mathcal De_2(\overline{e_2f})-\overline{e_2}\mathcal De_2(e_2f)\big) \notag \\
&\hspace{2.7cm}+e_1e_2\big(\overline{e_1}\mathcal D(e_1g)-\mathcal De_1(\overline{e_1g})\big)\big(\overline{e_2}(e_2f)-e_2(\overline{e_2f})\big) \notag \\
&\hspace{2.7cm}+e_1e_2\overline{e_1}\big(\mathcal D(e_1g)e_2+(\overline{e_1g})\mathcal De_2-\mathcal D(e_1g)\overline{e_2}-\mathcal De_2(e_1g)\big)(\overline{e_2f}). \label{Eq_Product_rule_growing_1}
\end{align}
Carrying now parts of the inverse $(e_1^2e_2^2\overline{e_1e_2})^{-1}$ inside the bracket gives the inclustion
\begin{align*}
\overline{\mathcal D}(gf)&\supseteq(e_1^2\overline{e_1})^{-1}\big(e_1\overline{e_1}\overline{\mathcal D}(e_1g)-\overline{e_1}\overline{\mathcal D}e_1(e_1g)+e_1\mathcal De_1(\overline{e_1g})-\overline{e_1}\mathcal De_1(e_1g)\big)e_2^{-1}(e_2f) \\
&\quad+e_1^{-1}(e_1g)(e_2^2\overline{e_2})^{-1}\big(e_2\overline{e_2}\overline{\mathcal D}(e_2f)-\overline{e_2}\overline{\mathcal D}e_2(e_2f)+e_2\mathcal De_2(\overline{e_2f})-\overline{e_2}\mathcal De_2(e_2f)\big) \\
&\quad+(e_1\overline{e_1})^{-1}\big(\overline{e_1}\mathcal D(e_1g)-\mathcal De_1(\overline{e_1g})\big)\big(e_2^{-1}(e_2f)-(\overline{e_2})^{-1}(\overline{e_2f})\big) \\
&=\overline{\mathcal D}g\,f+g\overline{\mathcal D}f+\mathcal Dg(f-\overline{f}),
\end{align*}
where in the third equation we simply rearranged the terms, and the line \eqref{Eq_Product_rule_growing_1} vanishes due to \eqref{Eq_Independence_regularizer_3} with $g$ instead of $f$. \medskip

iv)\;\;For \eqref{Eq_F_product_rule_growing}, we use the product rules \eqref{Eq_S_product_rule_decaying} for $(e_1ge_2f)$, \eqref{Eq_Q_product_rule_decaying} for $D(e_1ge_2f)$ and $D(e_1e_2)$ as well as \eqref{Eq_F_product_rule_decaying} for $\Delta(e_1ge_2f)$ and $\Delta(e_1e_2)$, to rewrite Definition \ref{defi_Functional_calculus_growing} iv) as
\begin{align*}
\Delta(gf)&=(e_1^2e_2^2\overline{e_1e_2})^{-1}\big(e_1e_2\overline{e_1e_2}\Delta(e_1e_2gf)-\overline{e_1e_2}\Delta(e_1e_2)(e_1e_2gf) \\
&\hspace{2.7cm}+e_1e_2D(e_1e_2)\mathcal D(e_1e_2gf)-(\mathcal D(e_1e_2))^2(e_1e_2gf)\big) \\
&=(e_1^2e_2^2\overline{e_1e_2})^{-1}\Big(e_1e_2\overline{e_1e_2}\big(\Delta(e_1g)(e_2f)+(e_1g)\Delta(e_2f)-\mathcal D(e_1g)\mathcal D(e_2f)\big) \\
&\hspace{2.7cm}-\overline{e_1e_2}(e_1g)\big(\Delta e_1e_2+e_1\Delta e_2-\mathcal De_1\mathcal De_2\big)(e_2f) \\
&\hspace{3cm}+e_1e_2\mathcal D(e_1e_2)\mathcal D(e_1g)(\overline{e_2f})+e_1e_2(\mathcal De_1e_2+\overline{e_1}\mathcal De_2)(e_1g)\mathcal D(e_2f) \\
&\hspace{2.7cm}-(e_1g)\big(\mathcal De_1\overline{e_2}+e_1\mathcal De_2\big)\big(\mathcal De_1e_2+\overline{e_1}\mathcal De_2\big)(e_2f)\Big).
\end{align*}
Rearranging the terms and using in \eqref{Eq_Product_rule_growing_1} the identity
\begin{align*}
\mathcal D(e_1e_2)(\overline{e_2f})-\overline{e_1e_2}\mathcal D(e_2f)&=\mathcal D(e_1e_2)(e_2f)-e_1e_2\mathcal D(e_2f) \\
&=\overline{e_2}\mathcal De_1(e_2f)+e_1\mathcal De_2(e_2f)-e_1e_2\mathcal D(e_2f),
\end{align*}
which is a result of the two versions of the product rule \eqref{Eq_Q_product_rule_decaying}, in the fourth line of the upcoming equation, turns \eqref{Eq_Product_rule_growing_1} into.
\begin{align*}
\Delta(gf)&=(e_1^2e_2^2\overline{e_1e_2})^{-1}\Big(e_2\overline{e_2}\big(e_1\overline{e_1}\Delta(e_1g)-\overline{e_1}\Delta e_1(e_1g)-(\mathcal De_1)^2(e_1g)\big)(e_2f) \\
&\hspace{2.7cm}+e_1\overline{e_1}(e_1g)\big(e_2\overline{e_2}\Delta(e_2f)-\overline{e_2}\Delta e_2(e_2f)+e_2\mathcal De_2\mathcal D(e_2f)-(\mathcal De_2)^2(e_2f)\big) \\
&\hspace{2.7cm}-e_1e_2\mathcal De_1(e_1g)\big(\mathcal De_2(e_2f)-e_2\mathcal D(e_2f)\big) \\
&\hspace{2.7cm}+e_1e_2\mathcal D(e_1g)\big(\mathcal D(e_1e_2)(\overline{e_2f})-\overline{e_1e_2}\mathcal D(e_2f)\big) \\
&=(e_1^2e_2^2\overline{e_1e_2})^{-1}\Big(e_2\overline{e_2}\big(e_1\overline{e_1}\Delta(e_1g)-\overline{e_1}\Delta e_1(e_1g)+e_1\mathcal De_1\mathcal D(e_1g)-(\mathcal De_1)^2(e_1g)\big)(e_2f) \\
&\hspace{2.7cm}+e_1\overline{e_1}(e_1g)\big(e_2\overline{e_2}\Delta(e_2f)-\overline{e_2}\Delta e_2(e_2f)+e_2\mathcal De_2 \mathcal D(e_2f)-(\mathcal De_2)^2(e_2f)\big) \\
&\hspace{2.7cm}-e_1e_2\big(e_1\mathcal D(e_1g)-\mathcal De_1(e_1g)\big)\big(e_2\mathcal D(e_2f)-\mathcal De_2(e_2f)\big) \\
&\supseteq(e_1^2\overline{e_1})^{-1}\big(e_1\overline{e_1}\Delta(e_1g)-\overline{e_1}\Delta e_1(e_1g)+e_1\mathcal De_1\mathcal D(e_1g)-(\mathcal De_1)^2(e_1g)\big)e_2^{-1}(e_2f) \\
&\quad+e_1^{-1}(e_1g)(e_2^2\overline{e_2})^{-1}\big(e_2\overline{e_2}\Delta(e_2f)-\overline{e_2}\Delta e_2(e_2f)+e_2\mathcal De_2\mathcal D(e_2f)-(\mathcal De_2)^2(e_2f)\big) \\
&\quad-(e_1\overline{e_1})^{-1}\big(e_1\mathcal D(e_1g)-\mathcal De_1(e_1g)\big)(e_2\overline{e_2})^{-1}\big(e_2\mathcal D(e_2f)-\mathcal De_2(e_2f)\big) \\
&=\Delta g\,f+g\Delta f-\mathcal Dg\mathcal Df. \qedhere
\end{align*}
\end{proof}

As the final result of this paper, we investigate the action of the $H^\infty$-functional calculi act on powers $f(s)=s^n$. Comparing it, to how $\mathcal D$, $\overline{\mathcal D}$ and $\Delta$ from \eqref{Eq_Cauchy_Fueter_operator}, \eqref{Eq_Cauchy_Fueter_operator_conjugate} and \eqref{Eq_Laplace_operator} act as quaternionic derivatives on powers (see also \cite[Lemma 1]{B}) as
\begin{align*}
\mathcal D(q^n)&=-2\sum\limits_{k=0}^{n-1}\overline{q}^{n-1-k}q^k, \\
\overline{\mathcal D}(q^n)&=2nq^{n-1}+2\sum\limits_{k=0}^{n-1}\overline{q}^{n-1-k}q^k, \\
\Delta(q^n)&=-4\sum\limits_{k=1}^{n-1}k\overline{q}^{n-1-k}q^{k-1}.
\end{align*}

\begin{prop}\label{prop_Functional_calculus_of_powers}
Let $\alpha\geq\frac{1}{3}$, $\beta\in(0,\frac{1}{3}]$, $\omega\in(0,\pi)$ and $T$ of type $(\alpha,\beta,\omega)$ with $T,\overline{T}$ injective. Then for every $n\in\mathbb{N}$ there holds
\begin{align*}
&\text{i)}\;\;(s^n)(T)=T^n, && \text{iii)}\;\;\overline{\mathcal D}(s^n)(T)\supseteq 2nT^{n-1}+2\sum\limits_{k=0}^{n-1}\overline{T}^{n-1-k}T^k \\
&\text{ii)}\;\;\mathcal D(s^n)(T)\supseteq-2\sum\limits_{k=0}^{n-1}\overline{T}^{n-1-k}T^k, && \text{iv)}\;\;\Delta(s^n)(T)\supseteq-4\sum\limits_{k=1}^{n-1}k\overline{T}^{n-1-k}T^{k-1}.
\end{align*}
\end{prop}

\begin{proof}
Throughout this proof we will use the commutation properties from Corollary \ref{cor_Commutation_decaying} iv) and relation in Proposition \ref{prop_Recurrence_powers} i), namely
\begin{equation*}
f(T)T\subseteq Tf(T)=(sf)(T)\qquad\text{and}\qquad\widehat{f}(T)T\subseteq T\widehat{f}(T).
\end{equation*}
In order to make the calculations more readable, we will again use the notation \eqref{Eq_Notation}. \medskip

i)\;\;Let us choose the special regularizer function $e(s)=\frac{s^N}{(1+s)^{2n}}$ from \eqref{Eq_Regularizer} with $N$ large enough. Then with the Definition \ref{defi_Functional_calculus_growing} i) of the $S$-functional calculus, we get
\begin{align*}
(s^n)(T)&=e(T)^{-1}\Big(\frac{s^{n+N}}{(1+s)^{2N}}\Big)(T)=(1+T)^{2N}T^{-N}T^{n+N}(1+T)^{-2N} \\
&=(1+T)^{2N}T^n(1+T)^{-2N}=T^n(1+T)^{2N}(1+T)^{-2N}=T^n.
\end{align*}
ii)\;\;For the induction start $n=0$, the Definition \ref{defi_Functional_calculus_growing} ii) with $f(s)=1$ gives
\begin{equation*}
\mathcal D(1)=(e\overline{e})^{-1}(e\mathcal De-\mathcal De\,e)=0,
\end{equation*}
where we used that $\mathcal De\,e=e\mathcal De$ commutes due to Corollary \ref{cor_Commutation_decaying}. For the induction step $n-1\rightarrow n$ we know by the induction assumption and Definition \ref{defi_Functional_calculus_growing} ii), that
\begin{equation}\label{Eq_Rational_equivalence_growing_6}
-2\sum\limits_{k=0}^{n-2}\overline{T}^{n-2-k}T^k\subseteq(e\overline{e})^{-1}\big(e\mathcal D(es^{n-1})-\mathcal De(es^{n-1})\big).
\end{equation}
Multiplying this equation with $T$ from the left, subtracting the term $2\overline{T}^{n-1}$ and using the recurrence relation from Proposition \ref{prop_Recurrence_powers} ii), gives
\begin{align*}
-2\sum\limits_{k=0}^{n-1}\overline{T}^{n-1-k}T^k&\subseteq(e\overline{e})^{-1}\big(e\mathcal D(es^{n-1})-\mathcal De(es^{n-1})\big)T-2\overline{T}^{n-1} \\
&\subseteq(e\overline{e})^{-1}\big(eT\mathcal D(es^{n-1})-\mathcal De(es^n)-2e\overline{T}^{n-1}\overline{e}\big) \\
&\subseteq(e\overline{e})^{-1}\big(e\mathcal D(es^n)-\mathcal De(es^n)\big)=\mathcal D(s^n).
\end{align*}
iii)\;\;For the induction start $n=0$, it follows from the commutation of $\mathcal De$ and $\overline{\mathcal D}e$ with $e$ and $\overline{e}$, see Corollary \ref{cor_Commutation_decaying} ii), that Definition \ref{defi_Functional_calculus_growing} iii) with $f(s)=1$ turns into
\begin{equation*}
\overline{\mathcal D}(1)=(e^2\overline{e})^{-1}\big(e\overline{e}\overline{\mathcal D}e-\overline{e}\overline{\mathcal D}e\,e+e\mathcal De\,\overline{e}-\overline{e}\mathcal De\,e\big)=0.
\end{equation*}
For the induction step $n-1\rightarrow n$ we know by the induction assumption, that
\begin{align*}
2(n-1)T^{n-2}+2\sum\limits_{k=0}^{n-2}\overline{T}^{n-2-k}T^k\subseteq(e^2\overline{e})^{-1}\Big(&e\overline{e}\overline{\mathcal D}(es^{n-1})-\overline{e}\overline{\mathcal D}e(es^{n-1}) \\
&+e\mathcal De(\overline{es^{n-1}})-\overline{e}\mathcal De(es^{n-1})\Big).
\end{align*}
Multiplying $T$ from the right, adding $2T^{n-1}+2\overline{T}^{n-1}$ and using Proposition \ref{prop_Recurrence_powers} iii), gives
\begin{align*}
&2nT^{n-1}+2\sum\limits_{k=0}^{n-1}\overline{T}^{n-1-k}T^k \\
&\subseteq(e^2\overline{e})^{-1}\Big(e\overline{e}\overline{\mathcal D}(es^{n-1})-\overline{e}\overline{\mathcal D}e(es^{n-1})+e\mathcal De(\overline{es^{n-1}})-\overline{e}\mathcal De(es^{n-1})\Big)T+2T^{n-1}+2\overline{T}^{n-1} \\
&\subseteq(e^2\overline{e})^{-1}\Big(e\overline{e}T\overline{\mathcal D}(es^{n-1})-\overline{e}\overline{\mathcal D}e(es^n)+eT\mathcal De(\overline{es^{n-1}})-\overline{e}\mathcal De(es^n)+2e(es^{n-1})\overline{e}+2e^2(\overline{es^{n-1}})\Big) \\
&=(e^2\overline{e})^{-1}\Big(e\overline{e}\overline{\mathcal D}(es^n)-\overline{e}\overline{\mathcal D}e(es^n)-\overline{e}\mathcal De(es^n)+e\big(T\mathcal De+2e-2\overline{e}\big)(\overline{es^{n-1}})\Big).
\end{align*}
Using now the identity $T\mathcal De+2e=\overline{T}\mathcal De+2\overline{e}$, which follows from the two equivalent versions of the recurrence relation in Proposition \ref{prop_Recurrence_powers} ii) with $n=1$, reduces the last term to
\begin{equation*}
2nT^{n-1}+2\sum\limits_{k=0}^{n-1}\overline{T}^{n-1-k}T^k\subseteq(e^2\overline{e})^{-1}\Big(e\overline{e}\overline{\mathcal D}(es^n)-\overline{e}\overline{\mathcal D}e(es^n)-\overline{e}\mathcal De(es^n)+e\mathcal De(\overline{es^n})\Big)=\overline{\mathcal D}(s^n).
\end{equation*}
iv)\;\;For the induction start $n=0$ we get from the commutation in Corollary \ref{cor_Commutation_decaying}, that
\begin{equation*}
\Delta(1)=(e^2\overline{e})^{-1}\big(e\overline{e}\Delta e-\overline{e}\Delta e\,e+e\mathcal De\mathcal De-(\mathcal De)^2e\big)=0.
\end{equation*}
For $n=1$, we use $(es)=Te$, $\mathcal D(es)=T\mathcal De-2\overline{e}$ and $\Delta(es)=T\Delta e+2\mathcal De$ from Proposition \ref{prop_Recurrence_powers} i), ii) and iv) with $n=1$, to get
\begin{align*}
\Delta(s)&=(e^2\overline{e})^{-1}\big(e\overline{e}\Delta(es)-\overline{e}\Delta e(es)+e\mathcal De\mathcal D(es)-(\mathcal De)^2(es)\big) \\
&=(e^2\overline{e})^{-1}\big(e\overline{e}(T\Delta e+2\mathcal De)-\overline{e}\Delta eTe+e\mathcal De(T\mathcal De-2\overline{e})-(\mathcal De)^2Te\big)=0.
\end{align*}
For the induction step $n-1\rightarrow n$ we know by the induction assumption, that
\begin{equation*}
-4\sum\limits_{k=1}^{n-2}k\overline{T}^{n-2-k}T^{k-1}\subseteq(e^2\overline{e})^{-1}\big(e\overline{e}\Delta(es^{n-1})-\overline{e}\Delta e(es^{n-1})+e\mathcal De\mathcal D(es^{n-1})-(\mathcal De)^2(es^{n-1})\big).
\end{equation*}
Multiplying $T$ from the right and subtracting $4\sum_{k=1}^{n-1}\overline{T}^{n-1-k}T^{k-1}$, gives
\begin{align}
&-4\sum\limits_{k=1}^{n-1}k\overline{T}^{n-1-k}T^{k-1} \notag \\
&\subseteq(e^2\overline{e})^{-1}\Big(e\overline{e}\Delta(es^{n-1})-\overline{e}\Delta e(es^{n-1})+e\mathcal De\mathcal D(es^{n-1})-(\mathcal De)^2(es^{n-1})\Big)T-4\sum\limits_{k=1}^{n-1}\overline{T}^{n-1-k}T^{k-1} \notag \\
&\subseteq(e^2\overline{e})^{-1}\Big(e\overline{e}T\Delta(es^{n-1})-\overline{e}\Delta e(es^n)+e\mathcal DeT\mathcal D(es^{n-1})-(\mathcal De)^2(es^n)-4e^2\overline{e}\sum\limits_{k=1}^{n-1}\overline{T}^{n-1-k}T^{k-1}\Big) \notag \\
&=(e^2\overline{e})^{-1}\Big(e\overline{e}\Delta(es^n)-\overline{e}\Delta e(es^n)+e\mathcal De\mathcal D(es^n)-(\mathcal De)^2(es^n) \notag \\
&\hspace{2cm}-2e\overline{e}\mathcal D(es^{n-1})+2e\mathcal De(\overline{es^{n-1}})-4e^2\overline{e}\sum\limits_{k=1}^{n-1}\overline{T}^{n-1-k}T^{k-1}\Big), \label{Eq_Functional_calculus_of_powers_1}
\end{align}
where in the last equation we used the recurrence relation in Proposition \ref{prop_Recurrence_powers} ii) \& iv). Multiplying \eqref{Eq_Rational_equivalence_growing_6} with $e\overline{e}$ and using the two versions of the product rule \eqref{Eq_Q_product_rule_decaying}, gives
\begin{equation*}
-2e\overline{e}\sum\limits_{k=0}^{n-2}\overline{T}^{n-2-k}T^k\subseteq e\mathcal D(es^{n-1})-\mathcal De(es^{n-1})=\overline{e}\mathcal D(es^{n-1})-\mathcal De(\overline{es^{n-1}}).
\end{equation*}
This now turns \eqref{Eq_Functional_calculus_of_powers_1} into the stated
\begin{equation*}
-4\sum\limits_{k=1}^{n-1}k\overline{T}^{n-1-k}T^{k-1}\subseteq(e^2\overline{e})^{-1}\big(e\overline{e}\Delta(es^n)-\overline{e}\Delta e(es^n)+e\mathcal De\mathcal D(es^n)-(\mathcal De)^2(es^n)\big)=\Delta(s^n). \qedhere
\end{equation*}
\end{proof}

\end{document}